\documentclass[a4paper,11pt]{article}
\usepackage[margin=0.7in]{geometry}

\usepackage{siunitx}
\usepackage[english]{babel}
\usepackage{amsfonts}
\usepackage{amssymb}
\usepackage{amsmath}
\usepackage{amsthm}
\usepackage{eucal}
\usepackage{mathrsfs}
\allowdisplaybreaks[0]
\usepackage{bm}
\usepackage{graphicx}
\usepackage{caption}
\usepackage{subcaption}
\usepackage{algorithm,algorithmic}
\usepackage{framed}
\usepackage{diagbox}
\usepackage[titletoc,title]{appendix}
\usepackage{tikz}
\usepackage{color}

\usepackage{hyperref}
\usepackage{cleveref}
\usepackage{autonum}
\usepackage{amsopn}

\newcommand{\Mbold}{\mathbf{M}}
\newcommand{\Wbold}{\mathbf{W}}
\newcommand{\Rbold}{\mathbf{R}}
\newcommand{\Lambdabold}{\bm{\Lambda}}
\newcommand{\lambdabold}{\bm{\lambda}}

\newcommand{\R}{\mathbb{R}}
\DeclareMathOperator*{\argmin}{arg\,min}
\DeclareMathOperator*{\argmax}{arg\,max}
\DeclareMathOperator*{\diag}{diag}

\newtheorem{theorem}{Theorem}[section]
\newtheorem{corollary}{Corollary}[theorem]
\newtheorem{lemma}[theorem]{Lemma}
\newtheorem{proposition}[theorem]{Proposition}
\newtheorem{definition}[theorem]{Definition}
\newtheorem{remark}[theorem]{Remark}

\usepackage{lipsum}
\usepackage{epstopdf}
\ifpdf
  \DeclareGraphicsExtensions{.eps,.pdf,.png,.jpg}
\else
  \DeclareGraphicsExtensions{.eps}
\fi

\usepackage{enumitem}
\setlist[enumerate]{leftmargin=.5in}
\setlist[itemize]{leftmargin=.5in}

\graphicspath{{./images/}{./reconstruction}{./test1}{./test2}}

\title{A  flexible space-variant anisotropic regularisation for image restoration with automated parameter selection\thanks{LC acknowledges the support of the Fondation Math\'ematiques Jacques Hadamard (FMJH). The research of LC and AL was supported by the Research in Paris (RiP) project 2018 \emph{Space-variant anisotropic regularisation for image restoration}, IHP, Paris. Research of AL, MP and FS was supported by the ``National Group for Scientific Computation
(GNCS-INDAM)'' and by the ex60 project ``Funds for selected research topics''.}}

\date{}

\author{
Luca Calatroni\thanks{CMAP, \'Ecole Polytechnique, Palaiseau, 91128, Route de Saclay, France
(\href{mailto:luca.calatroni@polytechnique.edu}{luca.calatroni@polytechnique.edu}).} \and Alessandro Lanza\thanks{Department of Mathematics, University of Bologna, Piazza di Porta San Donato 5, Bologna, Italy 
  (\href{mailto:alessandro.lanza2@unibo.it}{alessandro.lanza2@unibo.it},  \href{mailto:monica.pragliola2@unibo.it}{monica.pragliola2@unibo.it}, \href{mailto:fiorella.sgallari@unibo.it}{fiorella.sgallari@unibo.it}).} 
  \and Monica Pragliola\footnotemark[3] 
  \and  Fiorella Sgallari\footnotemark[3]
}

\usepackage{amsopn}

\begin{document}

\maketitle

\begin{abstract}
We propose a new space-variant anisotropic regularisation term for variational image restoration, based on the statistical assumption that the gradients of the target image distribute locally according to a bivariate generalised Gaussian distribution. The highly flexible variational structure of the corresponding regulariser encodes several free parameters which hold the potential for faithfully modelling the local geometry in the image and describing local orientation preferences.
For an automatic estimation of such parameters, we design a robust maximum likelihood approach and report results on its reliability on synthetic data and natural images. 
For the numerical solution of the corresponding image restoration model, we use an iterative algorithm based on the Alternating Direction Method of Multipliers (ADMM). 
A suitable preliminary variable splitting together with a novel result in multivariate non-convex proximal calculus yield a very efficient minimisation algorithm. Several numerical results showing significant quality-improvement of the proposed model with respect to some related state-of-the-art competitors are reported, in particular in terms of texture and detail preservation.
\end{abstract}


\vspace{0.3cm}

\textbf{Note}: Accepted for publication in SIAM Journal of Imaging Sciences. Please cite as appropriate.


\section{Introduction}
\label{sec:intro}
Image restoration is the task of recovering a clean and sharp image from a noisy, and potentially blurred, observation. In mathematical terms,
let $\Omega$ be a rectangular image domain of size $d_1\times d_2$ and let $n:=d_1 d_2$ be the total number of image pixels in $\Omega$. For a given blurred and noisy image $g\in\mathbb{R}^n$, the typical image restoration inverse problem can be written as
\begin{equation} \label{eq:GDM}
\text{find }u\in \mathbb{R}^{n}\qquad\text{such that}\qquad g= \mathcal{T}\left( K u \right) ,
\end{equation}
where $K \in \mathbb{R}^{n \times n}$ is a known linear blurring operator, while $\mathcal{T}(\cdot)$ denotes the operator modelling the presence of noise in $g$ in a non-deterministic and very likely non-linear way.

Due to the ill-posedness of the problem \eqref{eq:GDM}, it is in general impossible to find $u$ from \eqref{eq:GDM} due to the lack of stability and/or uniqueness properties. Therefore, in practice, the task can be reformulated as the problem of finding an estimate $u^*$ of the desired $u$ as accurate as possible via a well-posed problem. In particular, variational regularisation methods compute the restored image  $u^*\in\mathbb{R}^n$ as a minimiser of a cost functional $\mathcal{J}: \mathbb{R}^n \to \mathbb{R}^+$ such that the problem can be formulated as 
\begin{equation}
\text{find }u^* \in \argmin_{u \in \mathbb{R}^n}
~\Big\{ \mathcal{J}(u):= R(u) + \mu F(Ku;g) \Big\}.
\label{eq:GVM}
\end{equation}
The functionals $R$ and $F$ are commonly referred to as the \emph{regularisation} and the \emph{data fidelity} term, respectively. While $R$ encodes prior information on the desired image $u$ (such as its regularity and its sparsity patterns), the $F$ is a data term which measures the `distance' between the given image $g$ and $u$ after the action of the operator $K$ with respect to some norm corresponding to the noise statistics in the data, cf., e.g., \cite{stuartInversebook}. The regularisation parameter $\mu > 0$ controls the trade-off between the two terms.

A very popular choice for $R$ is the Total Variation (TV) semi-norm \cite{ROF,chambollelions1997,vese2001}, which is defined in the discrete setting as
\begin{equation}
R(u) \:\;{=}\;\: \mathrm{TV}(u) \,\;{:=}\; \sum_{i=1}^{n}  \| (\nabla u)_{i} \|_2 \, ,
\label{eq:TV}
\end{equation}
where for each $i=1,\ldots,n,$ by $(\nabla u)_i \in \mathbb{R}^2$ we denote the discrete gradient of image $u$ at pixel $i$.
The choice of TV-type regularisations for image restoration problems became very popular in the last three decades due mainly to its convexity and, most importantly, its edge-preservation capability.

As mentioned above, the choice of $F$ depends on the noise distribution in the data. In this paper, we are particularly interested in the case of additive (zero-mean) white Gaussian noise (AWGN),
i.e. we consider the following form for degradation model in \eqref{eq:GDM}:
\begin{equation}   \label{eq:SDM}
g=\mathcal{T}\left( K u \right) = Ku+b,
\end{equation}
%
%
%
%
where the additive corruption $b \in \mathbb{R}^n$ stands for a vector of independent realisations drawn from the same univariate Gaussian distribution with zero mean and variance $\sigma^2$.
Note that such  noise distribution is indeed fully described by the unique scalar parameter $\sigma>0$.
Other similar noise models appearing in applications are the Additive White Laplacian Noise (AWLN) and the impulsive Salt and Pepper Noise (SPN), which can also be fully described by a unique scalar parameter, being it either the standard deviation or the probability of a pixel of being corrupted, respectively.

Finally, the regularisation parameter $\mu$ in \eqref{eq:GVM} plays a crucial role in the reconstruction results since its size balances the smoothing provided by the regularisation and the trust in the data. Very often, $\mu$ is chosen empirically by brute-force optimisation with respect to some fixed image quality measure (such as the SNR or the SSIM). However, for AWGN data, when the noise level $\sigma$ is known effective techniques based on discrepancy principles or L-curve can be used \cite{engl2000regularization}. More recently, similar approaches with `adaptive' discrepancy principles have been proposed for possibly combined AWGN and SPN noise models in \cite{Langer2017}, while learning approaches based on the use of training sets and not requiring any prior knowledge of the noise level have been studied for optimal parameter selection in \cite{bilevellearning}.

\smallskip

It is well known that a statistically-consistent data fidelity term modelling the presence of AWGN in the data is the squared L$_2$ norm of the residual image,
%
%
which, combined with the TV regulariser \eqref{eq:TV} results in the popular TV-L$_2$ - or Rudin Osher Fatemi (ROF) \cite{ROF} - image restoration model:
\begin{equation}
\text{find } u^*\quad\text{such that}\quad u^*\in\argmin_{u \in \mathbb{R}^n}
\left\{ \,
\mathrm{TV}(u) \,\;{+}\;\,\frac{\mu}{2} \| K u - g \|_2^2
\, \right\}.
\label{eq:TVLq}
\end{equation}
Due to presence of the TV regulariser, model \eqref{eq:TVLq} is non-smooth, a fundamental feature which guarantees the desirable property of edge preservation. Furthermore, its convexity makes it appealing for several efficient optimisation methods - see \cite{chambolle2016introduction} for a review - and it is often used as a reference model for the study of either higher-order regularisations (e.g. the Total Generalised Variation \cite{TGV}) or of non-Gaussian \cite{benningFidelities,Nikolova2004,duvaltvL1,Sciacchitano2015} and possibly combined \cite{Calatroni2017,lanza2013} noise distributions. 

However, in addition to the well-known reconstruction drawbacks such as the \emph{staircasing effect}, the TV regulariser in \eqref{eq:TV} suffers from additional limitations. First of all, it is \emph{global} or \emph{space-invariant}, i.e. its local regularisation contribution at each pixel takes exactly the same form and, as a result, it cannot adapt its functional shape to local image structures. Furthermore, it is not adapted to situations where clear local directional texture may appear, as it happens for instance in fiber and seismic imaging applications. For the mentioned problems the use of some dominant \cite{KonDonKnu17,BayramDTV2012} or local \cite{Zhang2013} anisotropy information can strongly improve the quality of the reconstruction.

The intrinsic limits of the TV regulariser have been discussed in great detail in \cite{tvpl2} from a statistical point of view. There, the authors point out how the use of TV regularisation implicitly corresponds to consider a space-invariant one-parameter half-Laplacian Distribution (hLD) for the gradient magnitudes of $u$, which is in general too restrictive to model the actual distribution of gradient magnitudes in real images. To overcome this issue, in \cite{tvpl2} the authors propose the more general half-Generalised Gaussian Distribution (hGGD) as a prior which results in the following TV$_p$ regularisation model
\begin{equation}\label{eq:TVp}
\mathrm{TV}_p(u) := \sum_{i=1}^{n}  \| (\nabla u)_{i} \|_2^p \, ,
\quad p \:{\in}\: (0,2] \, .
\end{equation}
The exponent $p$ appearing in \eqref{eq:TVp} is a free parameter which provides the TV$_p$ regulariser with higher flexibility than the TV regulariser. The parameter $p$, however, is fixed over the whole image domain and, hence, does not allow to capture locality in the image. 

In \cite{vip,CMBBE} the authors consider a space-variant extension of the TV$_p$ regulariser in \eqref{eq:TVp}  which can better adapt to local image smoothness upon suitable parameter estimation. The new TV$_{\alpha,p}^{\mathrm{sv}}$ regulariser is there defined by
\begin{equation}\label{eq:TVpsv}
\mathrm{TV}_{\alpha,p}^{\mathrm{sv}}(u) := \sum_{i=1}^{n} \alpha_i \| (\nabla u)_{i} \|_2^{p_i} \, ,
\quad p_i \:{\in}\: (0,2], \quad \alpha_i > 0 
\;\;\, \forall \, i=1,\ldots,n,
\end{equation}
and shown to be effective on several image restoration problems.

%

\subsection{Contribution} In this paper we propose a \emph{space-variant} and {directional} image regulariser denoted by DTV$_{p}^{\mathrm{sv}}$ to extend even further the TV$_{\alpha,p}^{\mathrm{sv}}$ regularization model \eqref{eq:TVpsv} as:
\begin{equation}
\mathrm{DTV}_{p}^{\mathrm{sv}}(u) \,\;{:=}\; \sum_{i=1}^{n}  \left\| \Lambda_i R_{\theta_i} \, (\nabla u)_{i} \right\|_2^{p_i} ,
\quad \;\: p_i >0 \;\;\, \forall \, i=1,\ldots,n.
\label{eq:PMa}
\end{equation}
For every $i=1,2,\ldots,n$, the weighting and rotation matrices $\Lambda_i, R_{\theta_i} \in \R^{2 \times 2}$ are defined respectively by:
\begin{equation} \label{eq:PMb2}
\Lambda_i := \begin{pmatrix}
\lambda^{(1)}_i & 0 \\ 0 & \lambda^{(2)}_i
\end{pmatrix}, \quad
\lambda_i^{(1)}\geq \lambda_i^{(2)}>0, \qquad
R_{\theta_i} := \begin{pmatrix}
\cos\theta_i & -\sin\theta_i\\
\sin\theta_i & \cos\theta_i
\end{pmatrix}, \quad 
\theta_i \:{\in}\; [0,2\pi),
\end{equation}
so that $\theta_i$ has to be understood as the local image orientation, while the parameters $\lambda^{(1)}_i$ and $\lambda^{(2)}_i$ weight at any point the TV-like smoothing along the direction $\theta_i$ and its orthogonal, respectively. 

Under this definitions, we can then define our space-variant, anisotropic (or directional) and possibly non-convex DTV$_{p}^{\mathrm{sv}}$-L$_2$ variational model for image restoration:
\begin{equation} 
\label{eq:PMb}
\text{find }\;u^* \in \; 
\argmin_{u \in \mathbb{R}^n}
~\Big\{ \, \mathcal{J}(u):=
\mathrm{DTV}_{p}^{\mathrm{sv}}(u) +  \, \frac{\mu}{2}\|Ku-g\|^2_2 \, \Big\}, \quad \mu>0.
\end{equation}
Note here that the non-convexity arises whenever $0<p_i<1$ for at least one $i=1,\ldots,n$. 

%
The proposed DTV$_{p}^{\mathrm{sv}}$ regulariser \eqref{eq:PMa}-\eqref{eq:PMb2} is highly flexible as it potentially adapts to local smoothness and directional properties of the image at hand, provided that a reliable estimation of the parameters $\lambda^{(1)}_i, \lambda^{(2)}_i, \theta_i$ and $p_i$ is given. In fact, in comparison to the previous work by the authors in \cite{vip,CMBBE}, our proposal extends the TV$_{\alpha,p}^{\mathrm{sv}}$ regularisation model \eqref{eq:TVpsv} so as to accommodate further local directional information, which can significantly improve the restoration results in the case, for instance, of textured and/or high-detailed images.

The statistical rationale of our approach relies on a prior assumption on the distribution of the gradient magnitudes of the desired image $u$ which we assume to be space-variant and locally drawn from a Bivariate Generalised Gaussian Distribution (BGGD) \cite{MGGD,shape1,shape2}. 

\smallskip

The main contribution of this work is twofold: on one side, we propose the highly-flexible DTV$_{p}^{\mathrm{sv}}$ regulariser in \eqref{eq:PMa}-\eqref{eq:PMb2} and justify its variational form via MAP estimation.  On the other hand, to guarantee its actual applicability on image restoration problems, we propose an automated efficient method for the robust estimation of the model parameters from the observed image $g$ by means of a Maximum Likelihood (ML) estimation approach. The effectiveness of such estimation is confirmed numerically on several synthetic and natural examples, showing a good agreement with local geometrical structures in the images considered in terms of their `local' shape. From a numerical point of view, we solve the optimisation problem \eqref{eq:PMb} by means of an efficient iterative minimisation algorithm based on the ADMM \cite{BOYD_ADMM} and apply it to several test images under different degradation levels, comparing the results with other relevant competing approaches.
Finally, in order to get a fully-automated image restoration approach, the regularisation parameter $\mu$ in our model \eqref{eq:PMb} is automatically adjusted along the ADMM iterations as described in \cite{APE}, such that the computed solution $u^*$ satisfies the discrepancy principle \cite{WC12}, i.e. it belongs to the discrepancy set 
\begin{equation}
\mathcal{D} \,\;{:=}\;\, 
\left\{ \,
u \in \R^{n} : 
\| K u - g \|_2 \leq \delta:= \tau \sigma \sqrt{n} \, \right\} .
\label{discr_set}
\end{equation} 
In \eqref{discr_set}, the discrepancy threshold value $\delta$ depends on the a priori known or estimated noise level $\sigma$, the number of pixels $n$ and the discrepancy parameter $\tau$, which is typically chosen to be slightly greater than one, in order to avoid under-estimation of the noise.

\subsection{Organisation of the paper} 

Firstly, in Section \ref{sec:cont} we draw some analogies between the proposed DTV$_{p}^{\mathrm{sv}}$ discrete regulariser \eqref{eq:PMa}-\eqref{eq:PMb2} and some related previous studies on its infinite-dimensional correspondent. Then, in Section \ref{sec:map} we show that the DTV$_{p}^{\mathrm{sv}}$ regularisation model can be derived via standard MAP estimation by assuming that the image gradients are drawn locally from a space-variant BGGD. In Section \ref{sec:parameter_estimation} we describe in detail the ML procedure used 
for automatically estimating the local parameters appearing in the DTV$_{p}^{\mathrm{sv}}$ regulariser from the observed corrupted image $g$. The existence of global minimisers for the total DTV$_{p}^{\mathrm{sv}}$-L$_2$ objective functional in \eqref{eq:PMb} is then proved in Section \ref{sec:exist} via standard arguments. Next, in Section \ref{sec:admm} we describe in detail the ADMM algorithm used to compute such minimisers and present a novel useful result in multivariate non-convex proximal calculus. As far as our numerical tests are concerned, we report in Section \ref{sec:parameter_estimation_results} the results of the ML approach described above for a robust estimation of the parameter maps. In Section \ref{sec:reconstruction} we report the results obtained by the  DTV$_{p}^{\mathrm{sv}}$-L$_2$ image restoration model applied to some image deblurring/denoising problems observing its good performance in terms, mainly, of texture and detail preservation. Finally, we conclude our work with some outlook for future research directions in Section \ref{sec:conc}.

\section{Formulation in function spaces}\label{sec:cont}

The formulation of the DTV$_{p}^{\mathrm{sv}}$ regulariser \eqref{eq:PMa}-\eqref{eq:PMb2} in an infinite-dimensional function spaces defined over a regular $\Omega\subset\mathbb{R}^2$ reads as:
\begin{equation}   \label{eq:energy_DTV_cont}
\mathcal{DTV}_{p(\cdot)}^{\mathrm{sv}}(u) := \int_\Omega \left| \Mbold_{\lambdabold,\bm{\theta}}(x) \nabla u(x)\right|^{p(x)}~dx,
\end{equation}
where $p:\Omega\to (0,\infty)$ stands for the variable exponent and the tensor $\Mbold_{\lambdabold,\bm{\theta}}$ is defined for any $x\in\Omega$ in terms of analogous weighting and rotation operators as in \eqref{eq:FF4} by:
\begin{equation}  \label{eq:metric}
\Mbold_{\lambdabold,\theta}(x):=\Lambdabold_{\lambdabold}(x) \Rbold_{\theta}^T(x),
\end{equation}
where $\lambdabold = (\lambda_1,\lambda_2)\in L^\infty(\Omega;\R_{+}^2)$ and $\theta:\Omega\to [0,2\pi)$. Note that whenever $\Mbold_{\lambdabold,\theta}=\bm{\mathrm{I}}$, there is no directionality encoded in the problem. In the following, we will refer to this special case as \emph{isotropic} model.

Several well-known image regularisation models can be cast in a functional form similar to \eqref{eq:energy_DTV_cont} or in their corresponding PDE counterparts.


\subsection{Constant exponent $p\geq 1$} In the convex and constant case $p(x) = p \in [1,\infty)$ for every $x\in\Omega$, \eqref{eq:energy_DTV_cont} can be thought as an anisotropic image regulariser where images are chosen as elements in the Sobolev space $W^{1,p}(\Omega)$ or, more generally, modelled as Radon measures in some subspace of $\mathrm{BV}(\Omega)$, the space of functions of bounded variation \cite{AmbrosioBV}.

In the special case $p=2$ the functional \eqref{eq:energy_DTV_cont} can be re-written  for $u\in H^1(\Omega)$ as
\begin{equation}   \label{eq:energy_DTV_contp=2}
\mathcal{DTV}_{2}^{\mathrm{sv}}(u)=\int_\Omega \left\|\nabla u(x)\right\|^2_{\Wbold_{\lambdabold,\theta}}~dx
\end{equation}
where $\|e\|_{\Wbold_{\lambdabold,\theta}}:=\sqrt{\langle e,\, \Wbold_{\lambdabold,\theta} e~\rangle}$ is a scalar product and $\Wbold_{\lambdabold,\theta} := \Mbold_{\lambdabold,\theta}^T\Mbold_{\lambdabold,\theta}$ is a symmetric and positive semi-definite anisotropic tensor. By taking the $L^2$-gradient flow of the energy in \eqref{eq:energy_DTV_contp=2}, we can easily draw connections between this choice and the standard anisotropic diffusion PDE models proposed by Weickert in \cite{weickert98,weickert02}. Indeed, by endowing $\Omega$ with Neumann boundary conditions we get that minimising $\mathcal{DTV}_{2}^{\mathrm{sv}}$ corresponds to compute the stationary solution of:
\begin{equation} \label{eq:anis_diff}
\begin{cases}
u_t = \mathrm{div}\Big( \Wbold_{\lambdabold,\theta}\nabla u \Big) &\text{on } \Omega\times (0,\infty], \\
u(x,0)=f(x) & \text{on } \Omega,\\
\langle  \Wbold_{\lambdabold,\theta}\nabla u, \bm{n}\rangle=0  & \text{on }   \partial\Omega\times(0,T],
\end{cases}
\end{equation}
where $\bm{n}$ stands for the outward normal vector on $\partial\Omega$. The Cauchy problem \eqref{eq:anis_diff} is a reference model for anisotropic PDE approaches for image restoration. The tensor $\Wbold_{\lambdabold,\theta}$ stands for space-dependent diffusivity matrix which can introduce non-linearities in the model \cite{Bro02} or classically related to a  structure-tensor modelling as in \cite{weickert98,Roussos2010,Scharr2003}.


Recently, a similar formalism has been employed also in \cite{BayramDTV2012,Zhang2013,KonDonKnuDTGV17,KonDonKnu17} in the case $p(x)=p=1$ for `dominant' fixed principal direction $\theta(x)=\bar{\theta}\in[0,2\pi)$ and adapted in \cite{Tovey2019,Ehrhardt2016} to local directionalities in the context of medical imaging. In such case the functional in \eqref{eq:energy_DTV_cont} reads:
\begin{equation}   \label{eq:energy_DTV_p=1}
\mathcal{DTV}_{1}^{\mathrm{sv}}(u) := \int_\Omega \left| \Mbold_{\lambdabold,\theta}(x) \nabla u(x)\right|~dx,
\end{equation}
which can be seen as a directional version of TV regularisation, and thereafter called DTV regularisation. In \cite{KonDonKnuDTGV17} a higher-order Directional Total Generalized Variation (DTGV) regulariser is also studied to promote smoother reconstructions and a full analysis in function spaces is performed. The choice of considering only the main direction $\bar{\theta}$ in the image restricts the authors to study very simple images with regular stripe patterns. For this purpose, an algorithm estimating such direction and some experiments on its robustness/sensitivity to noise are presented.

\subsection{Variable exponent $p(x)\geq 1$} In the isotropic case, variable exponent models have been considered, e.g., in \cite{BlomgrenTVp} under the modelling assumption $p(x)= p(|\nabla u(x)|)$ with:
\begin{equation}   \label{eq:exponent}
\lim_{s\to 0} ~p(s)=2, \qquad \lim_{s\to+\infty}~p(s)=1.
\end{equation}
Heuristically, such conditions correspond to consider a quadratic smoothing in correspondence of flat areas (small gradients) and a TV-type in correspondence with edges (large gradients). To overcome the difficulties arising from the theoretical analysis of such general modelling, in \cite{Chen2006TVp,Li2010} some easier variable exponent models have been considered. There, the image regulariser takes the following form:
\begin{equation}  \label{eq:var-exp}
\mathcal{R}_{p(\cdot)}(u):= \int_\Omega\frac{1}{p(x)}|\nabla u(x)|^{p(x)}~dx,
\end{equation}
where for every $x\in\Omega$ the exponent function $p:\Omega\to [1,2]$ is defined via the following explicit formula
\begin{equation}   \label{eq:var-exp1}
p(x) = 1 + \frac{1}{1+k|G_\varsigma\ast \nabla g(x)|},\quad \varsigma,~ k>0,
\end{equation}
where $G_\varsigma$ is a convolution kernel of parameter $\varsigma$ and $g$ is the given corrupted image. Under such choice, the conditions \eqref{eq:exponent} are satisfied and all the possible intermediate values are allowed.
In \cite{Chen2006TVp,Li2010} the regulariser \eqref{eq:var-exp}-\eqref{eq:var-exp1} is combined with L$^2$-fidelity and shown to reduce staircasing compared to constant exponent models.


\subsection{Non-convex models with constant exponents $0<p<1$} 
More recently, some non-convex image regularisation models in the form \eqref{eq:energy_DTV_cont} with constant exponent $0<p<1$ have been considered. In \cite{Hintermuller2013,Hintermuller2015}, for instance, non-convex TV$_p$-type regularisers have been shown to be indeed preferable for some applications in comparison to convex ($p\geq 1$) models as the ones described above. The analysis in function spaces covered by the authors is motivated by the use of discrete models such as the ones proposed previously in \cite{NikolovaNonCvx2010,Rodriguez2010}. For this type of regularisation and upon an appropriate Huber-type smoothing, efficient Trust-Region-Based optimisation solvers are designed.

Generally speaking, non-convex regularisers in the form \eqref{eq:energy_DTV_cont} with constant exponent $p<1$ are nowadays well-known to promote stronger sparsity in the data, improving significantly the reconstruction obtained in terms of structure preservation and edge sharpness. On the other hand, such methods may result in an over-complication of the problem in correspondence of homogeneous image regions, where a plain isotropic smoothing may still be preferable. For this reason, a space-variant image regulariser adapting its convexity to the local geometrical structures sounds desirable and appealing for imaging applications. 

\medskip


We stress that the fine analysis of \eqref{eq:energy_DTV_cont} in a functional setting becomes very challenging in the case $p<1$, since the extension of such regularising functionals to spaces similar to $BV(\Omega)$ is not trivial at all. For some theoretical considerations in this direction we refer the reader to \cite{Hintermuller2015}. To simplify the difficulties arising in such framework, our model is studied in a purely discrete setting. Its extension and analysis in a functional framework is left for future research.

\section{Statistical derivation via MAP estimation}
\label{sec:map}

A common statistical paradigm in image restoration is the MAP approach by which the restored image is obtained as a global minimiser of the negative log-likelihood distribution given the observed image $g$ and the known blurring operator $K$ combined with some prior probability on the unknown target image $u$, see, e.g., \cite{Huang1999,Zhu1998}. In formulas:
\begin{equation}
u^* \in \argmax_{u \in \mathbb{R}^n}\;
P(u|g;K) 
\;{=}\;
\argmin_{u \in \mathbb{R}^n} \;
\left\{ \,-\log P(g|u;K)-\log P(u)
\, \right\}.
\label{eq:map2}
\end{equation}
The equality above comes from the application of the Bayes' formula after dropping the normalisation term $P(g)$.

In the case of AWGN the likelihood term in \eqref{eq:map2} takes the following special form
\begin{equation}
P(g|u;K)=\prod_{i=1}^{n}\,\frac{1}{\sqrt{2\pi}\sigma}\,\text{exp}\bigg(-\frac{(Ku-g)_{i}^{2}}{2\sigma^{2}}\,\bigg)=\frac{1}{W}\,\text{exp}\bigg(-\frac{\lVert Ku-g\rVert_{2}^{2}}{2\sigma^{2}}\,\bigg),
\label{eq:g_like}
\end{equation}
where $\sigma>0$ denotes the AWGN standard deviation and $W > 0$ is a normalisation constant.

As far as the unknown image $u$ is concerned, a standard choice consists in its modelling via a Markov Random Field (MRF) such that its prior $P(u)$ takes the form of a Gibbs prior, whose general form reads:
\begin{equation}
P(u)= \frac{1}{Z} \prod_{i = 1}^{n} \text{exp}\,(\,-\alpha \, V_{\mathcal{N}_{i}}(u)\,)=  \frac{1}{Z} \, \text{exp}\,\bigg(\,-\alpha\,\sum_{i = 1}^{n}  V_{\mathcal{N}_{i}}(u)\,\bigg),
\label{eq:mrf}
\end{equation}
where $\alpha>0$ is the MRF parameter and, for every $i=1,\ldots,n$, $\mathcal{N}_i$ denotes the set of all neighbouring pixels of $i$ (also known as `clique'), $V_{\mathcal{N}_{i}}$ stands for the potential function on $\mathcal{N}_{i}$ and $Z$ is the normalising partition function not depending on $u$. Such MRF modelling has been widely explored in the context of Bayesian models for imaging and combined in \cite{Roth2009} with learning strategies to design data-driven filters over extended neighbourhoods.

For more model-oriented approaches, by setting for any $i=1,\ldots,n$, $V_{\mathcal{N}_{i}}(u):=\lVert (\nabla u)_{i} \rVert_{2}$, the Gibbs prior in \eqref{eq:mrf} reduces to the TV prior:
$
P(u)= \frac{1}{Z} \,\text{exp}\,\left(-\alpha\,\sum_{i = 1}^{n} \lVert (\nabla u)_{i} \rVert_{2}\right)
$ 
%
or equivalently interpreted by saying that each $\|(\nabla u)_i\|_2$ is distributed according to
an half-Laplacian distribution with parameter $\alpha>0$. Via similar considerations, in \cite{tvpl2} the authors have shown how such one-parameter model is in fact too restrictive to describe the statistical distribution of the gradient in real images.

For this reason, we proceed differently and model the joint distribution of the two partial derivatives of the gradient vector $(\nabla u)_i$ at any pixel by a Bivariate Generalised Gaussian Distribution (BGGD) 
\cite{MGGD}. Namely, for all $i=1,\ldots,n$ we assume that 
\begin{equation}
P((\nabla u)_i;p_i,\Sigma_i)
=
\frac{1}{2 \pi |\Sigma_i|^{1/2}} \, \frac{p_i}{\Gamma(2/p_i) \, 2^{\:\!2/p_i}}
\: 
\text{exp}\left(-\frac{1}{2}((\nabla u)_i^{T}\Sigma_i^{-1}(\nabla u)_i)^{p_i/2}\right),
\label{eq:MGGDi}
\end{equation}
where $\Gamma$ stands for the Gamma function, the \emph{covariance matrices}  $\Sigma_i \in \mathbb{R}^{2 \times 2}$ are symmetric positive definite with determinant $|\Sigma_i|$ and $p_i/2$ is often referred to as \emph{shape parameter}. Note that when 
in \eqref{eq:MGGDi} $p_i=2$ for every $i=1,\ldots,n$, then the BGGD reduces to a standard bivariate Gaussian distribution with pixel-wise covariance matrices $\Sigma_i$.

Proceeding similarly as above, we can then deduce the expression of the corresponding prior under such assumption. It reads:
\begin{equation} \label{eq:propreg}
P(u)= \frac{1}{Z} \,\text{exp}\,
\bigg(\,
-\,\frac{1}{2} 
\sum_{i = 1}^{n} \left( (\nabla u)_{i}^T \Sigma_i^{-1} (\nabla u)_{i} \right)^{p_i/2} 
 \,\bigg)
\end{equation} 
The symmetric positive definite matrices $\Sigma_i$ contain information on both the directionality and the scale of the BGGD at pixel $i$.
To see that explicitly, we consider their eigenvalue decomposition:
\begin{equation}
\Sigma_i
\;{=}\;
V_i^T E_i V_i,
\quad
E_i \;{=}\; \begin{pmatrix}
{e^{(1)}}_i & 0\\
0 & e^{(2)}_i
\end{pmatrix}, \quad {e^{(1)}}_i \geq e^{(2)}_i > 0, \quad V_i^T\! V_i = V_i V_i^T = I \, , 
\label{eq:FF4}
\end{equation}
where for every $i=1,\ldots,n$, ${e^{(1)}}_i, e^{(2)}_i$ are the (positive) eigenvalues of $\Sigma_i$, $V_i$ is the orthonormal (rotation) modal matrix and $I$ denotes the $2 \times 2$ identity matrix. We then rewrite the terms in the sum appearing in \eqref{eq:propreg} as
\begin{equation}
\Big( \:\! (\nabla u)_{i}^T \Sigma_i^{-1} (\nabla u)_{i} \:\! \Big)^{\frac{p_{i}}{2}} \! = \: 
\Big( \:\! (\nabla u)_{i}^T V_i^T E_i^{-1} V_i (\nabla u)_{i} \:\! \Big)^{\frac{p_{i}}{2}} \! = \:
\left\| \, E_i^{-1/2} V_i \, (\nabla u)_i \right\|_2^{p_{i}} ,
\end{equation}
whence by setting 
\begin{equation}
\Lambda_i := E_i^{-1/2}, \quad R_{\theta_i}  := V_i,
\label{siginvdec}
\end{equation}
and after recalling the definition of the $\mathrm{DTV}_{p}^{\mathrm{sv}}$ regulariser given in \eqref{eq:PMa}-\eqref{eq:PMb2} we observe that the prior in  \eqref{eq:propreg} can indeed be expressed as:
\begin{equation} \label{eq:propreg2}
P(u)
=
\frac{1}{Z} \,\text{exp}\,
\bigg(\,
-\,\frac{1}{2} \, 
\mathrm{DTV}_{p}^{\mathrm{sv}}(u)
\,\bigg).
\end{equation} 
%

By plugging the expression of the Gaussian likelihood \eqref{eq:g_like} and the BGGD prior \eqref{eq:propreg2}  in the MAP inference formula \eqref{eq:map2} and after dropping the constant terms, we finally obtain the DTV$_{p}^{\mathrm{sv}}$-L$_{2}$ image restoration model 
\eqref{eq:PMb} for blur and AWGN removal by setting $\mu=2/\sigma^{2}$. 


\section{Automatic estimation of the DTV$_{p}^{\mathrm{sv}}$  parameters}   \label{sec:parameter_estimation}
The very high flexibility of the proposed space-variant anisotropic DTV$_{p}^{\mathrm{sv}}$ regulariser \eqref{eq:PMa}-\eqref{eq:PMb2} would be useless without an effective procedure for automatically and reliably estimating all its parameters from the observed corrupted data $g$. 

In this section we propose a statistical optimisation strategy for the estimation of the covariance matrices $\Sigma_i$ and the parameters $p_i$ of the BGGD defined in \eqref{eq:MGGDi} when a collection of samples is available. Similar strategies estimating model parameters for directional regularisers from statistical priors have been proposed, e.g., in \cite{Peter2015} for anisotropic PDEs in the form \eqref{eq:anis_diff}. For simplicity, we will drop in the following the dependence on $i$ of the quantities appearing in \eqref{eq:MGGDi} and denote by $x:=\nabla u \in\mathbb{R}^2$ the local gradient of the image $u$. 
Firstly, we observe that the requirement for $\Sigma$ to be symmetric positive definite means:
\begin{equation}   \label{eq:Sigma_def}
\Sigma=\left[\begin{matrix}\sigma_{1} & \sigma_{3} \\ \sigma_{3} & \sigma_{2}\end{matrix}\right]
\qquad
\text{with}
\qquad
\begin{cases}
& \sigma_{1}>0\\
& |\Sigma|=\sigma_{1}\sigma_{2}-\sigma_{3}^2>0\\
\end{cases}
\end{equation}
As suggested in \cite{shape1,shape2,Pascal2013}, it is possible to decouple the spread and the directionality of the BGGD by introducing a further \emph{scale parameter} $m>0$, so that \eqref{eq:MGGDi} takes the following form:
\begin{equation}
P(x;p,\Sigma,m)=\frac{1}{\pi\Gamma\big(\frac{2}{p}\big)2^{\frac{2}{p}}}\frac{p}{2m|\Sigma|^{\frac{1}{2}}} \text{exp}\bigg(-\frac{1}{2m^{p/2}}(x^{T}\Sigma^{-1}x)^{p/2}\bigg),
\label{pdf3}
\end{equation}

By imposing that the trace of the covariance matrix $\Sigma$ is fixed and equal to the dimension of the ambient space, i.e. $\text{tr}(\Sigma)=d=2$, we easily get the following expression of the constraint set $\mathcal{C}$ for the parameters $p, m, \sigma_1,\sigma_2, \sigma_3$ to be well defined:
\begin{equation}
\mathcal{C}:=
\begin{cases}
p>0 \\
m>0\\
\sigma_{1}+\sigma_{2}=2 \\
\sigma_{1}\sigma_{2}-\sigma_{3}^2>0\\
\end{cases}
\qquad
\longrightarrow
\qquad
\mathcal{C}=
\begin{cases}
p>0 \\
m>0\\
\sigma_{1}^2+\sigma_{3}^2-2\sigma_{1}<0.\\
\end{cases}
\label{ct}
\end{equation}
The set $\mathcal{C}$ is an open (unbounded) semi-cylinder in $\mathbb{R}^4$.  After a change of coordinates which shifts the centre of the circle in the $\sigma_1-\sigma_3$ plane to the origin, we obtain the following expression of $\Sigma^{-1}$
\begin{equation}
\tilde{\sigma}_{1} := 1 - \sigma_{1}
\quad
\longrightarrow
\quad
\Sigma^{-1}=\frac{1}{1-\tilde{\sigma}_{1}^2-{\sigma}_{3}^2}\left[\begin{matrix}1+\tilde{\sigma}_{1} & -{\sigma}_{3} \\ -{\sigma}_{3} &1-\tilde{\sigma}_{1}\end{matrix}\right].
\label{mat_t}
\end{equation}
To avoid heavy notation, we will still denote in the following by $\sigma_1$ the same variable after this change of coordinates.

\subsection{ML estimation of the BGGD parameters}
For any point $\Omega$ let $\mathcal{N}(x)=\left\{x_{1},...,x_{N}\right\}\subset\mathbb{R}^{2\times N}$ denote the neighbourhood centred in $x$ of $N$ independent and identically distributed samples drawn from a BGGD with unknown parameters $ p,\sigma_1,\sigma_3, m \in \mathcal{C}$. Then, the corresponding \textit{likelihood} function reads:
\begin{eqnarray}
\mathcal{L}(p,\Sigma,m;x)  &=&\prod_{x_j\in\mathcal{N}(x)} P(x_{j};p,\Sigma,m)=\prod_{j=1}^{N} P(x_{j};p,\Sigma,m)\\ \nonumber
 &=&\Bigg[  \frac{1}{|\Sigma|^{1/2}} \frac{p}{2\pi \Gamma\big(\frac{2}{p}\big)2^{2/p}m}\Bigg]^{N} \text{exp}\bigg(-\frac{1}{2m^{p/2}}\sum_{j=1}^{N}(x_{j}^{T} \Sigma^{-1}x_{j})^{p/2}\bigg)
\label{likelihood}
\end{eqnarray}
We now look for $(p^{*},\Sigma^{*},m^{*})\in\mathcal{C}$ maximising $\mathcal{L}$. Equivalently, by taking the negative logarithm, we aim to solve:
\begin{equation}
(p^{*},\Sigma^{*},m^{*})\longleftarrow \argmin_{p,\Sigma,m~\in~\mathcal{C}} ~\left\{ \mathcal{F}(p,\Sigma,m;x) :=-\log~\mathcal{L}(p,\Sigma,m;x)\right\},
\label{pb}
\end{equation}
whence, by recalling the fundamental property $\Gamma(z+1)=z\Gamma(z)$ for every $z\in\mathbb{R}$, we deduce:
\begin{align}
 \mathcal{F}(p,\Sigma,m;x) & = - \Bigg[ N\log \bigg(\frac{1}{|\Sigma|^{1/2}}\frac{1}{\pi \Gamma\big(\frac{2}{p} + 1\big)2^{2/p}m}\bigg) -\;\frac{1}{2m^{p/2}}\sum_{j=1}^{N}(\,x_{j}^{T}\,\Sigma^{-1}\,x_{j}\,)^{p/2}\Bigg] \notag \\
& =  N\log \bigg(|\Sigma|^{1/2}\pi \Gamma\bigg(\frac{2}{p} + 1\bigg)2^{2/p}\bigg) + N \log m +\;\frac{1}{2m^{p/2}}\sum_{j=1}^{N}(\,x_{j}^{T}\,\Sigma^{-1}\,x_{j}\,)^{p/2}. \notag
\label{log}
\end{align}
Note that $\mathcal{F}$ is differentiable on $\mathcal{C}$. Therefore, by simply imposing the first order optimality condition for $m$, we can find a closed formula for $m^*$ as follows:
\begin{equation}  \label{eq:expr_m}
\frac{\partial \mathcal{F} }{\partial m}=\frac{N}{m}-\frac{p}{4m^{\frac{p}{2}+1}}\sum_{j=1}^{N}(x_{j}^{T}\Sigma^{-1}x_{j})^{p/2}\quad\longrightarrow \quad
m^*=\bigg(\frac{p}{4N}\sum_{j=1}^{N}(x_{j}^{T}\Sigma^{-1}x_{j})^{p/2}\bigg)^{\frac{2}{p}}.
\end{equation}
We now substitute this formula in the expression of $\mathcal{F}$, thus getting:
\begin{equation}
\mathcal{F}(p,\Sigma;x) = N\log \bigg(|\Sigma|^{1/2}\pi \Gamma\bigg(\frac{2}{p} + 1\bigg)2^{2/p}\bigg) + \frac{2N}{p} \log \bigg(\frac{p}{4N}\sum_{j=1}^{N}(x_{j}^{T}\Sigma^{-1}x_{j})^{p/2}\bigg)+ \frac{2N}{p}.
\label{fun_semifin}
\end{equation}
By making explicit the dependence of $\mathcal{F}$ on the entries of $\Sigma$, we have that \eqref{fun_semifin} turns into:
\begin{eqnarray} \label{fun_fin}
\mathcal{F}(p,\sigma_{1},\sigma_{2},\sigma_{3};x)&=&
N\log \bigg(\frac{1}{|\Sigma|^{1/2}}\pi \Gamma\bigg(\frac{2}{p} + 1\bigg)2^{2/p}\bigg)+ \frac{2N}{p} + \frac{2N}{p}\log \frac{p}{4N} \\
\nonumber &+& \frac{2N}{p}\log\bigg(\sum_{j=1}^N(\sigma_{2}x_{j,1}^{2}+\sigma_{1}x_{j,2}^{2}-2\sigma_{3}x_{j,1}x_{j,2})^{p/2}\bigg).
\end{eqnarray}

We now study the behaviour of $\mathcal{F}$ expressed as above as $(p,\sigma_{1},\sigma_{2},\sigma_{3})$ approach the boundary of the set $\mathcal{C}$ defined in \eqref{ct}. Thanks to the formula for $m^*$ derived in \eqref{eq:expr_m}, we start noticing that $\mathcal{C}$ can be expressed in fact as a subset in $\mathbb{R}^3$ defined by the variables $p, \sigma_1$ and $\sigma_3$ only. By further switching to polar coordinates in the $\sigma_{1}-\sigma_{3}$ plane, we get:
\begin{equation}  \label{eq:polar_coordinate}
(\sigma_{1},\sigma_{3})=\varrho(\cos \phi, \sin \phi), \quad 0\leq \varrho<1, \quad\phi\in [0,2\pi),
\end{equation}
so that the matrices in \eqref{mat_t} take the following form:
\begin{equation}
\Sigma=\left[\begin{matrix}1-\varrho\cos \phi & \varrho\sin \phi \\ \varrho\sin \phi &1+\varrho\cos \phi\end{matrix}\right], \qquad
\Sigma^{-1}=\frac{1}{1-\varrho^2}\left[\begin{matrix}1+\varrho\cos \phi & -\varrho\sin \phi \\ -\varrho\sin \phi &1-\varrho\cos \phi\end{matrix}\right],
\label{mat}
\end{equation}
and the functional $\mathcal{F}$ in \eqref{fun_fin} becomes: 
\begin{eqnarray}
\mathcal{F}(p,\phi,\varrho;x)&=&N\log\bigg( \Gamma\bigg(\frac{2}{p}+1\bigg)\frac{\pi}{\sqrt{1 - \varrho^2}}\bigg(\frac{p}{2N}\bigg)^{2/p}\,\bigg)+\frac{2N}{p}+\frac{2N}{p}\log \frac{p}{4N}\notag\\
&+&\frac{2N}{p} \log \Bigg[\sum_{j=1}^{N}((1+\varrho\cos\phi)x_{j,1}^2+(1-\varrho\cos\phi)x_{j,2}^2-2\varrho\sin \phi~ x_{j,1}x_{j,2})^{p/2}\Bigg].
\label{eq:funct_unbound}
\end{eqnarray}
As a conclusion, we can finally rewrite the ML problem \eqref{pb} as the following constrained optimisation problem
\begin{equation}
(p^{*},\phi^{*},\varrho^{*})\in \argmin_{\substack{p\in(0,\infty),\\
\phi\in[0,2\pi),\\
\varrho\in[0,1)}}~ \mathcal{F}(p,\phi,\varrho).   \label{eq:min_prob_unbb}
\end{equation}
Note, that since the problem \eqref{eq:min_prob_unbb} is formulated over a non-compact set of $\R^3$, the existence of its solution is in general not guaranteed.

\subsection{Reformulation on a compact set}  \label{sec:compact}

One possible way to overcome this problem consists in characterising explicitly the configurations of the samples $x_1,\ldots,x_N\in\mathcal{N}(x)$ for which the functional $\mathcal{F}$ in \eqref{eq:funct_unbound} does not attain its minimum inside $\mathcal{C}$. To do so, let us first rename the last term in \eqref{eq:funct_unbound} as:
\begin{equation}
A(\phi,\varrho):=\frac{2N}{p} \log \Bigg[\sum_{i=j}^{N}((1+\varrho\cos\phi)x_{j,1}^2+(1-\varrho\cos \phi)x_{j,2}^2-2\varrho\sin \phi~ x_{j,1}x_{j,2})^{p/2}\Bigg].
\end{equation}
For any $p\in (0,+\infty)$, if $A(\phi,\varrho)$ is bounded as $\varrho\to 1^-$, then the functional $\mathcal{F}$ in \eqref{eq:funct_unbound} tends to $+\infty$ and the minimum is necessarily attained in the interior of $\mathcal{C}$.
However, if $A(\phi,\varrho)$ is unbounded as $\varrho\to 1^-$, nothing can be said about the behaviour of $\mathcal{F}$ at the boundary and, as a consequence, nothing can be said about its minima. In particular, in this situation there may exist one or multiple configurations of the samples $x_1,\ldots,x_N\in\mathcal{N}(x)$ for which $\mathcal{F}$ tends to $-\infty$ at the boundary.
In order to characterise such configurations, note that as $\varrho\to 1^-$ we have that by continuity:
\begin{equation}
A(\phi,\varrho)\to \frac{2N}{p} \log \Bigg[\sum_{j=1}^{N} (\sqrt{1+\cos \phi}~x_{j,1}-\sqrt{1-\cos \phi}~x_{j,2})^{p}\Bigg],
\end{equation}
which tends to $-\infty$ if and only if the argument of the logarithm tends to zero, i.e. when
\begin{equation}  \label{eq:configurations}
x_{j,2}=\sqrt\frac{\sqrt{1+\cos \phi}}{\sqrt{1-\cos \phi}}~x_{j,1},\qquad \forall j=1,...,N.
\end{equation}
This situation corresponds to the case when the samples $x_{j}$ lie all on the line passing through the origin with slope $\sqrt\frac{\sqrt{1+\cos \phi}}{\sqrt{1-\cos \phi}}$. 

\medskip

A possible way to guarantee the existence of solutions of the problem \eqref{eq:min_prob_unbb} is to re-formulate the problem over a compact subset of $\R^3$. Although this may sound a little bit artificial, note that for imaging applications such assumption makes perfect sense for different reasons. Firstly, as far as the range for the parameter $\varrho$ is concerned, note that the degenerate configurations \eqref{eq:configurations} happening as $\varrho$ approaches $1^-$ are easily detectable in a pre-processing step and, in practice, very unlikely for natural images since they would correspond to situations where gradient components are linearly correlated for any sample $j=1,\ldots,N$. Therefore, provided we can perform such preliminary check, the case $\varrho=1$ becomes admissible since no other possible configurations are allowed under this choice.

Regarding the admissible values for $p$, we notice that the more we enforce sparsity (i.e. the closer $p$ is to zero), the more the BGGD will tend to a Dirac delta distribution, making the estimation of local anisotropy in a neighbourhood of the point considered very hard (see Section \ref{sec:parameter_estimation} for more details). Additionally, as it is commonly done in previous work for variable exponent models for imaging, an upper bound for such values -- typically chosen as $\bar{p}\geq 2$ -- can be fixed. Therefore, in practice, we can fix lower and upper bounds $0<\varepsilon<\bar{p}$ for the exponent range.


After these observation, we can then reformulate the problem \eqref{eq:min_prob_unbb} as follows
\begin{eqnarray}\label{minpbpar}
(p^{*},\phi^{*},\varrho^{*})&\longleftarrow&\argmin_{p,\theta,\varrho} \mathcal{F}(p,\phi,\varrho;x)\\
\notag&&\\
\notag &\textrm{s.t.}& p\in[\varepsilon,\bar{p}],\;\;0 \leq \varrho \leq 1,\;\;0 \leq \phi \leq 2\pi,
\end{eqnarray}
where now the constraint set is compact, which, combined with the continuity of $\mathcal{F}$, guarantees that the minimisation problem admits a minimum.

\smallskip

Before carrying on with our discussion, we recall once again that the ML procedure described above is \emph{local}, i.e. it has to be repeated for any pixel in the image domain, thus resulting in the estimation of the parameter map $(p^*_i,\phi^*_i,\varrho^*_i)$ for $i=1,\ldots,n$.
\\For each pixel $i=1,\ldots,n$, the triple of estimated parameters is involved in the computation of the matrices $\Lambda_i,R_{\theta_i}$ defining the regulariser in \eqref{eq:PMa}.
Relying on \eqref{mat}, the eigenvalues $e_i^{(1)},e_i^{(2)}$ can be easily computed. Observe that, due to the normalisation condition on the trace introduced in \eqref{ct}, the minimum eigenvalue $e_i^{(2)}$ can be directly derived by the maximum eigenvalue $e_i^{(1)}$:
\begin{equation}
e_i^{(1)} = 1 + \varrho_i, \quad e_i^{(2)} = 2 - e_i^{(1)} = 1 -\varrho_i.
\end{equation}
Therefore, recalling \eqref{siginvdec}, the matrix $\Lambda_i$ is obtained as follows:
\begin{equation}  \label{def:lambda_e1}
\Lambda_i := \begin{pmatrix}
\lambda_i^{(1)} & 0 \\ 0 & \lambda_i^{(2)}
\end{pmatrix}=\begin{pmatrix}
\frac{1}{\sqrt{e_i^{(1)}}} & 0 \\ 0 & \frac{1}{\sqrt{e_i^{(2)}}}
\end{pmatrix}.
\label{lambdai}
\end{equation}
Once $e_i^{(1)}$ is available, its corresponding eigenvector  $(v_1)_i$, satisfying $\Sigma_i (v_1)_i = {e^{(1)}}_i (v_1)_i$ , can be further calculated using the formula
\begin{equation}
(v_1)_i = \sqrt{\frac{1 + \cos \phi_i}{2}}\begin{bmatrix}
\frac{\sin \phi_i}{\sqrt{1 + \cos \phi_i}}\\
1\\
\end{bmatrix}.
\end{equation}
As a consequence, the local angle $\theta_i$ describing the local orientation is computed by
\begin{equation}
\theta_i = \arctan{\frac{\sqrt{1 + \cos \phi_i}}{\sin \phi_i}},
\label{arctanth}
\end{equation}
and the rotation matrix $R_{\theta_i}$ is given as in \eqref{eq:PMb2}.



Furthermore, it is helpful to represent the estimated BGGD to visualise its shape in the plane $((D_h u)_i,(D_v u)_i)$. 
In order to draw the corresponding level curves, we only need the maximum eigenvalue $e_i^{(1)}$ and the rotation angle $\theta_i$. 
Such curves are the ellipses having semi-axes $a_i$, $b_i$, and eccentricity $\epsilon_i$ given by:
\begin{equation}
a_i := \sqrt{{e^{(1)}}_i},\quad b_i := \sqrt{e^{(2)}_i},\quad \epsilon_i := \frac{\sqrt{{a_i}^2 - {b_i}^2}}{a_i} = \frac{\sqrt{{e^{(1)}}_i - e^{(2)}_i}}{\sqrt{{e^{(1)}}_i}}. 
\label{axis}
\end{equation}

An illustrative drawing of the anisotropy ellipses described above is reported in Figure \ref{fig:ellipses}.

\begin{figure}
    \centering
    \includegraphics[height=5cm]{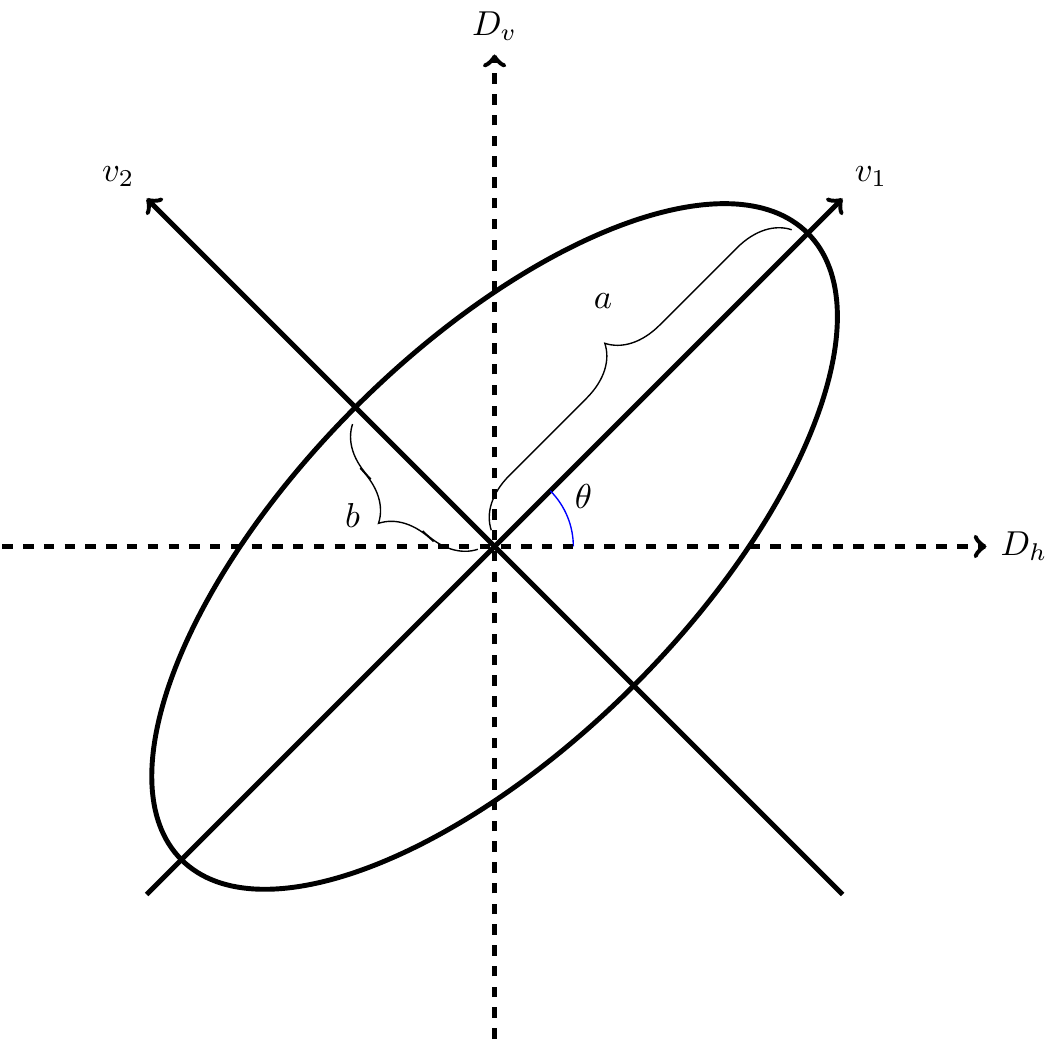}
    \caption{Representation of anisotropy ellipses describing BGGD level lines in the plane $D_h-D_v$ in terms of the eigenvalues and eigenvectors of the estimated matrix $\Sigma$.}
    \label{fig:ellipses}
\end{figure}


\section{Existence of solutions}  
\label{sec:exist} 

In this section, we provide an existence result for the solutions of the proposed DTV$_p^\mathrm{sv}$-L$_2$ variational model \eqref{eq:PMa}-\eqref{eq:PMb}. In general, the DTV$_p^\mathrm{sv}$-L$_2$ functional is not convex, therefore it is not guaranteed to admit a unique global minimiser. However, by applying a general lemma whose proof can be found in \cite[Lemma 2.7.1]{Ciak2015} we will prove that existence of global minimisers is guaranteed. In the following, we will use the notations $\mathrm{null}(M)$, $\mathrm{span}(v_1,\ldots,v_m)$, $I_m$, $\bm{0}_m$ and $\bm{1}_m$ to denote the null space of the linear operator $M$, the linear span of the set of vectors $v_1,\ldots,v_m$, the identity matrix of order $m$ and the all-zeros and all-ones $m$-dimensional vectors, respectively. We have that the following Lemma holds true \cite{Ciak2015}.

\begin{lemma}
\label{lem:sumcoerc}
Let $A_1 \in \R^{m \times n}$, $A_2 \in \R^{q \times n}$ be two linear operators satisfying
\begin{equation}
\mathrm{null}(A_1) \,\;{\cap}\;\, \mathrm{null}(A_2) \; = \; \{ \bm{0}_n \} \, ,
\label{eq:nullinters}
\end{equation}
and let $f_1: \R^m \to [-\infty, +\infty]$ and $f_2: \R^q \to [-\infty, +\infty]$ be two proper, lower semicontinuous and coercive functions.
Then, the function $h: \R^n \to [-\infty, +\infty]$ defined by
\begin{equation}
h(x) := f_1(A_1 x) + f_2(A_2 x)
\end{equation}
is lower semicontinuous and coercive.
\end{lemma}

We now apply this result to the DTV$_p^\mathrm{sv}$-L$_2$  model.

\begin{proposition}
The $\mathrm{DTV}_p^\mathrm{sv}$-$\mathrm{L}_2$ functional  $\mathcal{J}: \R^n \to \R$ defined in \eqref{eq:PMa}-\eqref{eq:PMb} is continuous, bounded from below by zero and coercive, hence it admits global minimisers.
\end{proposition}

\begin{proof}
Let $A_1 \in \R^{2n \times n}$ be the matrix defined by
\begin{equation}
A_1 = L D, \quad L = \diag \left( L_1,L_2,\ldots,L_n \right), \quad L_i = \Lambda_i R_{\theta_i} \in \R^2, \;\: i =1,\ldots,n,
\label{eq:AA}
\end{equation}
with $\Lambda_i,R_{\theta_i} \in \R^2$ the full rank matrices in \eqref{eq:PMb2} and $D \in \R^{2n \times n}$ a finite difference operator discretising the image gradient, let $A_2 = K$, and let $f_1: \R^{2n} \to \R$, $f_2: \R^n \to \R$ be the functions defined by
\begin{equation}
\begin{array}{rcll}
f_1(y) &\!\!{:=}\!\!& 
\sum_{i=1}^n \left\| ( y_{2i-1} , y_{2i} ) \right\|_2^{p_i}, &
\quad y \in \R^{2n},\vspace{0.13cm}\\
f_2(z) &\!\!{:=}\!\!& \displaystyle{\frac{\mu}{2} \left\| z - g \right\|_2^2,} & 
\quad z \in \R^{n}.
\end{array}
\label{eq:fandg}
\end{equation}
Then, the DTV$_p^\mathrm{sv}$-L$_2$ energy functional in \eqref{eq:PMa}-\eqref{eq:PMb} can be written as
\begin{equation} \label{eq:cost_functional}
\mathcal{J}(u)= f_1(A_1 u) + f_2(A_2 u). 
\end{equation}
As the block diagonal matrix $L$ in \eqref{eq:AA} has full rank (all matrices $L_i$ have full rank), the linear operator $A_1$ has the same null space as the discrete gradient operator $D$. 
It follows that  
\begin{equation}
\big( \mathrm{null}(A_1) = \mathrm{null}(D) = \mathrm{span}(\bm{1}_n) \big)
\,\;{\cap}\;\, 
\big( \mathrm{null}(A_2) = \mathrm{null}(K) \big) \; = \; \{ \bm{0_n} \} \, ,
\label{eq:AKnullinters}
\end{equation}
in fact constant images do not belong to the null space of the linear blur operator $K$.
Furthermore, functions $f_1$ and $f_2$ in \eqref{eq:fandg} are clearly continuous, bounded from below by zero and coercive. 
It thus follows from Lemma \ref{lem:sumcoerc} that the DTV$_p^\mathrm{sv}$-L$_2$ functional $\mathcal{J}$ in \eqref{eq:cost_functional} is continuous, bounded from below by zero and coercive, hence it admits at least one global minimiser.
\end{proof}

Uniqueness of solutions is in general not guaranteed. However, if the functional is strictly convex, this trivially holds.

\begin{corollary}
Let $\mathcal{J}:\R^n\to\R$ be the $\mathrm{DTV}_p^\mathrm{sv}$-$\mathrm{L}_2$ functional defined in \eqref{eq:PMa}-\eqref{eq:PMb}. If $p_i > 1$ for every $i = 1,\ldots,n$, then  $\mathcal{J}$ is strongly convex. Hence it admits a unique global minimiser.
\end{corollary}

Note, however, that as we discussed in the introduction, in this work we are more interested in the  non-convex case, e.g. when there exists at least one $i \in \{1,\ldots,n\}$  such that $p_i<1$, since in this better regularisation properties are enforced in  DTV$_p^\mathrm{sv}$-L$_2$. Therefore, in our applications uniqueness in general will not be guaranteed and we will be generally dealing with the case of local minima.

\section{Numerical solution by ADMM}
\label{sec:admm}
We can now describe the ADMM-based iterative algorithm \cite{BOYD_ADMM} used to  solve numerically the proposed DTV$_p^\mathrm{sv}$-L$_2$ model \eqref{eq:PMa}--\eqref{eq:PMb} once the values of all the parameters $p_i,\theta_i,\lambda_i^{(i)},\lambda_i^{(2)}$, $i = 1,\ldots,n$, which define the regulariser have been set according to the procedure illustrated in Section \ref{sec:parameter_estimation}.
To this purpose, first we introduce two auxiliary variables $r \in \mathbb{R}^n$ and $t \in \mathbb{R}^{2n}$ and rewrite model \eqref{eq:PMa}--\eqref{eq:PMb} 
in the following equivalent constrained form:
\begin{eqnarray}
\{ \, u^*,r^*,t^* \}
\:\;{\leftarrow}\;\:
\argmin_{u,r,t}
&\:&\bigg\{ \:
\sum_{i = 1}^{n} \left\| \Lambda_i R_{\theta_i} t_i \right\|_2^{p_i}
\;{+}\;
\frac{\mu}{2} \, \| r \|_2^2
\: \bigg\} \, ,
\label{eq:PM_ADMM_a} \vspace {0.2cm} \\
\mathrm{subject}\:\mathrm{to:}
&& \; r \;{=}\; K u - g \, , \;\: t \;{=}\; D u \, ,
\label{eq:PM_ADMM_b}
\end{eqnarray}
where
$D := (D_h^T,D_v^T)^T \in \mathbb{R}^{2n \times n}$ denotes the discrete gradient operator with \mbox{$D_h,~D_v \in \mathbb{R}^{n \times n}$} two linear operators representing finite difference discretisations of the first-order partial derivatives of the image $u$ in the horizontal and vertical direction, respectively, and where $t_i \:{:=}\: \big( (D_h u)_i \,,\, (D_v u)_i \big)^T \in \mathbb{R}^2$ stands for the discrete gradient of $u$ at pixel $i$.  
We notice that the auxiliary variable $t$ is introduced to transfer the discrete gradient operator out of the possibly non-convex non-smooth regulariser whereas the variable $r$ is aimed to adjust the regularisation parameter $\mu$ along the ADMM iterations such that the computed solution $u^*$ satisfies the discrepancy principle \cite{WC12}, i.e. belongs to the discrepancy set $\mathcal{D}$ in \eqref{discr_set}.
%
%
%

In order to solve problem \eqref{eq:PM_ADMM_a}--\eqref{eq:PM_ADMM_b} via ADMM, we start defining the augmented Lagrangian functional as follows:
\begin{eqnarray}
\mathcal{L}(u,r,t;\rho_r,\rho_t)
&:{=}&
\displaystyle{
	\sum_{i = 1}^{n}  \left\| \Lambda_i R_{\theta_i} t_i \right\|_2^{p_i}
	\;{+}\;
	\frac{\mu}{2} \, \| r \|_2^2
	\,{-}\; \langle \, \rho_t , t - D u \, \rangle
	\;{+}\;
	\frac{\beta_t}{2} \: \| t - D u \|_2^2
} \nonumber \\
&&\displaystyle{
	{-}\; \langle \, \rho_r , r - (Ku-g) \, \rangle
	\,\;\;{+}\;
	\frac{\beta_r}{2} \, \| \, r - (Ku-g) \|_2^2 \, ,
}
%
\label{eq:PM_AL}
\end{eqnarray}
where $\beta_r, \beta_t > 0$ are the scalar penalty parameters, while $\rho_r \in \mathbb{R}^n$, $\rho_t \in \mathbb{R}^{2n}$
are the vectors of Lagrange multipliers associated with the linear constraints $r = Ku-g$ and $t = Du$
in \eqref{eq:PM_ADMM_b}, respectively.
%
%

By setting for simplicity $x := (u,r,t)$, 
$y := (\rho_r,\rho_t)$, 
$X := \mathbb{R}^n \times \mathbb{R}^n \times \mathbb{R}^{2n}$ and $Y := \mathbb{R}^n \times \mathbb{R}^{2n}$, we observe that solving \eqref{eq:PM_ADMM_a}--\eqref{eq:PM_ADMM_b} amounts to seek for the solutions of the following saddle point problem:
\begin{eqnarray}
\mathrm{Find}&&
\;\, (x^*;y^*)
\;\;{\in}\;\;
X \times Y \nonumber \\
\mathrm{such}\;\mathrm{that}&&
\; \mathcal{L}(x^*;y)
\:\;{\leq}\;\;
\mathcal{L}(x^*;y^*)
\:\;{\leq}\;\;
\mathcal{L}(x;y^*) 
\;\;\;\;\: \forall \: (x;y)
\;\;{\in}\;\;
X \times Y
\: ,
\label{eq:PM_new_S}
\end{eqnarray}
where the augmented Lagrangian functional $\mathcal{L}$ is defined in \eqref{eq:PM_AL}.

Upon suitable initialisation, and for any $k\geq 0$, the $k$-th iteration of the ADMM iterative algorithm applied to solve the saddle-point problem \eqref{eq:PM_new_S} reads as follows:

\begin{eqnarray}
&
u^{(k+1)} &
\;{\leftarrow}\;\;\,\,
\argmin_{u \in \mathbb{R}^n} \;
\mathcal{L}(u,r^{(k)},t^{(k)};\rho_r^{(k)},\rho_t^{(k)}) \, ,
\label{eq:PM_ADMM_u} \\
&
r^{(k+1)} &
\;{\leftarrow}\;\;\,\,
\argmin_{r \in \mathbb{R}^n} \;
\mathcal{L}(u^{(k+1)},r,t^{(k)};\rho_r^{(k)},\rho_t^{(k)}) \, ,
\label{eq:PM_ADMM_r} \\
&
t^{(k+1)} &
\;{\leftarrow}\;\;\,\,
\argmin_{t \in \mathbb{R}^{2n}} \;
\mathcal{L}(u^{(k+1)},r^{(k+1)},t;\rho_r^{(k)},\rho_t^{(k)}) \, ,
\label{eq:PM_ADMM_t} \\
&
\rho_r^{(k+1)} &
\;{\leftarrow}\;\;\,\,
\rho_r^{(k)} \;{-}\; \beta_r \, \big( \, r^{(k+1)} \;{-}\; (K u^{(k+1)}-g) \, \big) \, ,
\label{eq:PM_ADMM_lr} \\
&
\rho_t^{(k+1)} &
\;{\leftarrow}\;\;\,\,
\rho_t^{(k)} \;{-}\; \beta_t \, \big( \, t^{(k+1)} \;{-}\; D u^{(k+1)} \, \big) \, .
\label{eq:PM_ADMM_lt}
\end{eqnarray}
We notice that sub-problems
\eqref{eq:PM_ADMM_u} and \eqref{eq:PM_ADMM_r} for the primal variables $u$ and $r$ admit solutions   
based on formulas given in \cite{CMBBE} for identical  sub-problems. 
In particular, sub-problem \eqref{eq:PM_ADMM_u} for $u$ reduces to the solution of the following $n \times n$ system of linear equations
\begin{equation}
\left(
D^T D
+ \frac{\beta_r}{\beta_t} K^T K
\right)
u
=
D^T \left( t^{(k)} - \frac{1}{\beta_t} \rho^{(k)}_t \right)
+
\frac{\beta_r}{\beta_t} K^T \left( r^{(k)} - \frac{1}{\beta_r} \rho^{(k)}_r + g  \right)
\: ,
\label{eq:sub_u_sol}
\end{equation}
which is solvable since 
\begin{equation}
\mathrm{null}\left( D^T D + 
\frac{\beta_r}{\beta_t} K^T K \right)
=
\mathrm{null}\left( D^T D \right) 
\;{\cap}\;
\mathrm{null}\left( K^T K \right)
=
\mathrm{null}\left( D \right) 
\;{\cap}\;
\mathrm{null}\left( K \right)
=
\{ \bm{0}_n \} ,
\end{equation}
where last equality has been previously stated in \eqref{eq:AKnullinters}.
Assuming periodic boundary conditions for $u$ - such that that both $D^T D$ and $K^T K$ are block circulant matrices with circulant blocks (BCCB) - the linear system \eqref{eq:sub_u_sol} can be solved efficiently by one application of the forward 2D Fast Fourier Transform (FFT) and one application of the inverse 2D FFT, each at a cost of $O(n \log n)$.

The solution of the sub-problem \eqref{eq:PM_ADMM_r} for $r$ is obtained by computing first the vector 
\begin{equation}
w^{(k+1)} \;{=}\;\: Ku^{(k+1)} -\, g \: + \, \frac{1}{\beta_r} \, \rho_r^{(k)} \; ,
\label{eq:v_def}
\end{equation}
and then, recalling \cite{CMBBE} and the definition of the discrepancy set in \eqref{discr_set}, by computing jointly the new values of both the regularisation parameter $\mu$ and the variable $r$ as follows: 
\begin{equation}
\label{eq:sub_r_q2_q2_sol}
\begin{array}{l}
\|  w^{(k+1)} \|_2 \leq \delta \quad
\!\!\Longrightarrow\!\! \quad
\mu^{(k+1)} = 0, \hspace{3.55cm}
r^{(k+1)} = w^{(k+1)} \\
\| w^{(k+1)} \|_2 > \delta \quad
\!\!\Longrightarrow\!\!  \quad
\mu^{(k+1)} = \beta_r \big( \| \, w^{(k+1)} \|_2 / \delta - 1 \big), \quad\!\!
r^{(k+1)}=\delta w^{(k+1)} / \| w^{(k+1)} \|_2. \end{array}
\end{equation}
As far as the minimisation sub-problem for $t$ in \eqref{eq:PM_ADMM_t} is concerned, after simple algebraic manipulations, we deduce that it can be re-written as follows:
\begin{eqnarray*}
%
%
&t^{(k+1)}& \;{\leftarrow}\; \argmin_{t \in \mathbb{R}^{2n}}~
\sum_{i=1}^{n}
\left\{
\left\| \Lambda_i R_{\theta_i} t_i \right\|_2^{p_i}
\;{+}\;
\frac{\beta_t}{2}
\left\| t_i - \left( \left( D u^{(k+1)} \right)_i + \frac{1}{\beta_t} \left( \rho_t^{(k)} \right)_i \right) \right\|_2^2
\right\} 
\: .
\end{eqnarray*}
%
Solving the $2n$-dimensional minimisation problem above is thus equivalent to solve the $n$ following independent $2$-dimensional problems:
\begin{eqnarray}
&t^{(k+1)}_i& {\leftarrow}\; \argmin_{t_i \in \mathbb{R}^2}
\left\{ \,
\left\| \Lambda_i R_{\theta_i} t_i \right\|_2^{p_i}
\;{+}\;
\frac{\beta_t}{2} \left\| t_i - q_i^{(k+1)} \right\|_2^2
\,\right\}
,
\quad i = 1, \ldots , n \: , \label{eq:sub_t_i}
\end{eqnarray}
where the vectors $q^{(k+1)}_i \in \mathbb{R}^2$ are defined explicitly at any iteration by
\begin{equation}
q^{(k+1)}_i \:\;{:=}\;\: \left( D u^{(k+1)} \right)_i + \frac{1}{\beta_t} \left( \rho^{(k)}_t \right)_i \;\: ,
\quad i = 1, \ldots , n \: .
\label{eq:sub_t_i_q}
\end{equation}

The solutions of the $n$ bivariate optimisation problems in \eqref{eq:sub_t_i} requires the computation of a special proximal mapping operator. We dedicate the following Section \ref{sec:po} to carefully discuss the solution of this optimisation problem and show that it can be eventually re-written as a one-dimensional optimisation problem and thus solved efficiently.

\medskip

To summarise, we report in Algorithm \ref{alg:1} the pseudocode of the proposed ADMM iterative
scheme used to solve the saddle-point problem \eqref{eq:PM_AL}--\eqref{eq:PM_new_S}.

Over the last decades, the ADMM algorithm has been applied to a wide range of convex and non-convex optimisation problems arising in several areas of signal and image processing.
In convex settings, several convergence results have been established for ADMM-type algorithms, see for example \cite{HY} and references therein.  Such convergence results cover the proposed DTV$_{p}^{\mathrm{sv}}$-L$_2$ model in the special convex case when $ p_i \ge 1$ for every $i=1,\ldots,n$. However, very few studies on the convergence properties of ADMM in non-convex regimes have been performed. To the best of our knowledge, provable convergence results of ADMM in non-convex regimes are still very limited to particular classes of problems and under certain conditions, see, e.g. \cite{HLR, Wang2018, Bolte2018}.
Nevertheless, from an empirical point of view, the ADMM works extremely well for various applications involving non-convex objectives, thus suggesting heuristically its good performance in such cases as well.

%
\begin{algorithm}
\caption{ADMM scheme for the solution of problem \eqref{eq:PMa}--\eqref{eq:PMb} \vspace{0.05cm} }
\label{alg:1}
\vspace{0.2cm}
{\renewcommand{\arraystretch}{1.2}
\renewcommand{\tabcolsep}{0.0cm}
\vspace{-0.08cm}
\begin{tabular}{ll}
\textbf{inputs}:      & observed image $\,g \:{\in}\; \mathbb{R}^n$, $\:$ noise standard deviation $\sigma > 0$ \vspace{0.04cm} \\
\textbf{parameters}:$\;\;\:$      & discrepancy parameter $\tau \simeq 1$, $\:$ ADMM penalty parameters $\beta_r, \beta_t > 0$ \vspace{0.04cm} \\
\textbf{output}: $\;\:$    & approximate solution $\,u^* {\in}\; \mathbb{R}^n$ of \eqref{eq:PMa}--\eqref{eq:PMb} \vspace{0.2cm} \\
\end{tabular}
}
\hspace{4cm}
\vspace{0.1cm}
{\renewcommand{\arraystretch}{1.2}
\renewcommand{\tabcolsep}{0.0cm}

\begin{tabular}{rcll}
1. & $\quad$ & \multicolumn{2}{l}{\textbf{Initialisation:}} \vspace{0.05cm}\\
2. && \multicolumn{2}{l}{$\;\;\;\;\cdot$ estimate model parameters $p_i,R_{\theta_i},\Lambda_i$, $i = 1,\ldots,n$, by ML approach in Section \ref{sec:parameter_estimation}} \vspace{0.1cm}\\
3. && \multicolumn{2}{l}{$\;\;\;\;\cdot$ set $\:\delta = \tau \sigma \sqrt{n}$, $u^{(0)} = g$, $r^{(0)} = K u^{(0)} - g$, $t^{(0)} = D u^{(0)}$, $\rho_r^{(0)} = \rho_t^{(0)} = 0$, $k=0$} \vspace{0.2cm}\\
4. && \multicolumn{2}{l}{\textbf{while} not converging \textbf{do}:} \vspace{0.1cm}\\
5. && $\quad\;\;\bf{\cdot}$ \textbf{update primal variables:} &  \vspace{0.05cm} \\
6. && $\qquad\qquad\cdot$ compute $\:u^{(k+1)}$   &\quad by solving \eqref{eq:sub_u_sol}  \vspace{0.05cm} \\
7. && $\qquad\qquad\cdot$ compute $\:r^{(k+1)}$   &\quad  by applying  \eqref{eq:v_def}, \eqref{eq:sub_r_q2_q2_sol} \vspace{0.05cm} \\
8. && $\qquad\qquad\cdot$ compute $\:t^{(k+1)}$   &\quad see Section \ref{sec:po}  \vspace{0.05cm} \\
9. && $\quad\;\;\bf{\cdot}$ \textbf{update dual variables:} &  \vspace{0.05cm} \\
10. && $\qquad\qquad\cdot$ compute $\:\rho_r^{(k+1)}$, $\:\rho_t^{(k+1)}$ &\quad by applying \eqref{eq:PM_ADMM_lr}, \eqref{eq:PM_ADMM_lt} \vspace{0.1cm} \\
11. && $\qquad\qquad\cdot$ $k=k+1$ \vspace{0.1cm} \\
12. && \multicolumn{2}{l}{\textbf{end$\;$for}} \vspace{0.09cm} \\
13. && \multicolumn{2}{l}{$u^* = u^{(k+1)}$}
\end{tabular}
}
\end{algorithm}

\subsection{A non-convex proximal mapping solving \eqref{eq:sub_t_i}}
\label{sec:po}
In this section, we describe a novel result in multi-variate non-convex proximal calculus which is crucial to solve efficiently step 8 in the ADMM Algorithm \ref{alg:1}, i.e. the problem \eqref{eq:sub_t_i}. Such problem can be interpreted as the calculation of a non-convex proximal mapping, see \cite{Hare2009}. We then start recalling its definition.
\begin{definition}[proximal map for non-convex functions]
Let $f: \R^n \to \R$ be a proper, lower semi-continuous and possibly non-convex function and let $\beta>0$. The proximal map of $f$ with parameter $\beta$ is the set-valued function 
$\mathrm{prox}_{\beta f}: \R^n \rightrightarrows \R^n$ defined for any $q\in\R^n$ by:
\begin{equation}  
\label{eq:prox}
\mathrm{prox}_{\beta f}(q) \;{:=}\; \argmin_{t \in \R^n} 
\left\{
f(t) + \frac{\beta}{2} \, \left\| t - q \right\|_2^2 
\right\}.
\end{equation}
\label{def:prox}
\end{definition}
Note that under such definition the set $\mathrm{prox}_{\beta f}(q)$ is in general not a singleton. Furthermore, for some particular choices of $\beta>0$ it may also be empty.

\medskip

We present in the following the results concerned with the computation of the proximal map $\mathrm{prox}_{\beta f}$ in \eqref{eq:prox}, in the case when $f: \R^2 \to \R$ is the function 
\begin{equation}\label{eq:p1}
f(t):=
\left( t^T \! A \, t \right)^{p/2} \!\!, 
\quad t \in \R^2, 
\quad A \in \R^{2 \times 2} \; \mathrm{symmetric}\;\,\mathrm{positive}\;\,\mathrm{definite}, \quad p>0 \, .
\end{equation}
The ADMM substep \eqref{eq:sub_t_i} will then be a special instance of \eqref{eq:prox} under the choice of $f$ as above,  $\beta=\beta_t$, $t=t_i$, $A=R_{\theta_i}^{T}\!\Lambda_i^2\!R_{\theta_i}$, $p=p_i$, and $q=q_i^{(k+1)}$, for $i=1,\ldots,n$ and $k\geq 0$.
%
%
%
%
%


We now ensure that under the choice \eqref{eq:p1} above the minimisation problem \eqref{eq:prox} admits solutions. Then, assuming that $A$ has condition number $\kappa>1$ we show how the calculation of the proximal map can be reduced to the solution of a one-dimensional problem, whose form depends on the input $q$ and the matrix $A$. Note that the case $\kappa=1$ boils down to consider a scalar and diagonal matrix $A$, which simplifies the problem and for which the results discussed in \cite{tvpl2} can be used.

\begin{proposition}
\label{prop:prox_exist}
Under the choice \eqref{eq:p1}, the optimisation problem \eqref{eq:prox} admits at least one solution.
\end{proposition}
\begin{proof}
Under the choice \eqref{eq:p1}, both the terms in the objective function in \eqref{eq:prox} are continuous, bounded from below by zero and coercive over the entire domain $\R^2$. It clearly follows that the total objective function is continuous, bounded from below by zero and coercive, hence it admits at least one global minimiser.
\end{proof}

In the following, for $v,w\in\R^n$ we denote by $v \circ w$, $|v|$ and $\mathrm{sign}(v)$ the component-wise (or Hadamard) product between $v$ and $w$ and the component-wise absolute value and sign of $v$, respectively.
\begin{proposition}
\label{prop:prox_sol1}
Let $p,~\beta > 0$, $q \in \R^2$ and let $A\in\R^{2\times 2}$ be a symmetric positive definite matrix with condition number $\kappa>1$ and eigenvalue decomposition
\begin{equation}
A = V^T \Lambda V, \quad V^TV=VV^T = I_2, \quad 
\Lambda = \diag(\lambda_1,\lambda_2), \;\: \lambda_1 > \lambda_2 > 0.
\label{eq:A_EIG}
\end{equation}
Let us further define
\begin{equation}
\tilde{q} := V q, \quad 
s := \mathrm{sign}(\tilde{q}), \quad
\bar{q} := |\tilde{q}|, \quad
\overline{\beta} := \frac{\beta}{\lambda_2^{p/2}}, \quad 
\overline{\Lambda} := \diag(\kappa,1), \quad \kappa = \frac{\lambda_1}{\lambda_2} \,.
\label{eq:poi}
\end{equation}
Then, any solution $t^* \in \R^2$ of the problem
\begin{equation}  \label{eq:prox_full}
 t^*\in\argmin_{t \in \R^2} 
\left\{
F(t):=(t^T A t)^{p/2} + \frac{\beta}{2} \, \left\| t - q \right\|_2^2 
\right\}.
\end{equation} 
can be expressed as
\begin{equation}
t^* 
\;{=}\;\, 
V^T \left( s \:{\circ}\: z^*\right),
\quad
z^* \;{\in}\; \argmin_{z \in \mathcal{H}_1} \, H(z) \, , 
\label{eq:propprox1}
\end{equation}
where the objective function $H: \R^2 \to \R$ and the  feasible 
set $\mathcal{H}_1 \subset \R^2$ are defined by
\begin{equation}
H(z) :=
\left(
z^T \, \overline{\Lambda} \, z 
\right)^{p/2}
{+}\;
\frac{\overline{\beta}}{2} 
\left\| \, z - \bar{q} \, \right\|_2^2, \qquad
\mathcal{H}_1 := \mathcal{H} \;{\cap}\; \left( \big[ 0,\bar{q}_1 \big] \times \big[0,\bar{q}_2\big] \right),  \label{eq:def_arc_hyperb}
\end{equation}
with $\mathcal{H}$ being the rectangular hyperbola defined by
\begin{eqnarray}
\mathcal{H}{:=} & 
\left\{ \, 
\left(z_1,z_2\right) \in \R^2: 
\left(z_1-c_1\right) \left(z_2-c_2\right)= c_1 c_2,~\quad c_1 = -\frac{\bar{q}_1}{\kappa-1}, ~c_2 = \frac{\kappa \, \bar{q}_2}{\kappa-1}  \right\}.
\label{eq:Ah}
\end{eqnarray}
\end{proposition}

\begin{proof}
We start noticing that the matrix $\Lambda$ in \eqref{eq:A_EIG} can be factorised as $\Lambda=\lambda_2 \overline{\Lambda}$, where $\overline{\Lambda}$ is  defined in \eqref{eq:poi}.
By substituting such factorisation into \eqref{eq:A_EIG}, we can reformulate problem \eqref{eq:prox_full} as:
\begin{equation}
t^* \;{\in}\; \argmin_{t \in \R^2}
\left\{ 
\lambda_2^{p/2}
\left(
t^T V^T \, \overline{\Lambda} \, V t
\right)^{p/2}
\;{+}\;
\frac{\beta}{2} \, \left\| t - q \right\|_2^2
\right\} \, .
\label{eq:proxprob1_bis}
\end{equation}
After introducing the bijective linear change of variable
\begin{equation}
y := V t \quad{\Longleftrightarrow}\quad t = V^T y,
\label{eq:xtoy}
\end{equation}
we have that problem \eqref{eq:proxprob1_bis} can be equivalently expressed as
\begin{eqnarray}
t^* &\;{=}\;\,& V^T y^* \, , \label{eq:proxprob1_ter} \\ 
y^* &\;{\in}\;\,&  \argmin_{y \in \R^2}
\left\{ \, 
G(y) 
:=
\left(
y^T \, \overline{\Lambda} \, y 
\right)^{p/2}
{+}\;
\frac{\overline{\beta}}{2} \left\| y - \tilde{q} \right\|_2^2
\right\}, \label{eq:proxprob2_ter} 
\label{eq:proxprob3_ter}
\end{eqnarray}
where $\overline{\beta}$ and $\tilde{q}$ are defined in \eqref{eq:poi}.

If $\tilde{q}_1 = \tilde{q}_2 = 0$ then one can trivially show that clearly $y^* = (0,0) \Longrightarrow t^* = (0,0)$.
We can then assume that $\tilde{q} \in \R^2 \setminus \{0\}$ and exploit symmetries of the function $G$ in \eqref{eq:proxprob2_ter} to restrict the optimisation problem to the case where $\tilde{q}$ lies in the first quadrant only.
First, we notice that, for any given $a \in \R$ and $b \in \R \setminus \{0\}$, 
we have
\begin{eqnarray}
a^2 &\;{=}\;\,& 
\left( \mathrm{sign}(b) \right)^2 a^2 
= \left( \mathrm{sign}(b) \, a \right)^2, 
\label{eq:ss1} \\
\left( a - b \right)^2 
&\;{=}\;\,& 
\left( a - \mathrm{sign}(b) |b| \right)^2 =
\left( \mathrm{sign}(b) \left( \frac{a}{\mathrm{sign}(b)} - |b| \right) \right)^2 
\nonumber \\
&\;{=}\;\,& 
\left( \mathrm{sign}(b) \right)^2 \left( \mathrm{sign}(b) a - |b| \right)^2 
= \left( \mathrm{sign}(b) a - |b| \right)^2 .
\label{eq:ss2}
\end{eqnarray}
By now recalling definitions of function $G$ in \eqref{eq:proxprob2_ter} and of matrix $\overline{\Lambda}$ in \eqref{eq:poi}, and then using \eqref{eq:ss1}-\eqref{eq:ss2}, we can write
\begin{eqnarray}
&G(y)&
\;{=}\; 
\left( \kappa y_1^2 + y_2^2 \right)^{p/2} 
+ 
\frac{\overline{\beta}}{2} 
\left( 
\left(y_1 - \tilde{q}_1 \right)^2 + 
\left(y_2 - \tilde{q}_2 \right)^2 \right)
\nonumber \\
&&
\;{=}\;  
\left( 
\kappa 
\left( \mathrm{sign}(\tilde{q}_1)y_1 \right)^2 
+ 
\left( \mathrm{sign}(\tilde{q}_2)y_2 \right)^2 
\right)^{p/2}
+ 
\frac{\overline{\beta}}{2} 
\left( 
\left(\mathrm{sign}(\tilde{q}_1) y_1 - |\tilde{q}_1| \right)^2 
+ 
\left(\mathrm{sign}(\tilde{q}_2) y_2 - |\tilde{q}_2|^2 \right) \right) .
\nonumber 
\end{eqnarray}
By setting $S:= \mathrm{diag}
\left( 
\mathrm{sign}\left(\tilde{q}_1\right),
\mathrm{sign}\left(\tilde{q}_2\right)
\right)$ we can now set 
\begin{equation}
z := S y \quad{\Longleftrightarrow}\quad y = S^{-1} z,
\end{equation}
which is a linear bijective change of variable since $\tilde{q}_1,\tilde{q}_2 \in \R \setminus \{0\} \;{\Longrightarrow}\; \mathrm{sign}\left(\tilde{q}_1\right),\mathrm{sign}\left(\tilde{q}_2\right) \in \{-1,1\}$. Recalling the definition of $s$ and $\overline{q}$ in \eqref{eq:poi}, we thus get that the optimisation problem \eqref{eq:proxprob2_ter} is equivalent to
\begin{eqnarray}
y^* &\;{=}\;\,& s \:{\circ}\: z^* \, , \label{eq:proxprob1_qua} \\ 
z^* &\;{\in}\;\,&  \argmin_{z \in \R^2}
\left\{ \, 
H(z) 
:=
\left( 
z^T \, \overline{\Lambda} \, z 
\right)^{p/2}
{+}\;
\frac{\overline{\beta}}{2} \left\| \, z - \overline{q} \, \right\|_2^2
\right\} \, , \label{eq:proxprob2_qua} 
\end{eqnarray}
where the vector $\overline{q} = (|\tilde{q}_1|,|\tilde{q}_1|)$ now lies in the first (open) quadrant $(0,+\infty)^2$.

We now prove that the solutions $z^*$ in \eqref{eq:proxprob2_qua} belong to the arc of hyperbola $\mathcal{H}_1$ defined in \eqref{eq:def_arc_hyperb}. To this aim, we consider the following one-parameter family of ellipses depending on a parameter $R>0$:
\begin{eqnarray}
\mathcal{E}_R 
&\;{:=}\;\,& 
\!\left\{ 
\left(z_1,z_2\right) \in \R^2\!: \;\, 
z^T \overline{\Lambda} z \;{=}\; R^2 
\right\}=
\!\left\{ 
\left(z_1,z_2\right) \in \R^2\!: \;\, 
\kappa \, z_1^2 + z_2^2 \;{=}\; R^2 
\right\} \nonumber \\
&\,\;{=}\;&
\!\left\{ 
\left(z_1,z_2\right) \in \R^2\!: \;\, 
z_1 = z_1(\theta;R) \!= \frac{R}{\sqrt{\kappa}} \, \cos \theta, \; z_2 = z_2(\theta;R) = R \sin \theta,~ \theta \in [0,2\pi[  
\right\}
\label{eq:ELL}
\end{eqnarray}
and, as a start, we show that the minimisers of the restriction of the function $H$ in \eqref{eq:proxprob2_qua} to any ellipse $\mathcal{E}_R$ in \eqref{eq:ELL} lie on the hyperbola $\mathcal{H}$ in \eqref{eq:Ah}. 
In Figure \ref{fig:proxa} we show the hyperbola $\mathcal{H}$ (magenta solid line) with its 
two orthogonal asymptotes, the arc $\mathcal{H}_1$ defined in \eqref{eq:def_arc_hyperb} (red solid thick line) and one ellipse $\mathcal{E}_R$ (blue dashed line) as in \eqref{eq:ELL}.
\begin{figure}[tbh]
	\center
	\includegraphics[scale = 0.6]{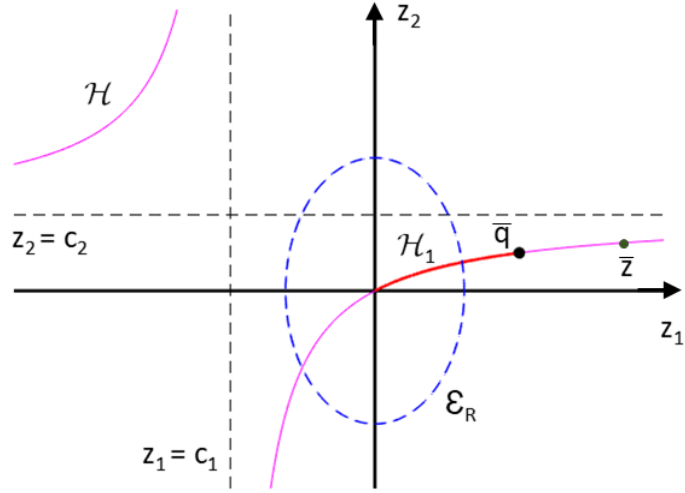}
    \caption{Graphical representation for the bivariate minimisation problem \eqref{eq:proxprob2_qua}.}
    \label{fig:proxa}
\end{figure}

Let us observe first that when restricted to an ellipse  $\mathcal{E}_R$ of the form in \eqref{eq:ELL}, the objective function $H$ depends only on $\theta$ ($R$ can be regarded as a fixed parameter). The restriction $H_R: \R \to \R$ takes then the following form
\begin{equation}
H_R(\theta;R) 
=
R^p + \frac{\overline{\beta}}{2} \, 
\left( 
\left( \frac{R}{\sqrt{\kappa}} \, \cos \theta - \overline{q}_1 \right)^2
+
\big( R \sin \theta - \overline{q}_2 \big)^2
\right) \, . 
\label{eq:h_R}
\end{equation}
For any $R > 0$, the function $H_R$ above is clearly periodic with period $2\pi$, bounded (from below and above) and infinitely many times differentiable in $\theta$, hence the minimisers of $H_R$ can be sought for among its stationary points in the interval $[0,2 \pi)$. 
The first-order derivative of $H_R$ is as follows:
\begin{eqnarray}
H_R'(\theta;R) &\,\;{=}\;\,& 
\overline{\beta} \, 
\left( 
-\frac{R}{\sqrt{\kappa}} \sin \theta \left( \frac{R}{\sqrt{\kappa}} \cos \theta - \overline{q}_1 \right)
+ R \cos \theta \left( R \sin \theta - \overline{q}_2 \right)
\right) 
\nonumber \\
%
%
%
%
%
%
%
%
%
%
&\,\;{=}\;\,& 
\overline{\beta} \,\, \frac{\kappa-1}{\sqrt{\kappa}} \,\,\,
\big(  
\left( z_1(\theta;R) - c_1 \right) \left( z_2(\theta;R) - c_2 \right) - c_1 c_2  
\, \big)
\label{eq:uu}
\end{eqnarray}
where \eqref{eq:uu} follows after some simple algebraic manipulations from the parametrisation in \eqref{eq:ELL}, with $c_1, c_2$ constants defined in \eqref{eq:Ah}.
Since $\beta > 0$, $\kappa > 1$ by assumption, the scalar quantity $\overline{\beta} \,(\kappa-1)/\sqrt{\kappa}$ in \eqref{eq:uu} is positive, hence we have
\begin{equation}
H_R'(\theta;R) \;{=}\; 0 \, (>0,<0) \:\;{\Longleftrightarrow}\;\:
\left( z_1(\theta;R) - c_1 \right) \left( z_2(\theta;R) - c_2 \right) - c_1 c_2 
\;{=}\; 0 \, (>0,<0) \, .
\label{eq:HRs}
\end{equation}
It thus follows that, for any fixed $R > 0$ (that is, for any ellipse $\mathcal{E}_R$ in \eqref{eq:ELL}), any stationary point $z(\theta_R^*:R)$ of $H_R$ satisfies
%
%
%
\begin{equation}
\big( z_1\left(\theta_R^*;R\right) , \, z_2\left(\theta_R^*;R\right) \big) 
\;{\in}\;\, \mathcal{E}_R \;{\cap}\; \mathcal{H} \, ,
\end{equation}
i.e. it belongs to the set of intersection points between the ellipse $\mathcal{E}_R$ and the hyperbola $\mathcal{H}$ (see the two intersection points in Figure \ref{fig:proxa}). It also follows from \eqref{eq:HRs} that the intersection point in the first quadrant is the global minimiser for $H_R$, whereas the one in the third quadrant is the global maximiser. 
Since previous considerations hold true for any ellipse $\mathcal{E}_R$, then any global minimiser $z^*$ of the unrestricted objective function $H$ in \eqref{eq:proxprob2_qua} must belong to the restriction of the hyperbola $\mathcal{H}$ in \eqref{eq:Ah} to the first quadrant. 

Finally, it is easy to further shrink the locus of potential global minimisers $z^*$ to the arc $\mathcal{H}_1$ defined in \eqref{eq:def_arc_hyperb}. 
Let us argue by contradiction and suppose there exists a global minimiser $\bar{z}$ belonging to the restriction of the hyperbola $\mathcal{H}$ to the first quadrant but not to $\mathcal{H}_1$ - see Figure \ref{fig:proxa}. We have:
\begin{equation}
H(\bar{z})-H(\bar{q}) 
=
\underbrace{
\left( 
\bar{z}^T \, \overline{\Lambda} \, \bar{z} 
\right)^{p/2}
{-}\; 
\left( 
\bar{q}^T \, \overline{\Lambda} \, \bar{q} 
\right)^{p/2}
}_{>0}
{+}\;
\underbrace{
\frac{\overline{\beta}}{2} 
\left( 
\left\| \, \bar{z} - \overline{q} \, \right\|_2^2
{-}\; 
\left\| \, \overline{q} - \overline{q} \, \right\|_2^2
\right)
}_{>0}
\;{>}\; 0,
\end{equation}
whence $\overline{z}$ can not be a global minimiser for the function $H$.
\end{proof}

In the following corollary we exploit and complete the  results in previous Proposition \ref{prop:prox_sol1} by showing how the bivariate minimisation problem in \eqref{eq:propprox1} can be reduced to an equivalent univariate problem.
\begin{corollary}
The minimisers $z^* \in \R^2$ in \eqref{eq:propprox1} can be obtained as follows:
\begin{equation}
z^* = 
\left( 
z_1^*,
c_2 \left( \frac{z_1^*}{z_1^* - c_1} \right) 
\right) \, ,
\end{equation}
with $c_1$, $c_2 \in \R$ defined in \eqref{eq:Ah} and $z_1^* \in \R$ the solution(s) of the following $1$-dimensional constrained minimisation problem:
\begin{equation}
z_1^* \,\;{\in}\;\, \argmin_{\xi \;{\in}\; [0,\bar{q}_1]}
\, \left\{ \,
h(\xi) 
\;{:=}\; \left( h_1(\xi) \right)^{p/2}
\:{+}\;\,
\frac{\overline{\beta}}{2} \,
h_1(\xi)
\,\;{-}\;\,
\frac{\overline{\beta}}{2} \,
h_2(\xi)
\, \right\} \, ,
\end{equation}
\begin{equation}
h_1(\xi) = \xi^2 \left( \kappa + \frac{c_2^2}{(\xi-c_1)^2} \right), \quad
h_2(\xi) = \xi \, (\kappa-1) \left( \xi - 2 c_1 + 2 \, \frac{c_2^2}{\kappa(\xi-c_1)} \right) \, .
\end{equation}
\end{corollary}

\begin{proof}
The proof is immediate by deriving the expression of $z_2$ as a function of $z_1$ from the definition of the hyperbola $\mathcal{H}$ in \eqref{eq:Ah}, then substituting this expression in the objective function $H$ in \eqref{eq:def_arc_hyperb} and, finally, carrying out some algebraic manipulations.
\end{proof}

\section{Parameters estimation results}  \label{sec:parameter_estimation_results}
In this section, an extensive evaluation on the accuracy of the ML estimation procedure described in Section \ref{sec:parameter_estimation} is carried out. 
In order to assess the quality of the estimation, we introduce in the following some useful statistical notions.
\begin{definition}  \label{def:errors}
Let $\omega>0$ be an unknown parameter of a fixed probability distribution $p_\omega$ and for $\ell>0$ let $\omega_j,~ j=1,\ldots,\ell$ be estimates of $\omega$ obtained by a given estimation procedure. The \textbf{sample estimator} $\hat{\omega}$ of $\omega$ is defined as the average:
$$
\hat{\omega}:=\frac{\sum_{j=1}^\ell \omega_j}{\ell}.
$$
We can then define the \textbf{relative bias} $\mathcal{B}_{\hat{\omega}}$, the \textbf{empirical variance} $\mathcal{V}_{\hat{\omega}}$ and the \textbf{relative root mean square error} $rmse_{\hat{\omega}}$ of the estimator $\hat{\omega}$ as:
\begin{equation}
\mathcal{B}_{\hat{\omega}}: = \frac{\mathbb{E}(\hat{\omega} - \omega )}{\omega},\quad \mathcal{V}_{\hat{\omega}} := \frac{1}{\ell - 1} \,\sum_{j=1}^{\ell} \,(\;\omega_{j} - \hat{\omega}\;)^2,\quad rmse_{\hat{\omega}} := \frac{\sqrt{\,\mathcal{V}_{\hat{\omega}}\,+\,\mathcal{B}_{\hat{\omega}}^{2}}}{\omega}.
\end{equation}
\end{definition}

In the following, the accuracy and the precision of the estimator is evaluated by analysing its performance on the estimation of the parameters $(p,{e^{(1)}},\theta)$. As discussed in Section \ref{sec:compact}, parameters ${e^{(1)}},\theta$ can be derived from $\varrho,\phi$ and from \eqref{def:lambda_e1}, we recall that ${e^{(1)}} = \bigg(\frac{1}{\lambda^{(1)}}\bigg)^{2}.
$
In addition, we also consider how the quality of the estimation of $(p,{e^{(1)}},\theta)$ affects the estimation of the scale parameter $m$, which is computed directly via the formula \eqref{eq:expr_m} as a non-linear function of  $(p,\phi,\varrho)$ or, equivalently, of $(p,{e^{(1)}},\theta)$, as well as of the samples. The non-linearity may affect the accuracy of its estimation.

\subsection{Parameter estimation: accuracy and precision} \label{sec:est}

We now perform some tests assessing the accuracy and the precision of the ML estimation procedure proposed in Section \ref{sec:parameter_estimation} in terms of the quantities defined above. As a first test we compare the results obtained by applying the ML procedure to estimate a BGGD of parameters $(\bar{p},\bar{e}^{(1)},\bar{\theta},\bar{m})=(1,1.4,\ang{45},0.3)$. We run our tests for an increasing number $N\in \left\{10, 10^2, 10^3, 10^4, 10^5, 10^6\right\}$ of samples drawn from the distribution. For each value of $N$, the estimation procedure is run $\ell = 200$ times. For any  $j=1,\dots,\ell$ we estimate the parameter triple $(p^*,\phi^*,\varrho^*)_j$ and consider the corresponding estimators of the true parameters as  defined in Definition \ref{def:errors}. 
The results are shown in Figures \ref{fig:relbias1} - \ref{fig:relrmse1}. 

\begin{figure}[tbh]
	\center
	\begin{tabular}{cc}
		\includegraphics[scale = 0.4]{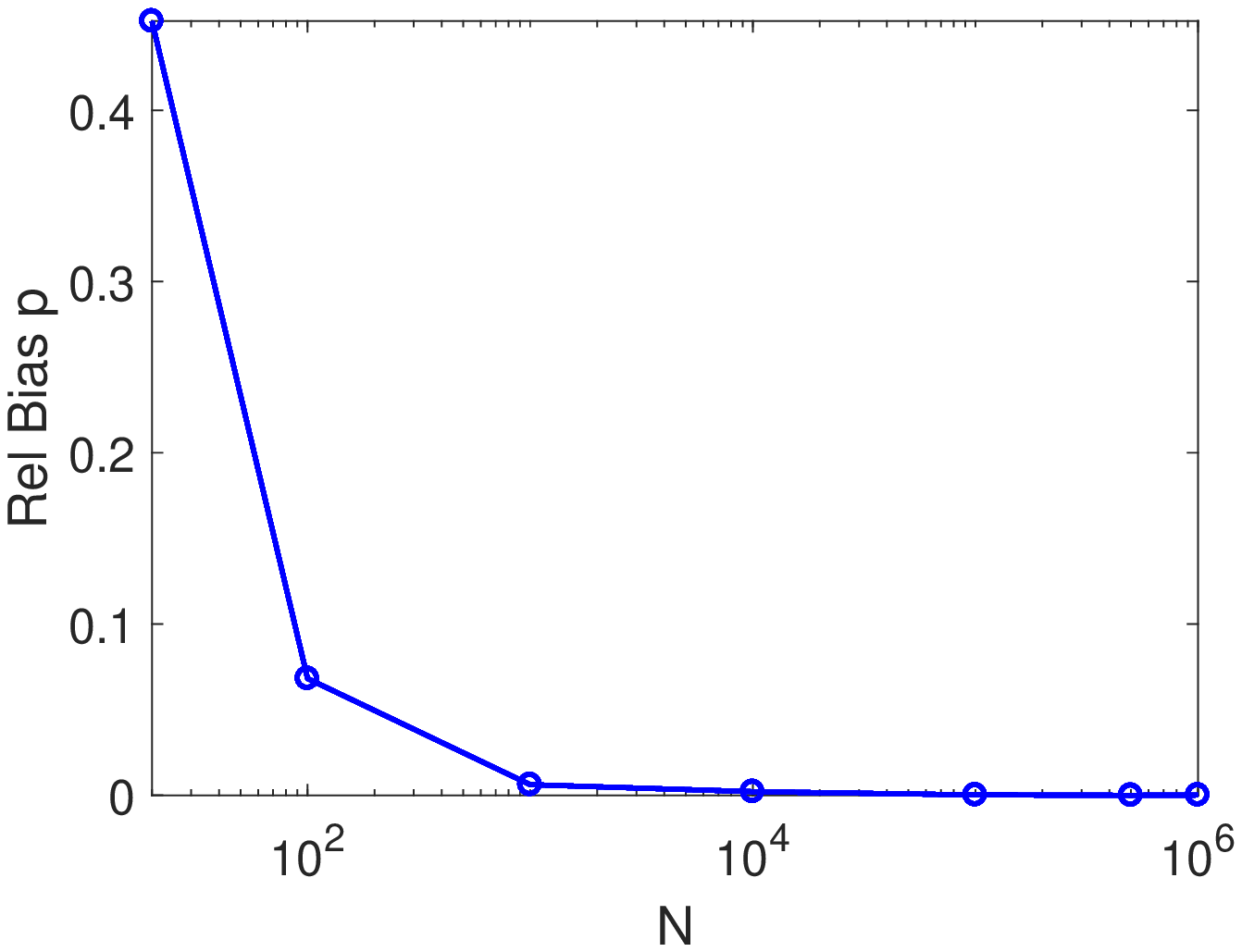} &
		\includegraphics[scale = 0.4]{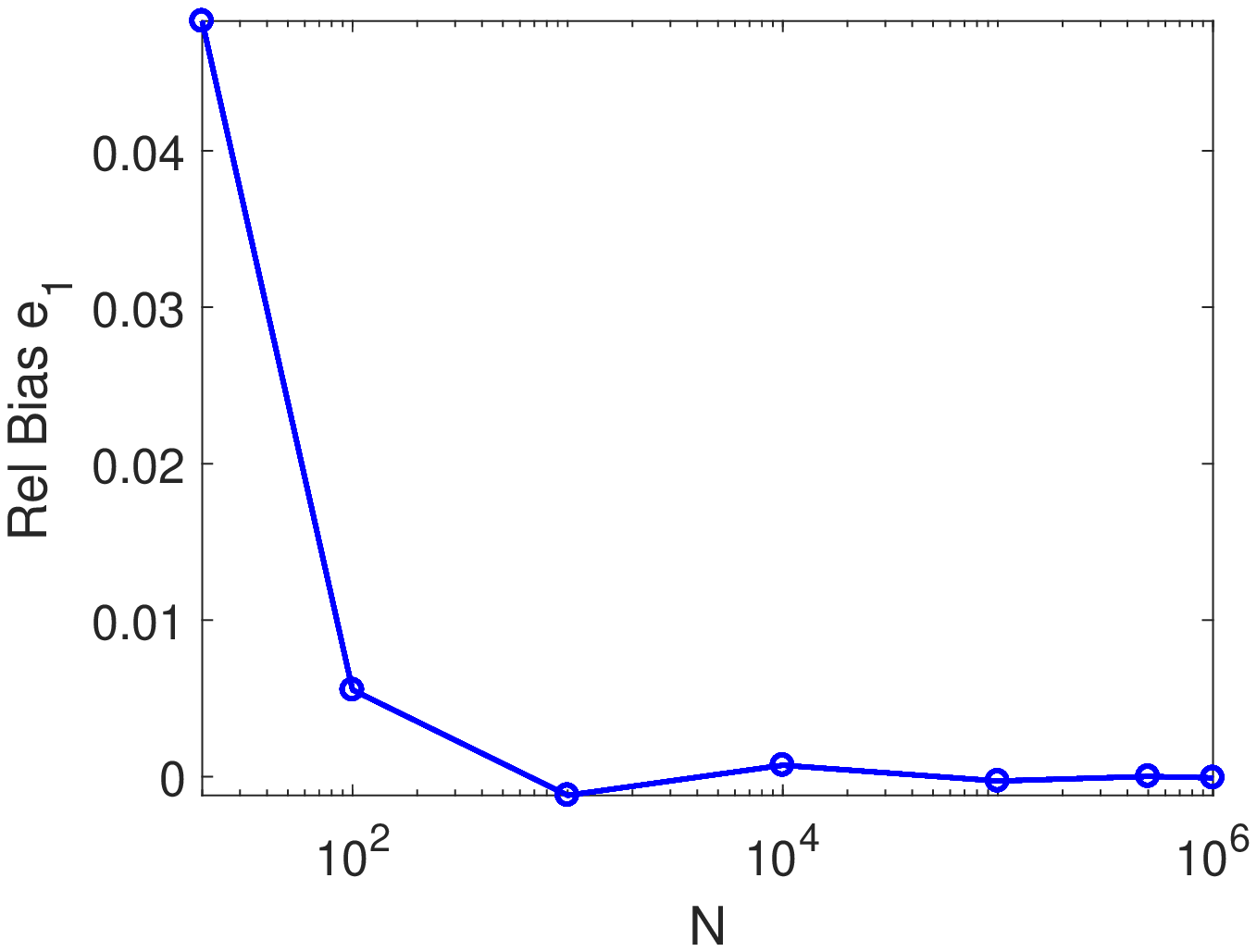}\\
		\includegraphics[scale = 0.4]{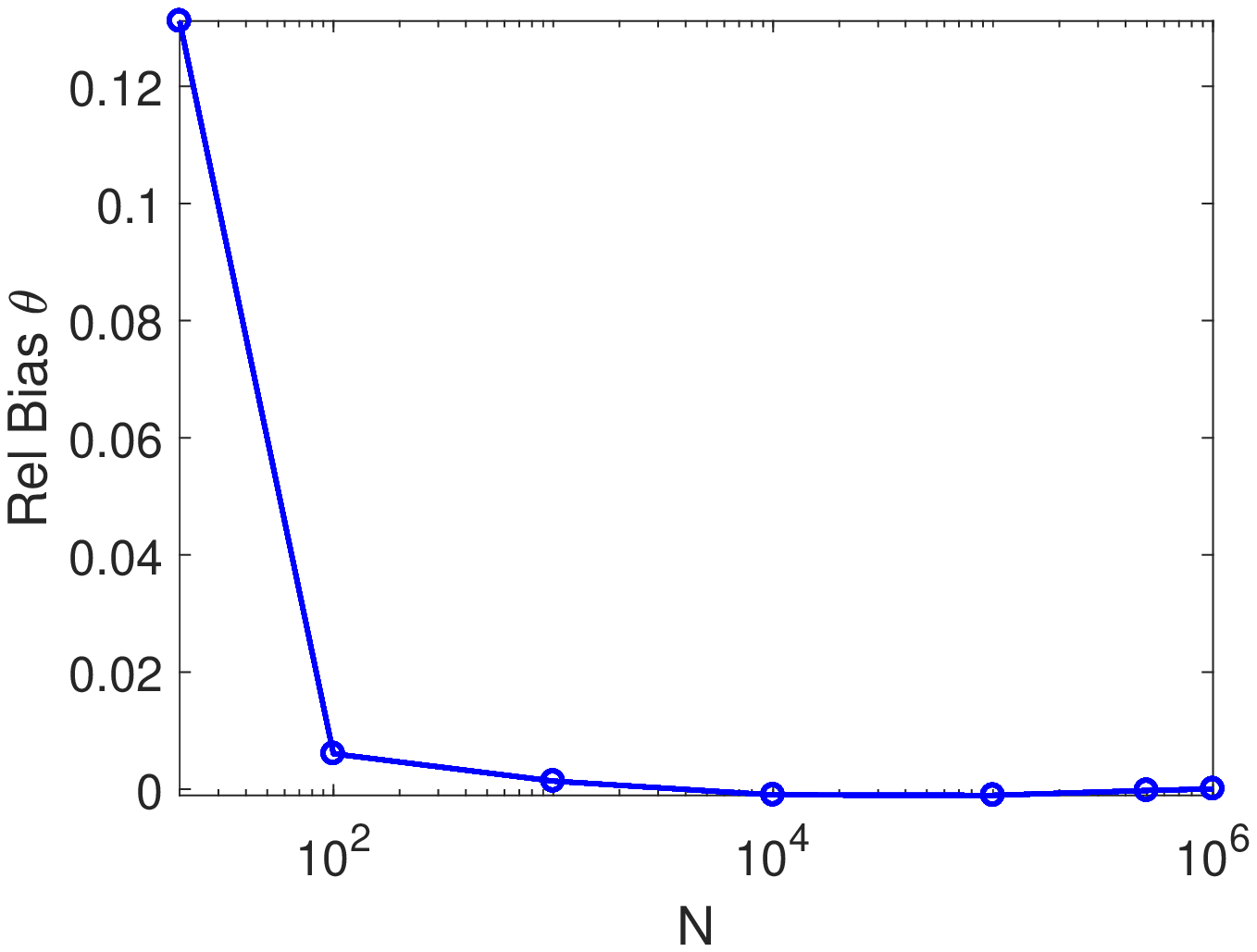}&
        \includegraphics[scale = 0.4]{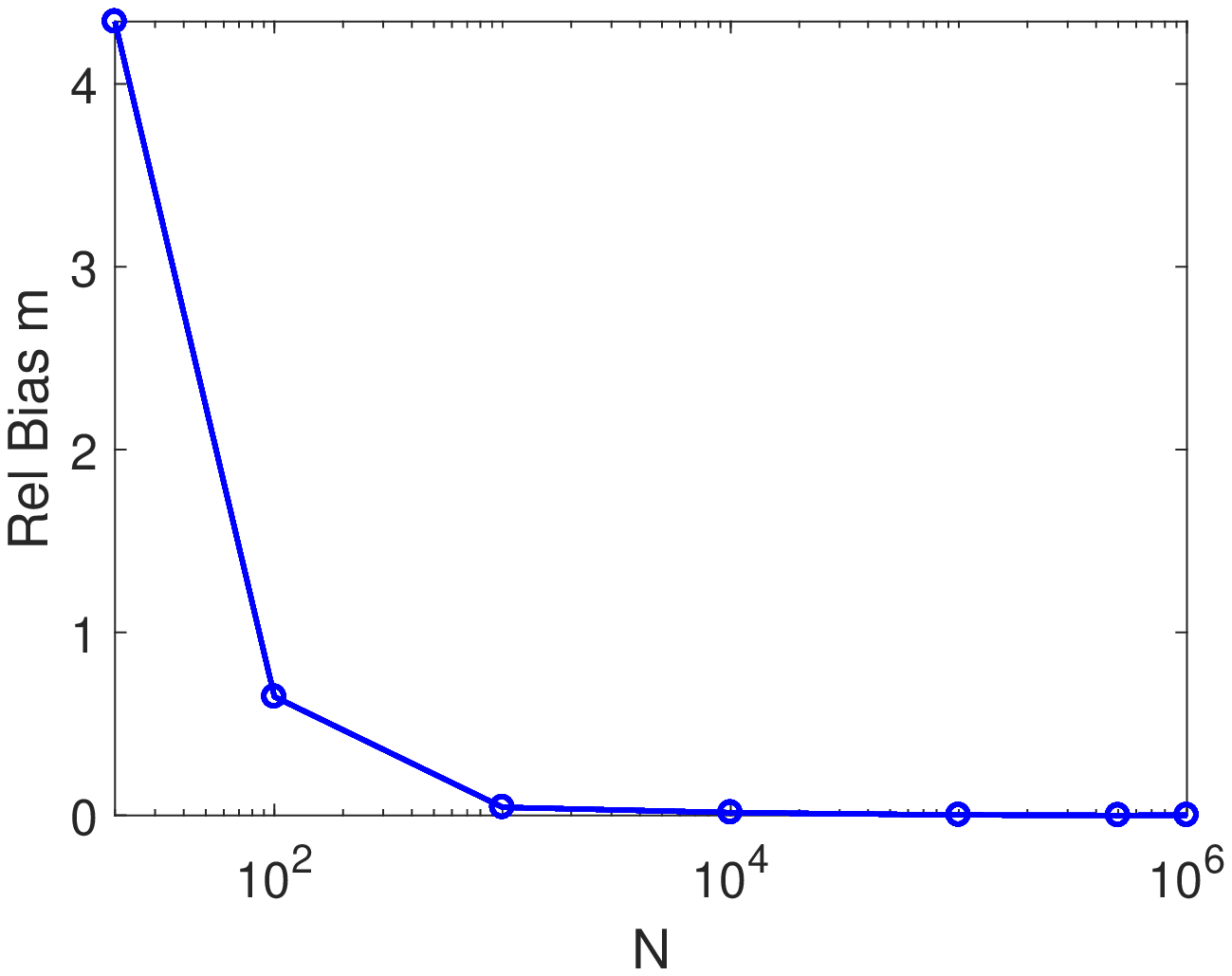}\\
	\end{tabular}
    \caption{Plots of relative bias for estimated $(p^*, {{e^{(1)}}}^*, \theta^*, m^*)$ in semi-logarithmic scale on $x$-axis.}
    	\label{fig:relbias1}
\end{figure}

\begin{figure}[tbh]
	\center
	\begin{tabular}{cc}
		\includegraphics[scale = 0.4]{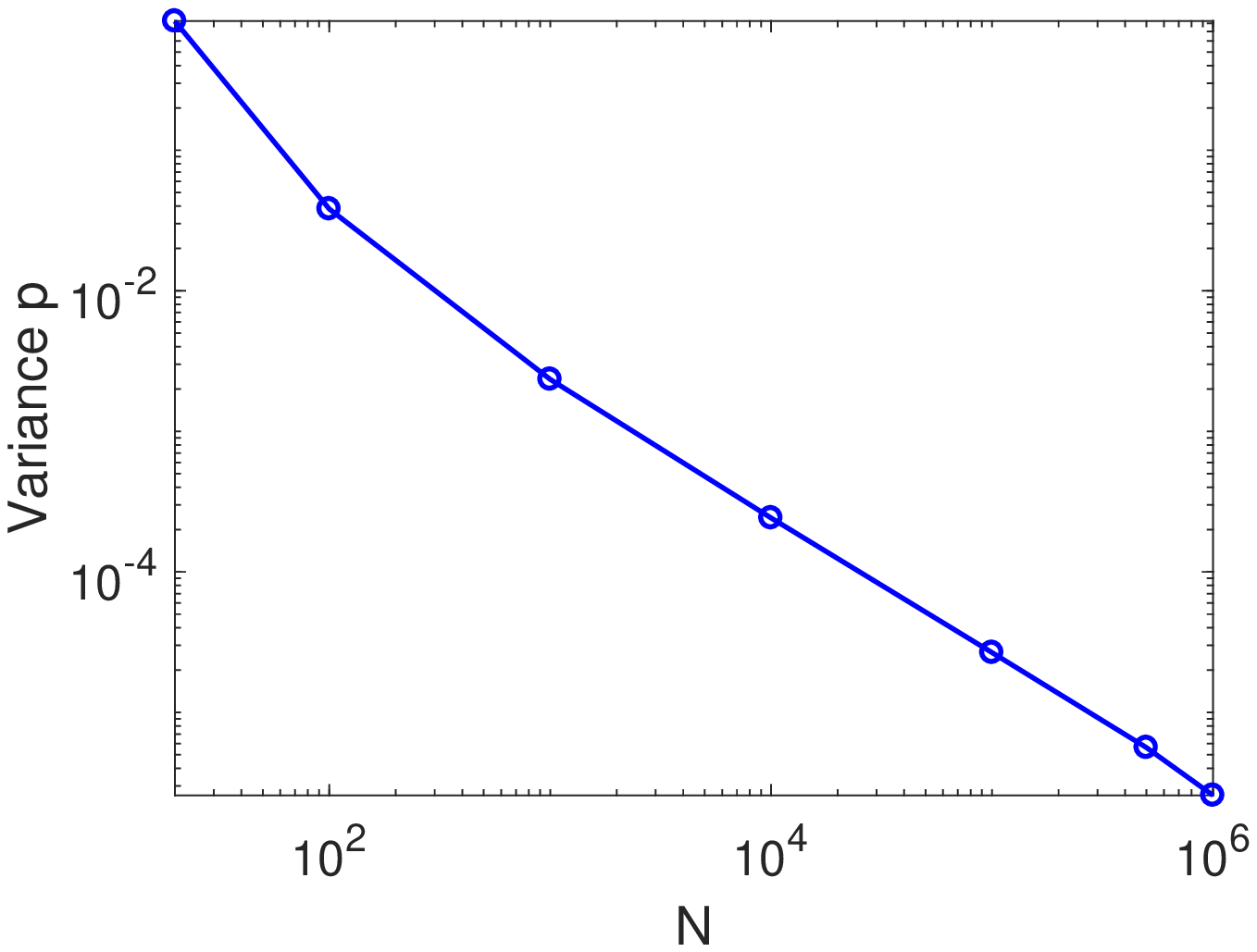} &
		\includegraphics[scale = 0.4]{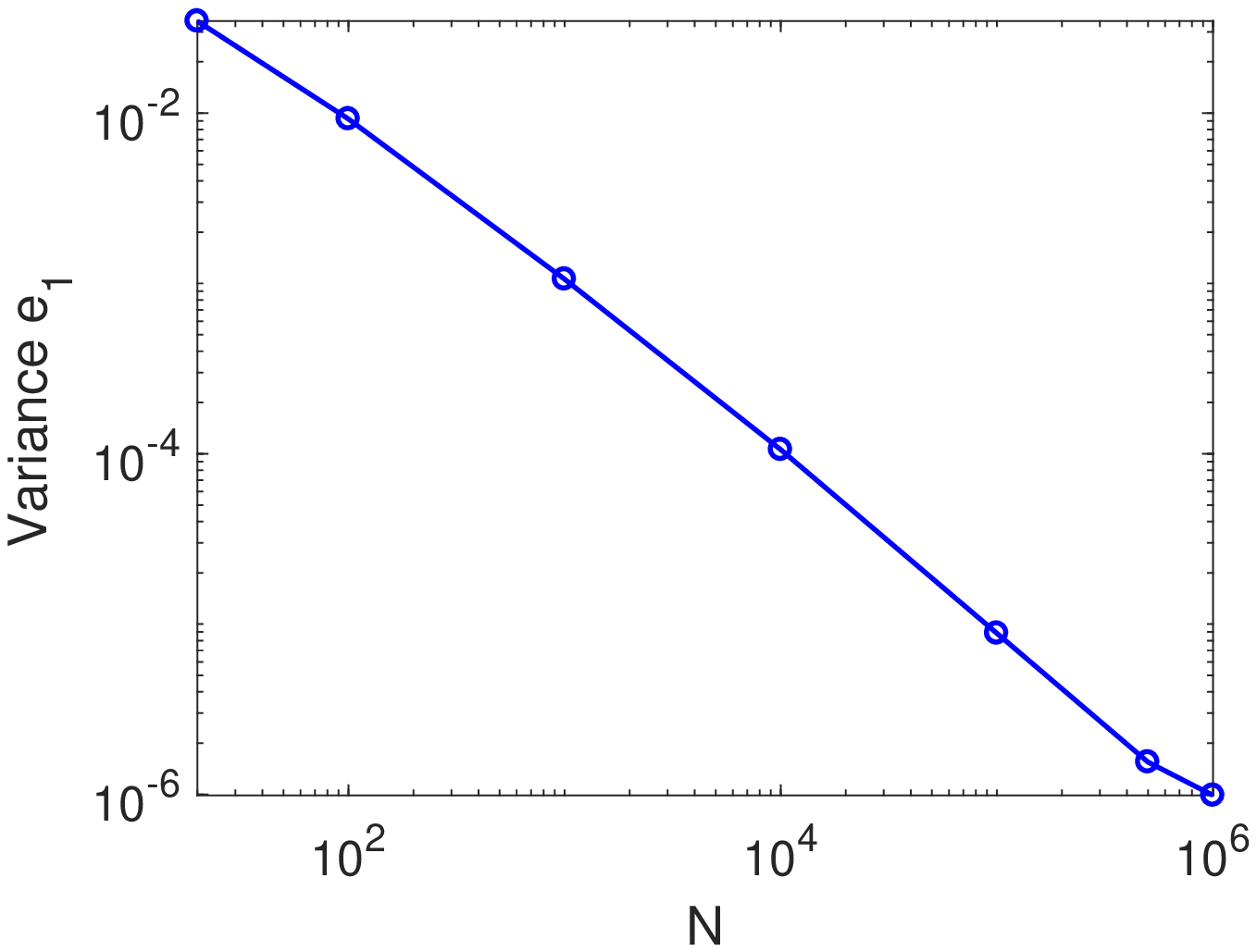}\\
		\includegraphics[scale = 0.4]{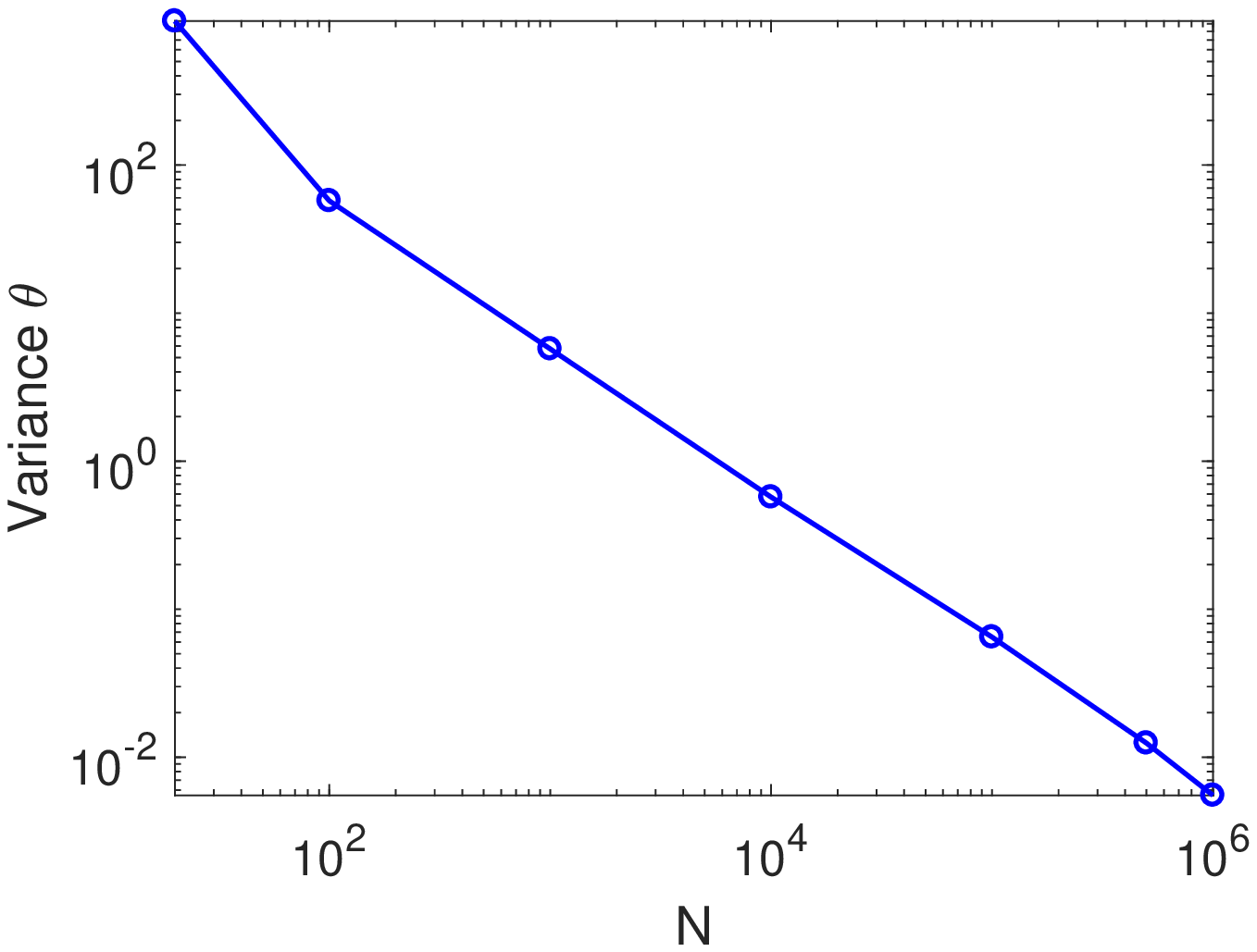}&
        \includegraphics[scale = 0.4]{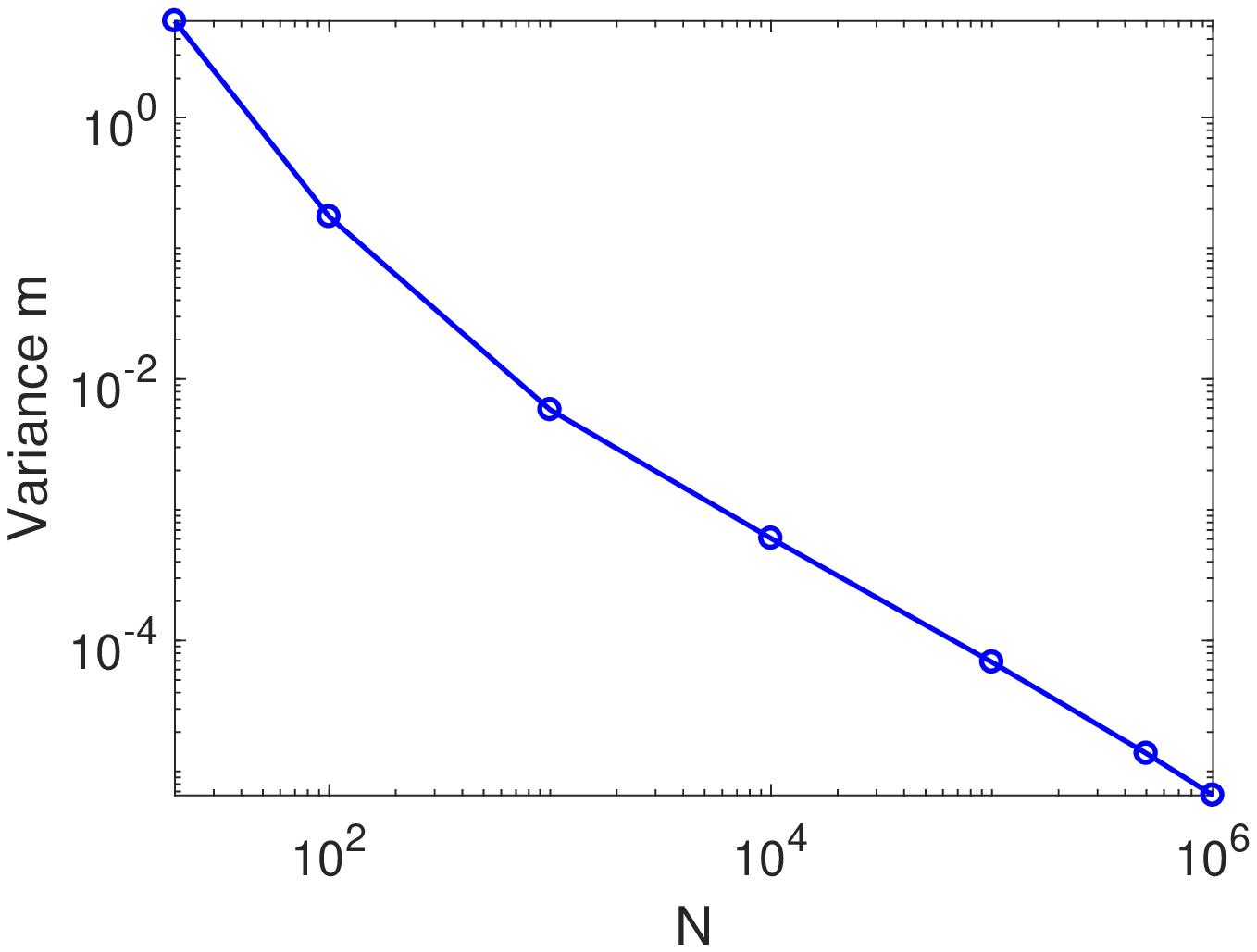}\\
	\end{tabular}
    \caption{Plot of the empirical variance for estimated $(p^*, {{e^{(1)}}}^*, \theta^*, m^*)$ in semi-logarithmic scale on $x$-axis.}
    	\label{fig:var1}

	\end{figure}

\begin{figure}[tbh]
	\center
	\begin{tabular}{cc}
		\includegraphics[scale = 0.4]{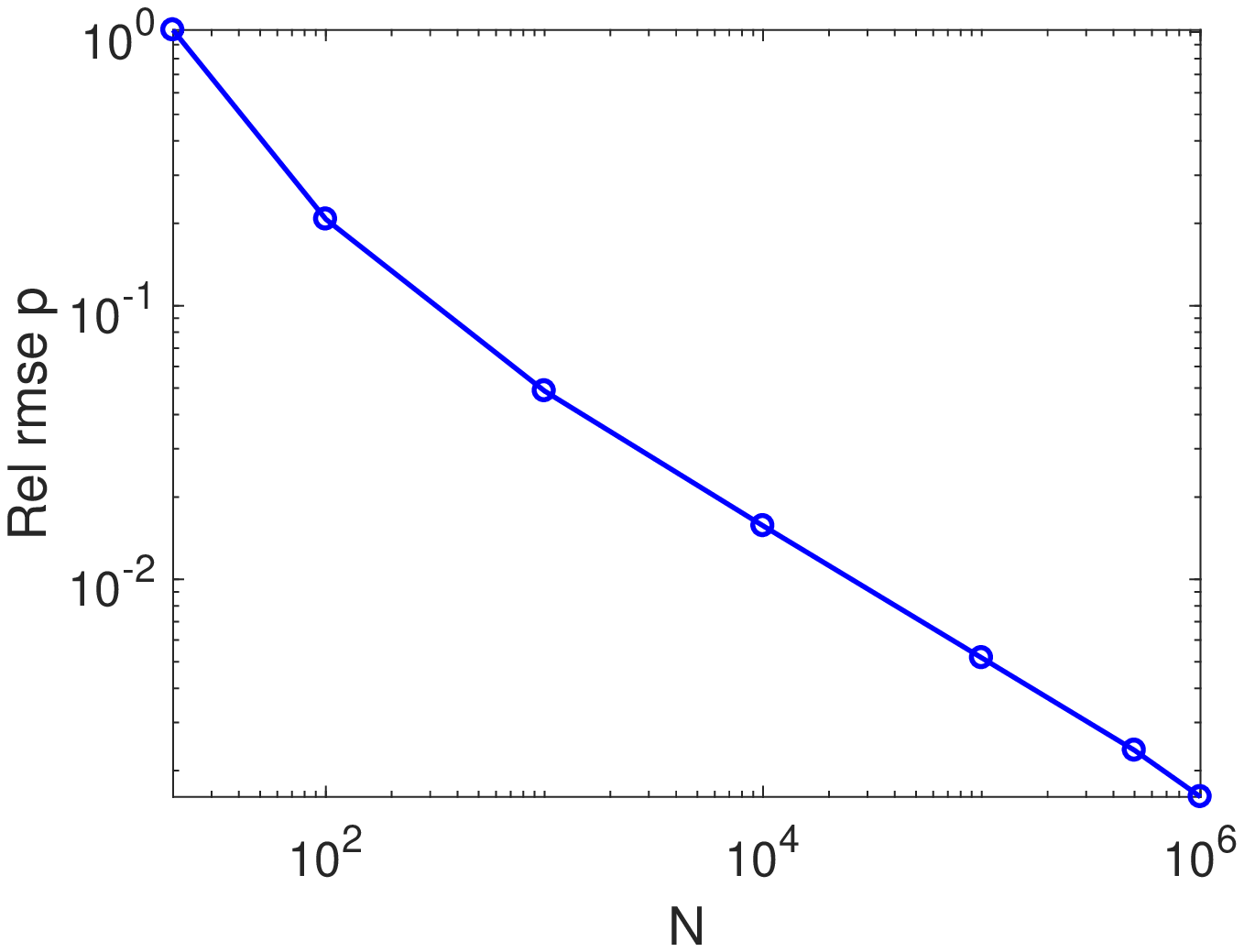} &
		\includegraphics[scale = 0.4]{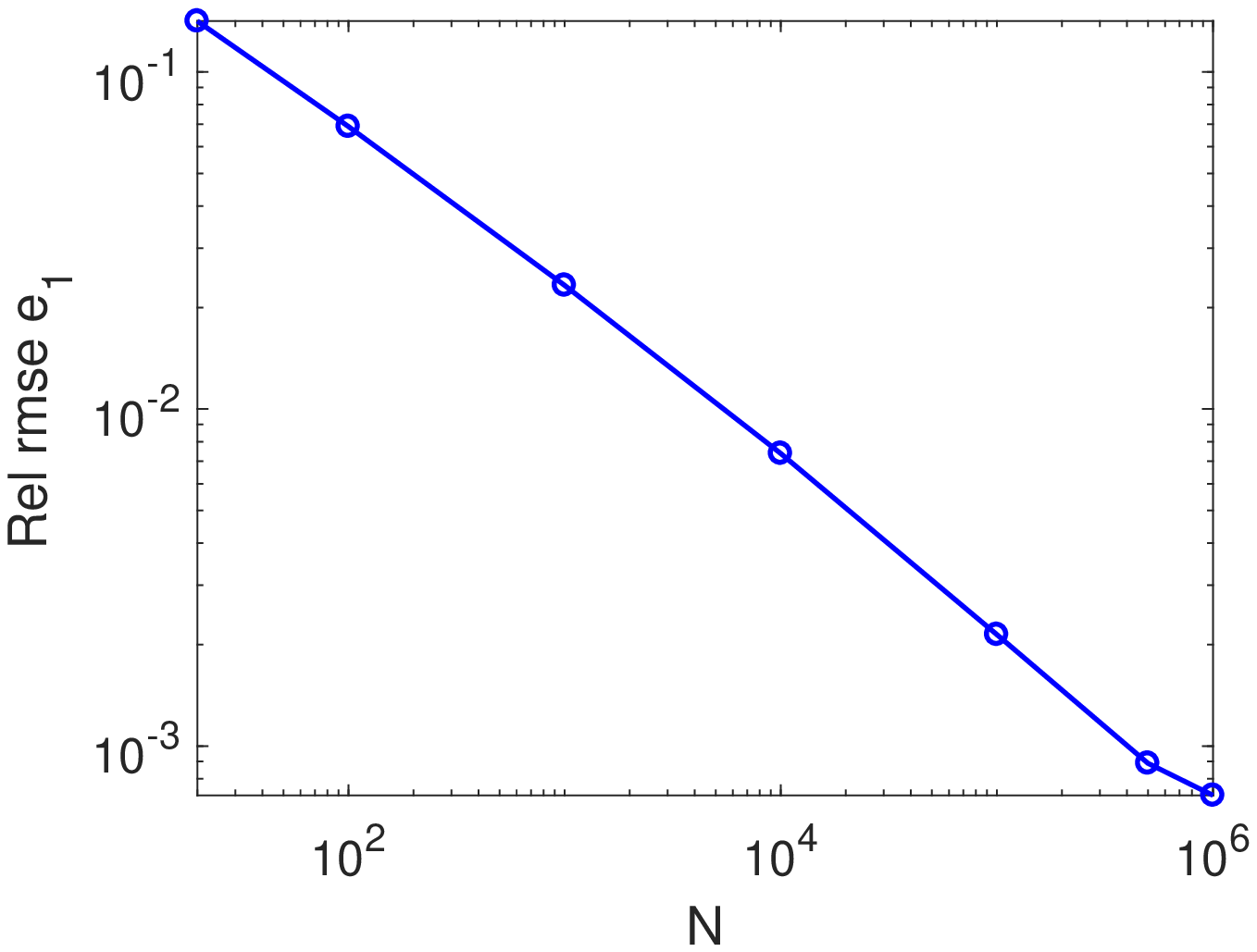}\\
		\includegraphics[scale = 0.4]{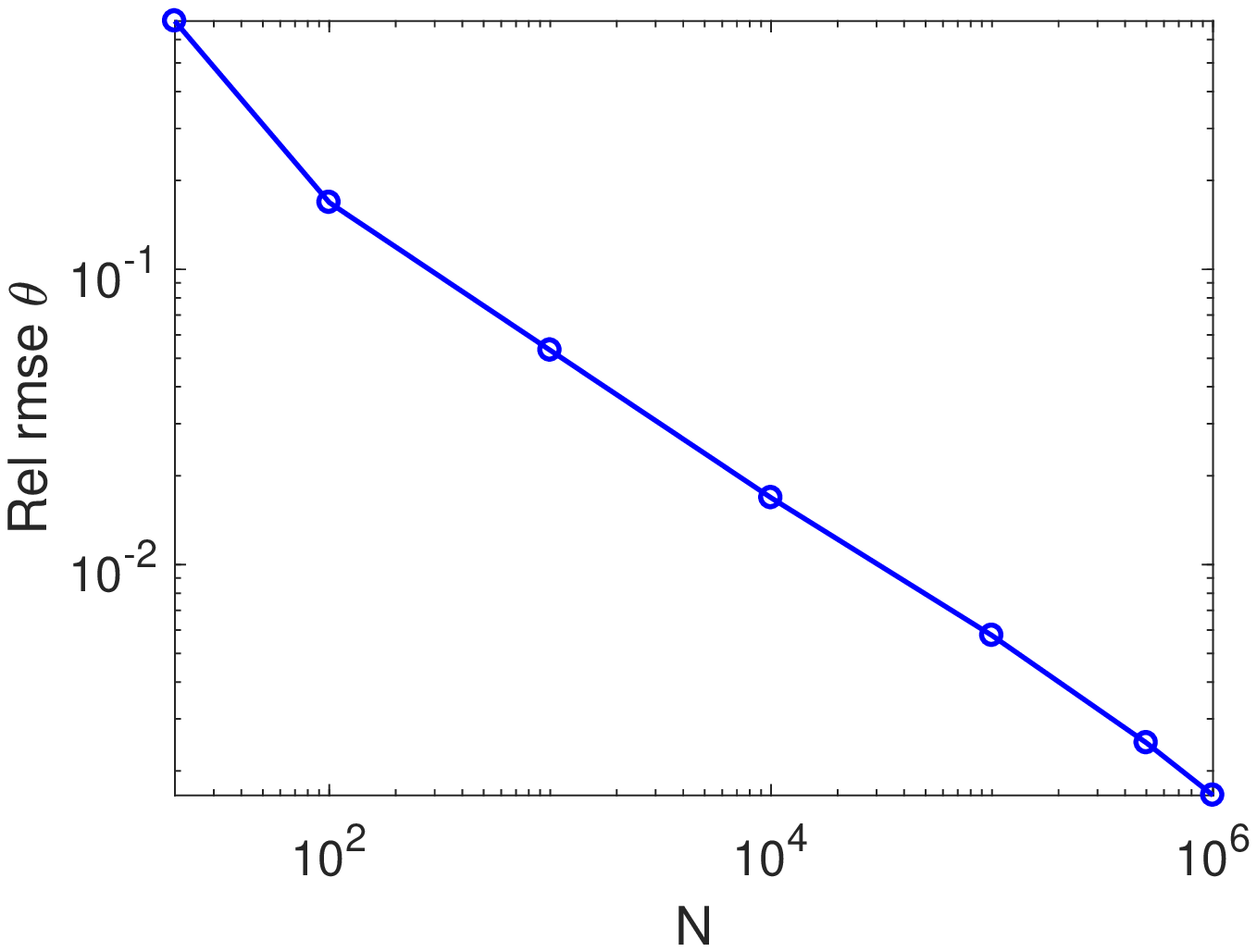}&
        \includegraphics[scale = 0.4]{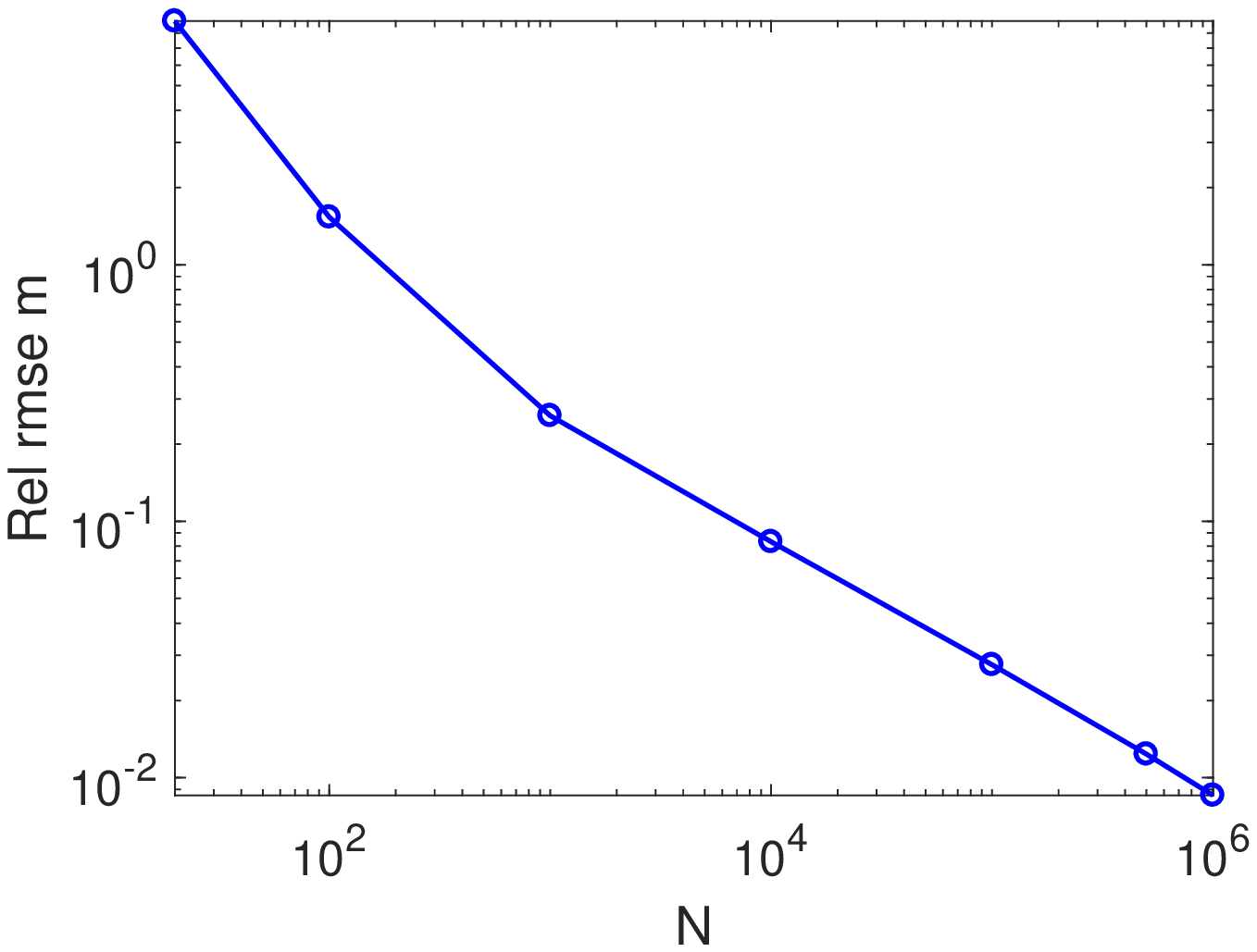}\\
	\end{tabular}
    \caption{Plot of relative root mean square error for estimated $(p^*, {{e^{(1)}}}^*, \theta^*, m^*)$ in semi-logarithmic scale on $x$-axis.}
    	\label{fig:relrmse1}
	\end{figure}

For all parameters (including the scale parameter $m$), the behaviour of relative bias, variance and relative root mean square error as the number of samples increases reveals good precision and accuracy. In particular, low values of such error quantities are already obtained when $N \approx 10^2$.


\subsection{Parameter estimation on synthetic neighbourhoods}

We now test the ML estimation procedure on a simple synthetic image reported in Figure \ref{fig:edgegeom1}. Here, the goal is to evaluate the effectiveness of the estimation when discriminating between different image regions such as edges, corners and circular profiles in terms of the functional shape of the estimated BGGD. In the following test, we estimate the parameters of the unknown BGGD in three different situations where a pixel surrounded by a $11\times 11$ neighbourhood is chosen to lie on a vertical edge (Fig. \ref{fig:edgegeom}), a corner (Fig. \ref{fig:cornergeom}) and on a circular profile (Fig. \ref{fig:circlegeom}). In order to avoid degenerate configurations of the gradients, such as the ones described in \eqref{eq:configurations}, we preliminary corrupt the image by a small Additive White Gaussian noise (AWGN) with $\sigma=0.03$.

\paragraph{Edge points} In Fig. \ref{fig:edgegeom2}, we report the scatter plot of the gradients of the edge points in the red-bordered region countered in Figure \ref{fig:edgegeom1}, which, as expected, shows its distribution along the $x$-axis.
The parameter estimation procedure of the BGGD at one of such edge points is run by taking $121$ samples of gradients in the $11 \times 11$ neighbourhood. The estimation procedure results in the following parameters $(p^*,~{{e^{(1)}}}^*,~\theta^*,~m^*)=(0.07,~1.60,~\ang{-177.82},~2*10^{-5})$. Note that the low value of the parameter $p$ leads to a very fat tail distribution, as shown in Fig. \ref{fig:edgegeom3}. The orientation and the eccentricity of the level curves are in line with the clear directionality of the samples as it can be seen in Figure \ref{fig:edgegeom4}.

\begin{figure}[h!]
	\centering
	\begin{subfigure}[t]{0.45\textwidth}
	    \centering
		\includegraphics[scale = 0.56]{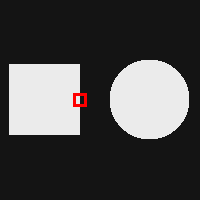}
		\caption{}
		\label{fig:edgegeom1}
    \end{subfigure}
    \begin{subfigure}[t]{0.45\textwidth}
    	\centering
		\includegraphics[scale = 0.35]{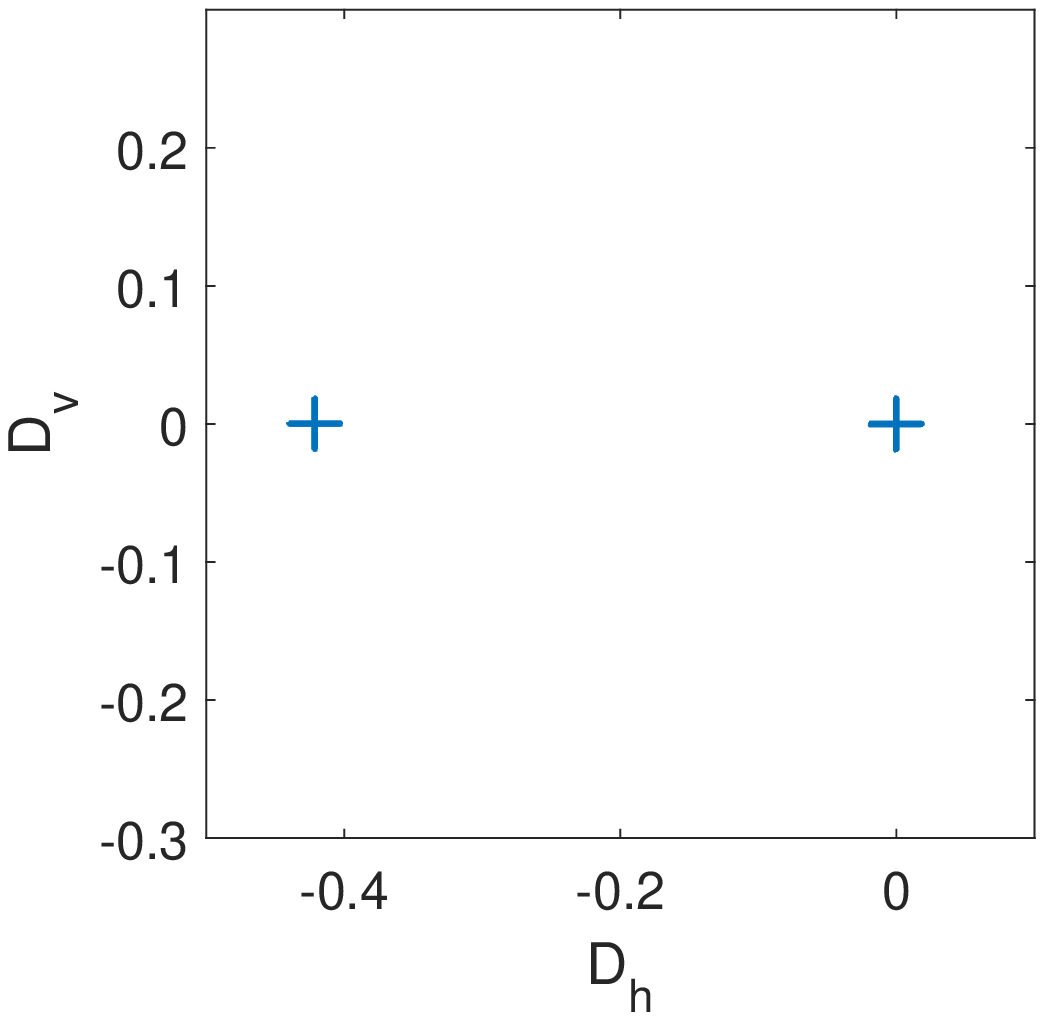}
		\caption{}
		\label{fig:edgegeom2}
    \end{subfigure}\\
    \begin{subfigure}[t]{0.45\textwidth}
    	    \centering
	\includegraphics[scale = 0.4]{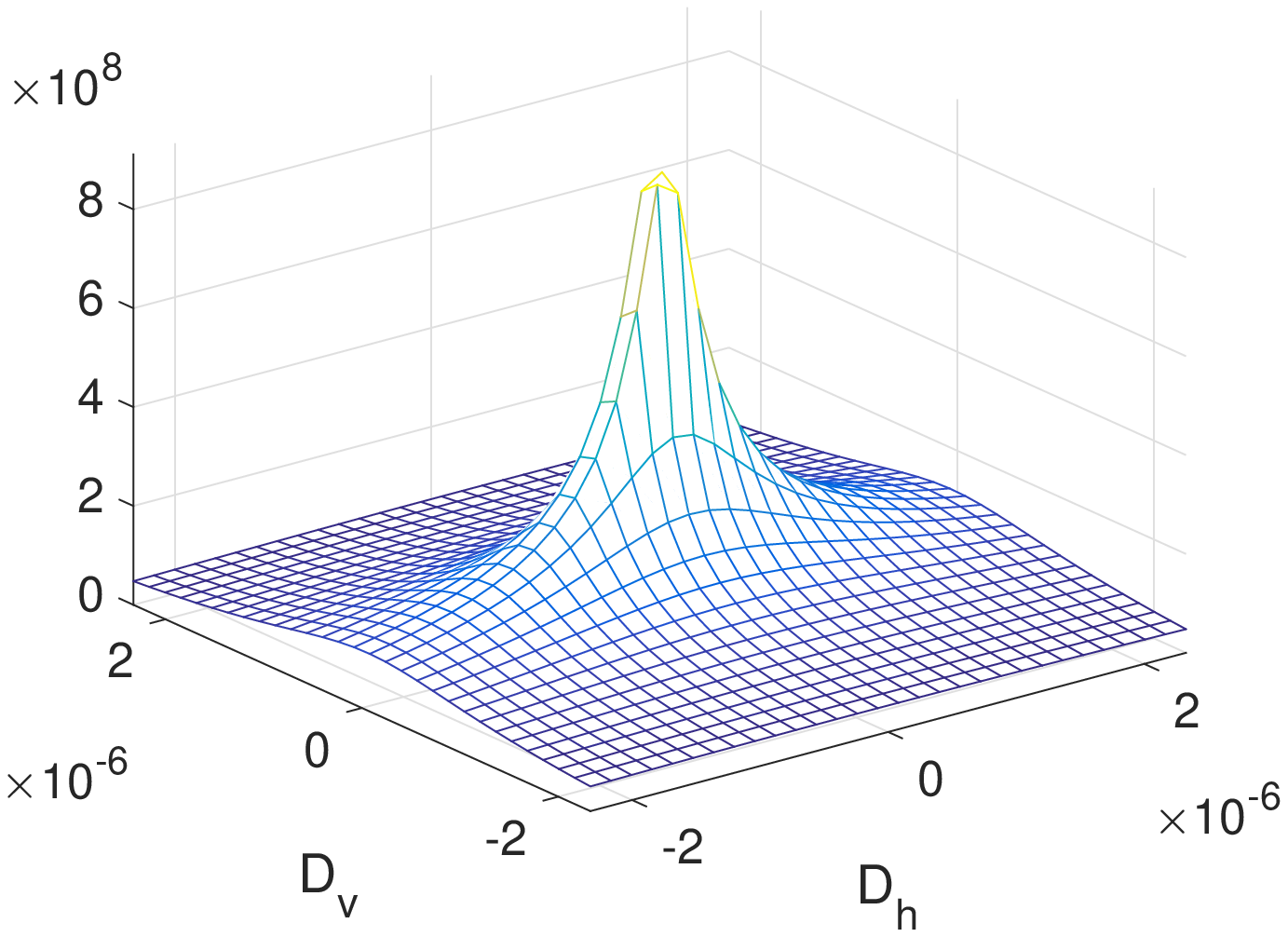}
    \caption{}
    \label{fig:edgegeom3}
    \end{subfigure}
    \begin{subfigure}[t]{0.45\textwidth}
    \centering
        \includegraphics[scale = 0.35]{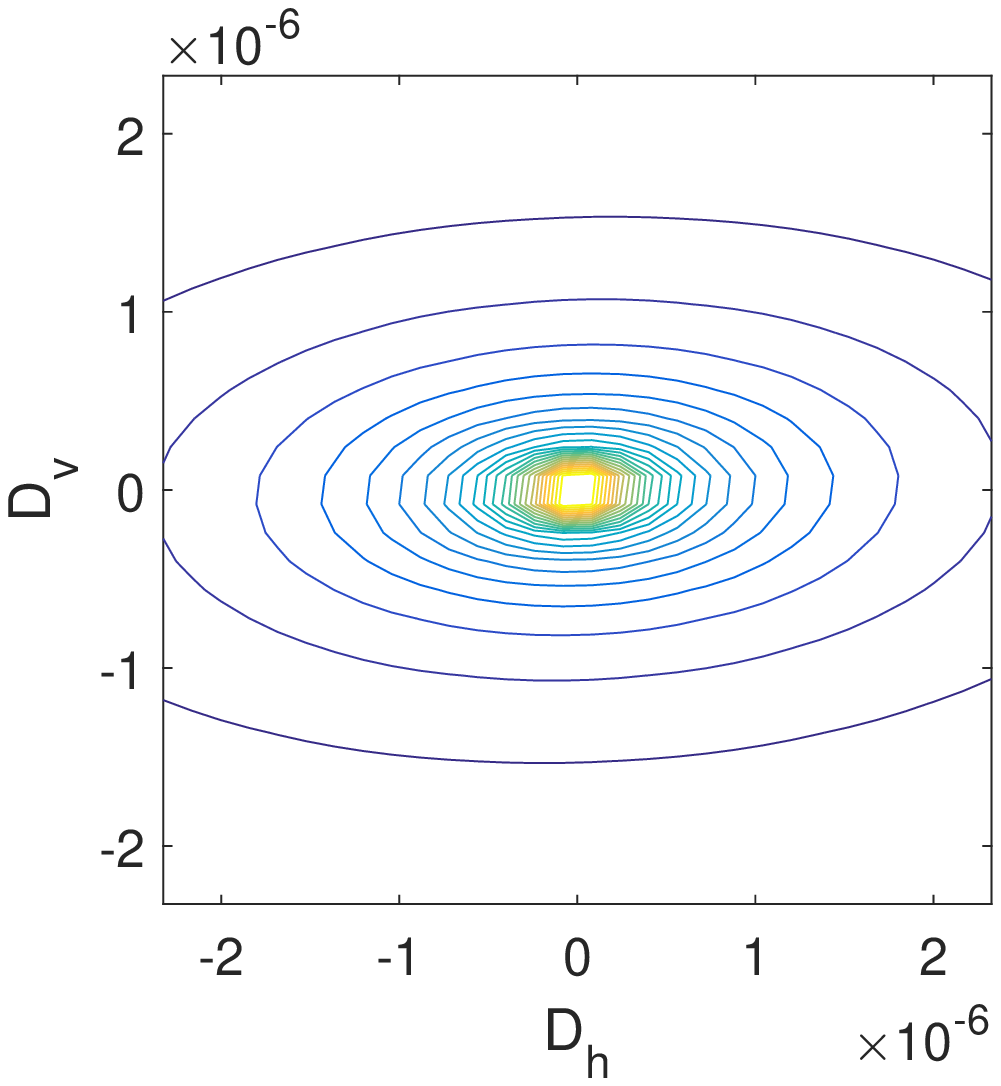}
        \caption{}
        \label{fig:edgegeom4}
    \end{subfigure}
	\caption{\ref{fig:edgegeom1}: BGGD Parameter estimation for a synthetic geometrical image. Test for edge image pixel. \ref{fig:edgegeom2}: Scatter plot of the gradients in the read-bordered region. \ref{fig:edgegeom3}: PDF with estimated parameters $(p^*,~{e^{(1)}}^*,~\theta^*,~m^*)=(0.07,~1.60,~\ang{-177.82},~2*10^{-5})$. \ref{fig:edgegeom4}: Level curves of the estimated PDF. }
	\label{fig:edgegeom}
\end{figure}

\paragraph{Corner points} For the corner example in Figure \ref{fig:cornergeom}, the scatter plot of the gradients is reported in Figure \ref{fig:cornergeom2}. The ML procedure results in this case in the estimation $(p^*,~{e^{(1)}}^*,~\theta^*,~m^*)$ = $ (0.07,~1.08,~\ang{72.49},~3*10^{-7})$. The estimated PDF is reported in Fig. \ref{fig:cornergeom3}. Similarly as before, note that a very fat-tail distribution is estimated. On the other hand, since ${e^{(1)}}^*\approx 1$, we also have ${e^{(2)}}^* \approx 1$ and the eccentricity of the ellipse $\epsilon\approx 0$. We can conclude that, in this case, the distribution is almost isotropic and the angle $\theta$ has a negligible influence on the orientation of the level curves as it can be seen in Figure \ref{fig:edgegeom4}.

\begin{figure}[h!]
	\centering
	\begin{subfigure}[t]{0.45\textwidth}
	    \centering
		\includegraphics[scale = 0.56]{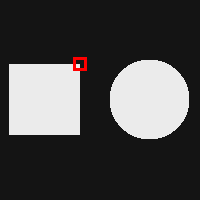}
		\caption{}
	\label{fig:cornergeom1}
    \end{subfigure}
    \begin{subfigure}[t]{0.45\textwidth}
    	\centering
		\includegraphics[scale = 0.35]{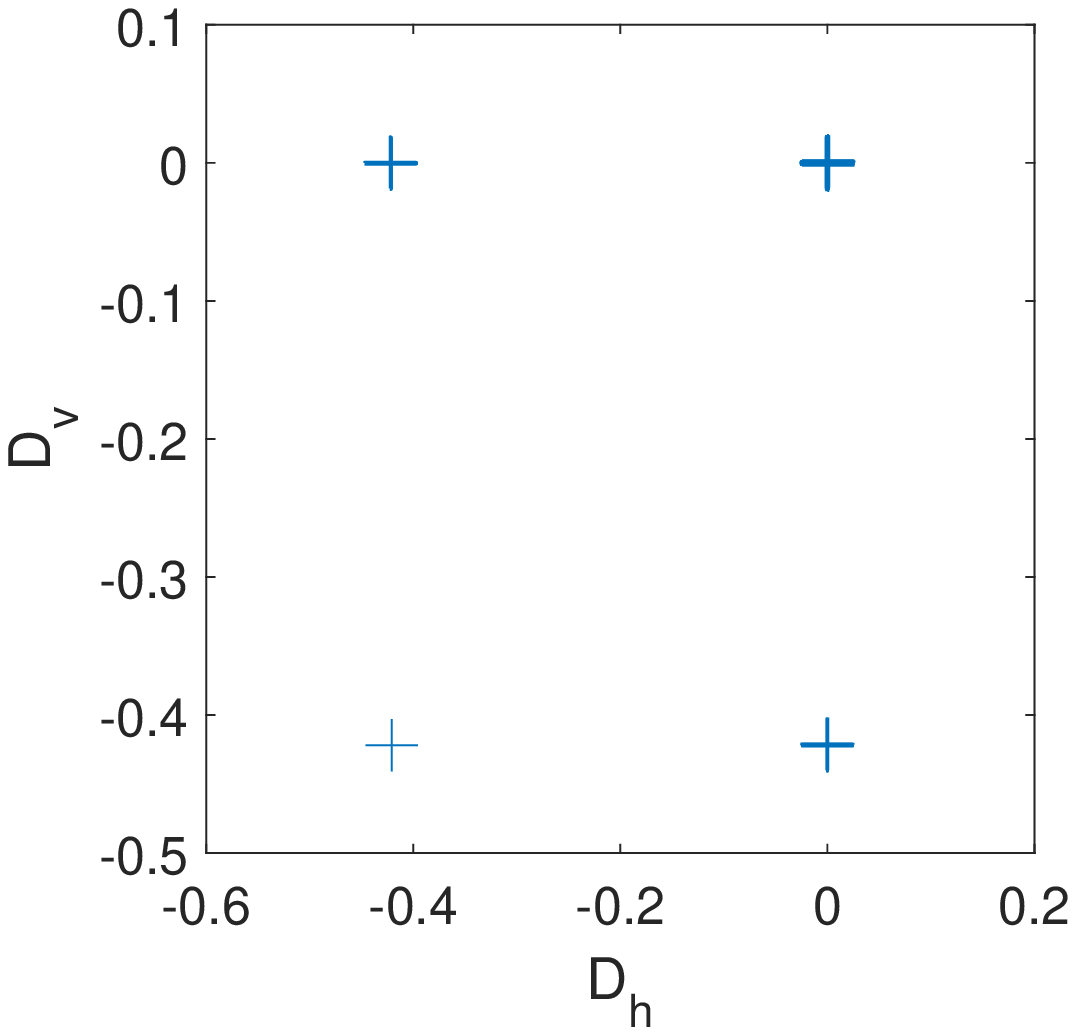}
		\caption{}
	\label{fig:cornergeom2}
    \end{subfigure}\\
    \begin{subfigure}[t]{0.45\textwidth}
    	    \centering
	\includegraphics[scale = 0.4]{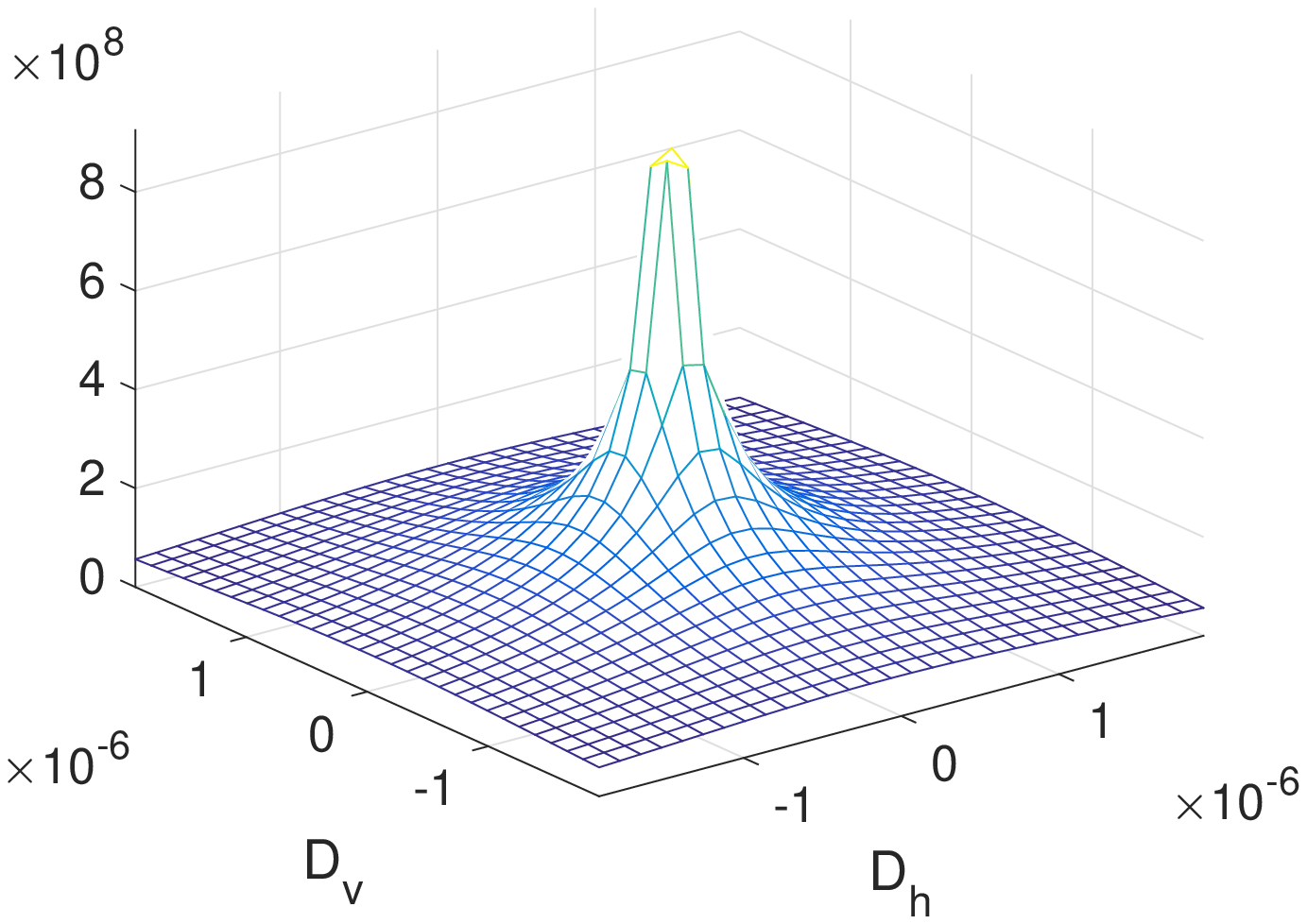}
    \caption{}
    \label{fig:cornergeom3}
    \end{subfigure}
    \begin{subfigure}[t]{0.45\textwidth}
    \centering
        \includegraphics[scale = 0.35]{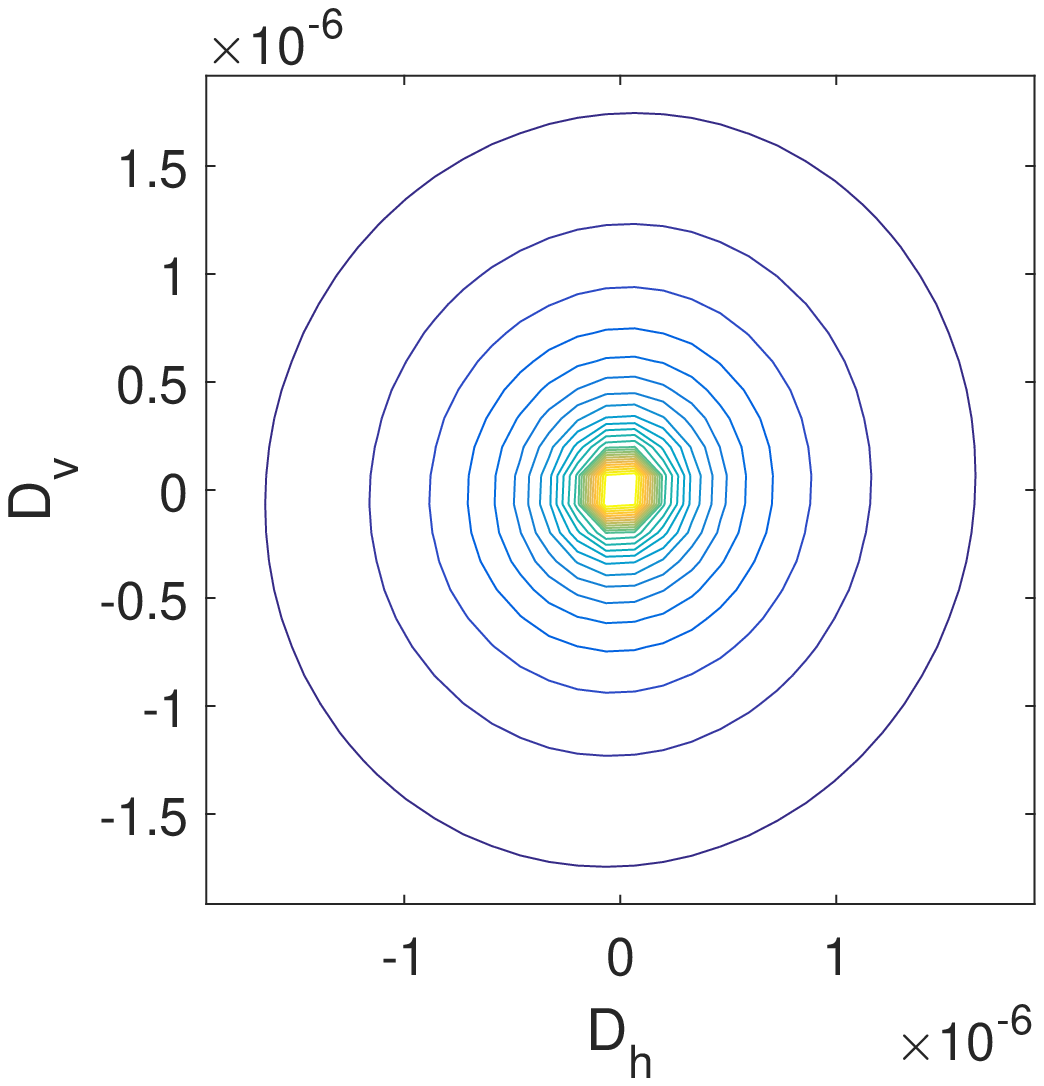}
        \caption{}
    \label{fig:cornergeom4}
    \end{subfigure}
	\caption{\ref{fig:cornergeom1}: BGGD Parameter estimation for a synthetic geometrical image. Test for corner image pixel. \ref{fig:cornergeom2}: Scatter plot of the gradients in the read-boarded region. \ref{fig:cornergeom3}: PDF with estimated parameters $(p^*,~{e^{(1)}}^*,~\theta^*,~m^*)$ = $ (0.07,~1.08,~\ang{72.49},~3*10^{-7})$ . \ref{fig:cornergeom4}: Level curves of the estimated PDF. }
	\label{fig:cornergeom}
\end{figure}

\paragraph{Circle points} Finally, we consider the ML parameter estimation procedure in correspondence with a pixel lying on a circular profile, see Figure \ref{fig:circlegeom}. In this case, the estimated parameters are $(p^*,{e^{(1)}}^*, \theta^*,m^*)$ = $(0.08,1.44,\ang{49.28},2*10^{-6})$. The values obtained for ${e^{(1)}}^*$ and $\theta^*$ reflect the spatial distribution of the gradients in Figure \ref{fig:circlegeom2}.

\begin{figure}[h!]
	\centering
	\begin{subfigure}[t]{0.45\textwidth}
	    \centering
		\includegraphics[scale = 0.56]{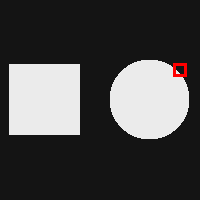}
		\caption{}
	\label{fig:circlegeom1}
    \end{subfigure}
    \begin{subfigure}[t]{0.45\textwidth}
    	\centering
		\includegraphics[scale = 0.35]{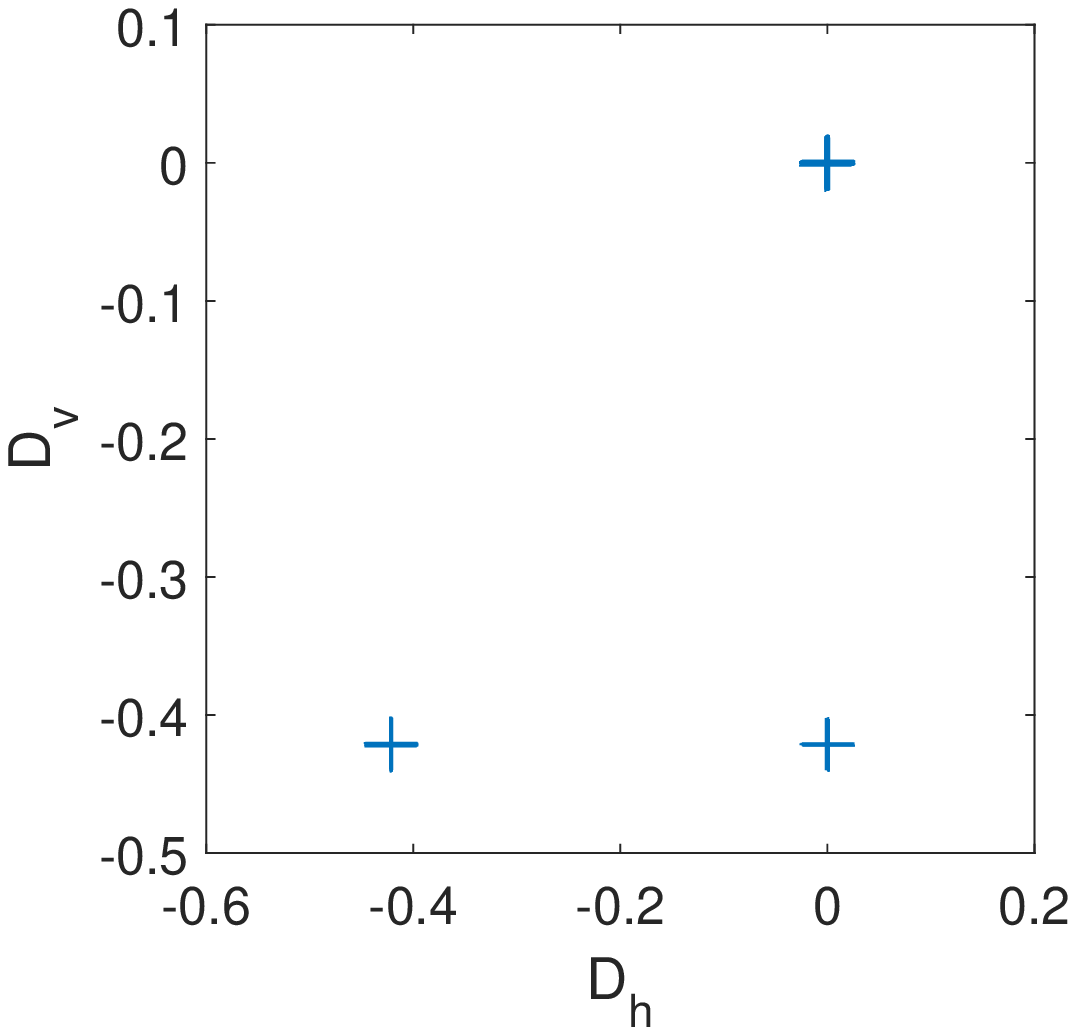}
		\caption{}
	\label{fig:circlegeom2}
    \end{subfigure}\\
    \begin{subfigure}[t]{0.45\textwidth}
    	    \centering
	\includegraphics[scale = 0.4]{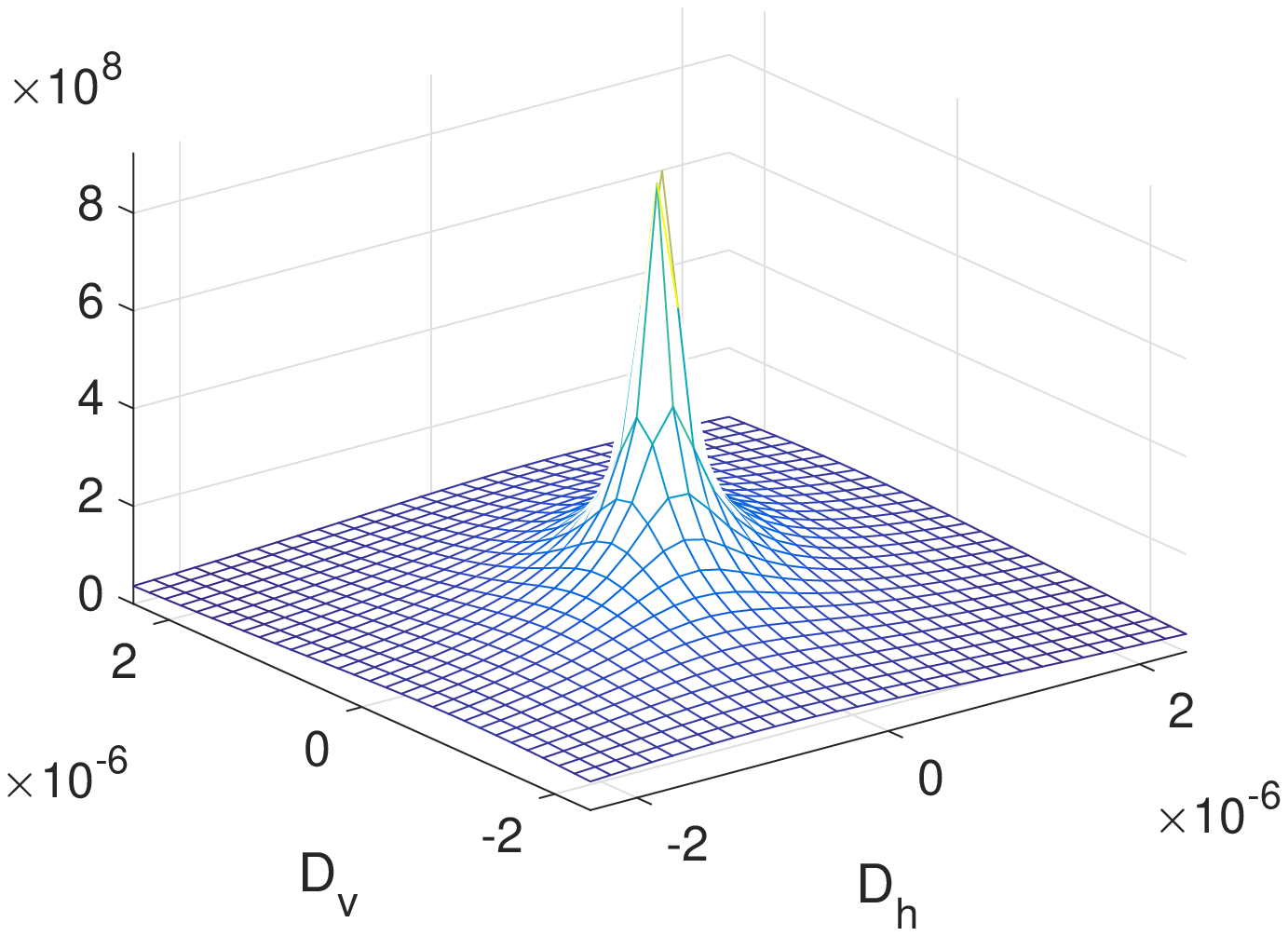}
    \caption{}
    \label{fig:circlegeom3}
    \end{subfigure}
    \begin{subfigure}[t]{0.45\textwidth}
    \centering
        \includegraphics[scale = 0.35]{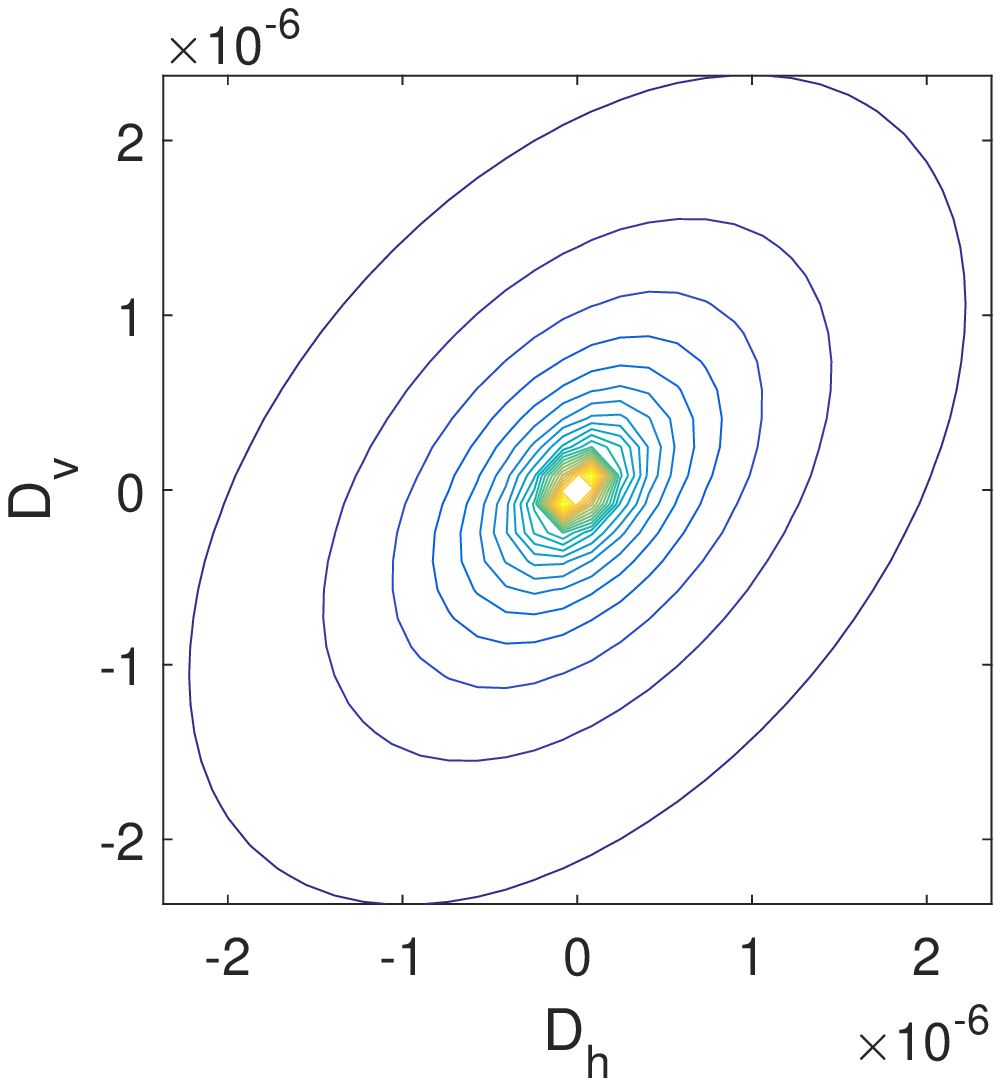}
        \caption{}
    \label{fig:circlegeom4}
    \end{subfigure}
	\caption{\ref{fig:circlegeom1}: BGGD Parameter estimation for a synthetic geometrical image. Test on image pixel lying on circular profile. \ref{fig:circlegeom2}: Scatter plot of the gradients in the read-boarded region. \ref{fig:circlegeom3}: PDF with estimated parameters $(p^*,{{e^{(1)}}}^*, \theta^*,m^*)$ = $(0.08,1.44,\ang{49.28},2*10^{-6})$. \ref{fig:circlegeom4}: Level curves of the estimated PDF. }
	\label{fig:circlegeom}
\end{figure}

\subsection{Parameter estimation on synthetic images}

Motivated by the good results above, we report in this section the numerical experiments concerned with the estimation of the four parameters $(p^*, {{e^{(1)}}}^*, \theta^*, m^*)$ at any image pixel.
For the following estimations, we fix a neighbourhood of $3\times 3$ pixels,
It is worth remarking here that the tests in section \ref{sec:est} have been computed on samples directly drawn from a BGGD. For such example, we remarked on how a large number of samples reflects on a reliable estimation of the BGGD parameters. When dealing with real images, however, our goal rather consists in estimating the parameters of the BGGD of the local gradient from the surrounding ones, since, clearly, one single sample is not sufficient to get a reliable estimate. However, the samples involved in the estimation procedure are in general not drawn from the same BGGD as their parameters may be different. Thus, their number has to be limited in order to reduce modelling errors as much as possible. In conclusion, the size of the neighbourhood is a trade off between the local properties of the image and the robustness of the estimate procedure, the former requiring small neighbourhoods, the latter requiring larger ones.
In order to avoid degenerate configurations, we corrupt the images by AWGN with $\sigma=0.03$. Moreover, the search interval for the shape parameter $p$ is set equal to $[0.1,5]$. 
We start considering the synthetic test image used already in the experiment above, i.e. Figure \ref{fig:mapgeom1}. Here we perform the estimation of the parameters at any pixel and report the local parameter maps in Figure \ref{fig:mapgeom3}, \ref{fig:mapgeom4}, \ref{fig:mapgeom5} and \ref{fig:mapgeom6}. Furthermore, we report in Figure \ref{fig:mapgeom2} the anisotropy ellipses representing the level curves of the estimated PDF,  drawn as described in Section \ref{sec:compact},  whose orientation, given by the $\theta$-map in \ref{fig:mapgeom5}, is in line with what we expected and with the test proposed in the previous sub section (see Fig. \ref{fig:edgegeom} - \ref{fig:cornergeom}). One can also observe that the higher values in the ${e^{(1)}}$-map are estimated to be along the edges, describing the strong anisotropy of the level curves there, while the higher values in the $p$-map are in the piece-wise constant regions. This can be explained by saying that in these regions the estimation procedure detects a plain Bivariate Gaussian Distribution characterised by a shape parameter $p=2$. This is of course due to the presence of AWGN.

The same experiments are proposed for \texttt{geometric} test image
in Figure \ref{fig:mapgeometric_image1}. Even though such image presents edges displaced along different orientations and details on different scales, the results showed in Figure \ref{fig:mapgeometric_image2}-\ref{fig:mapgeometric_image6} confirm the robustness of estimator in distinguishing between different image regions.



\begin{figure}[h!]
	\centering
	\begin{subfigure}[t]{0.3\textwidth}
	    \centering
		\includegraphics[width=3.2cm]{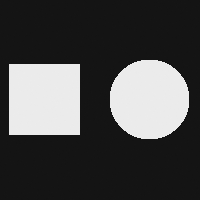}
		\caption{Test image.}
	\label{fig:mapgeom1}
    \end{subfigure}
    \begin{subfigure}[t]{0.3\textwidth}
    	\centering
		\includegraphics[width=3.2cm]{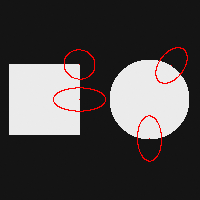}
		\caption{Anisotropy ellipses.}
	\label{fig:mapgeom2}
    \end{subfigure}
    \begin{subfigure}[t]{0.3\textwidth}
     \centering
	\includegraphics[width=4.4cm]{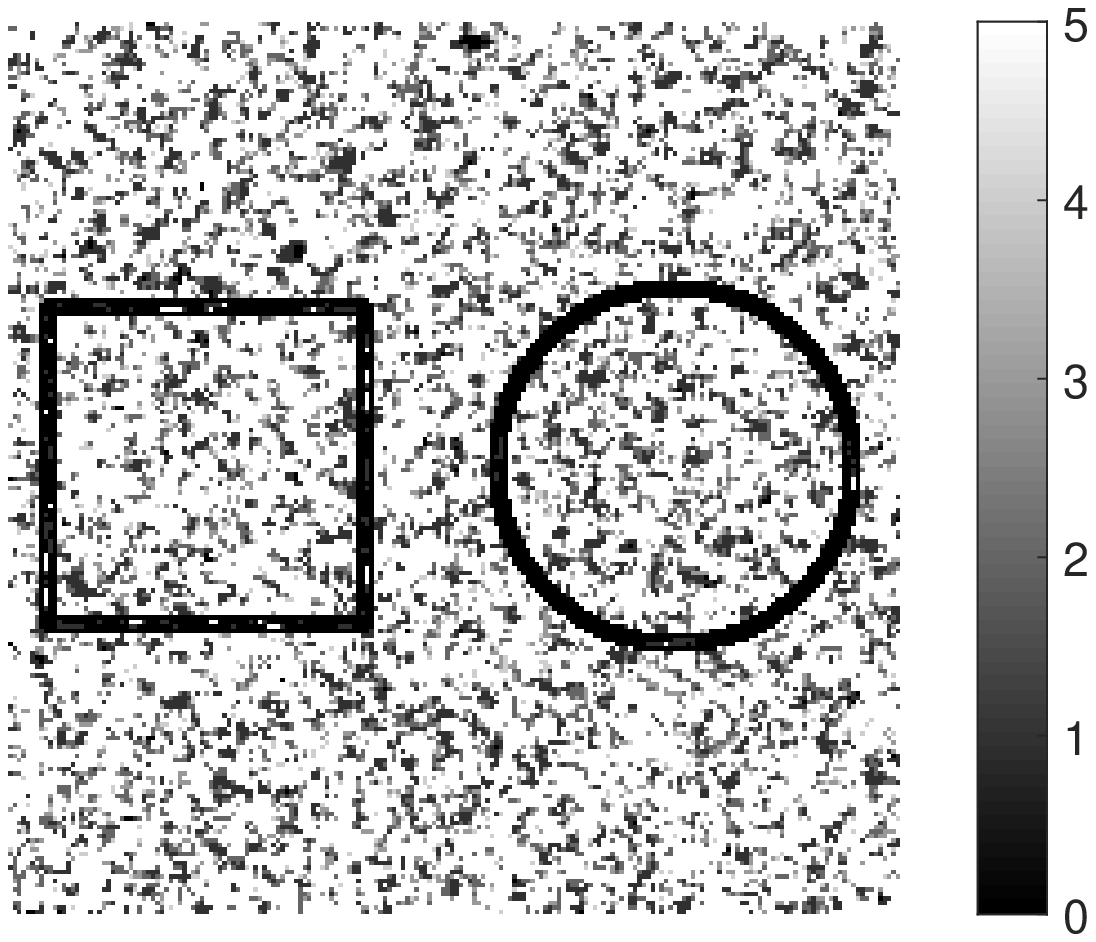}
    \caption{$p$ map.}
    \label{fig:mapgeom3}
    \end{subfigure}\\
    \begin{subfigure}[t]{0.3\textwidth}
    \centering
        \includegraphics[width=4.4cm]{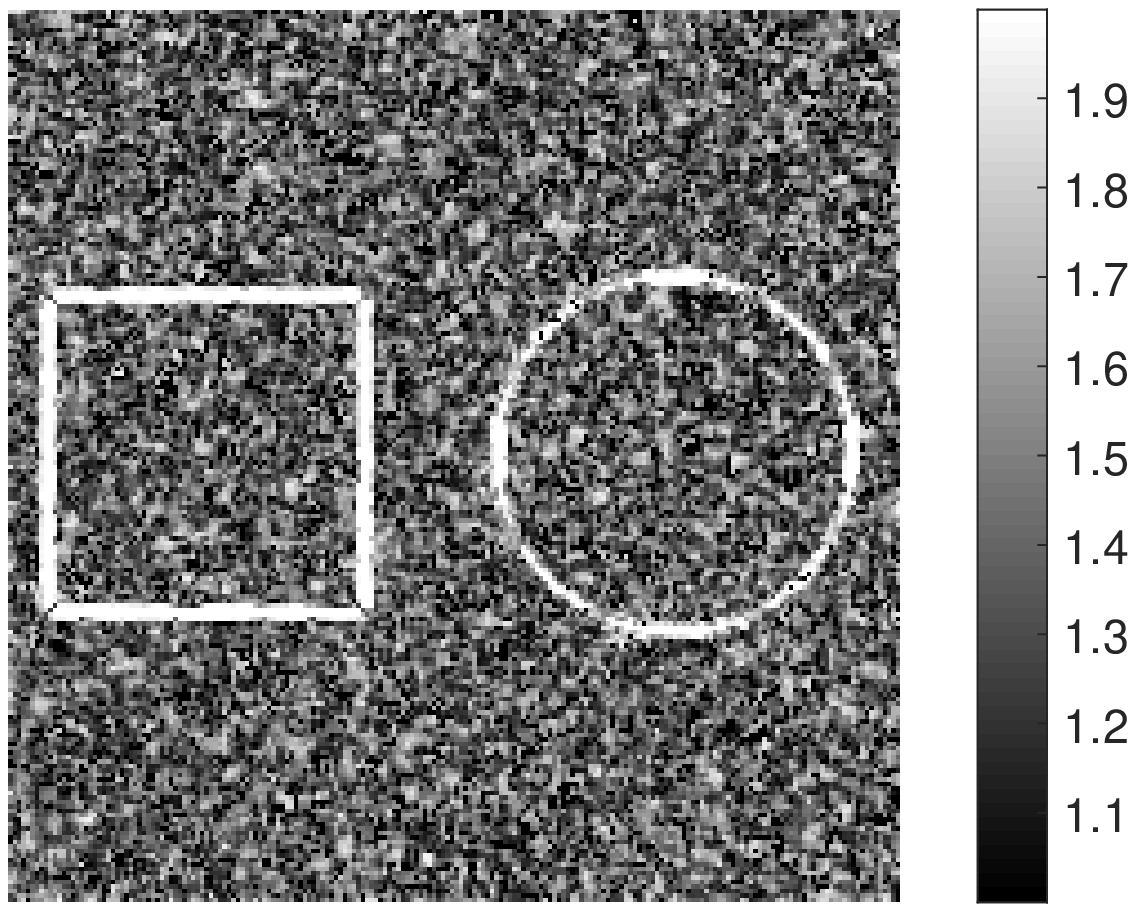}
        \caption{${e^{(1)}}$ map.}
    \label{fig:mapgeom4}
    \end{subfigure}
        \begin{subfigure}[t]{0.3\textwidth}
    	    \centering
	\includegraphics[width=4.4cm]{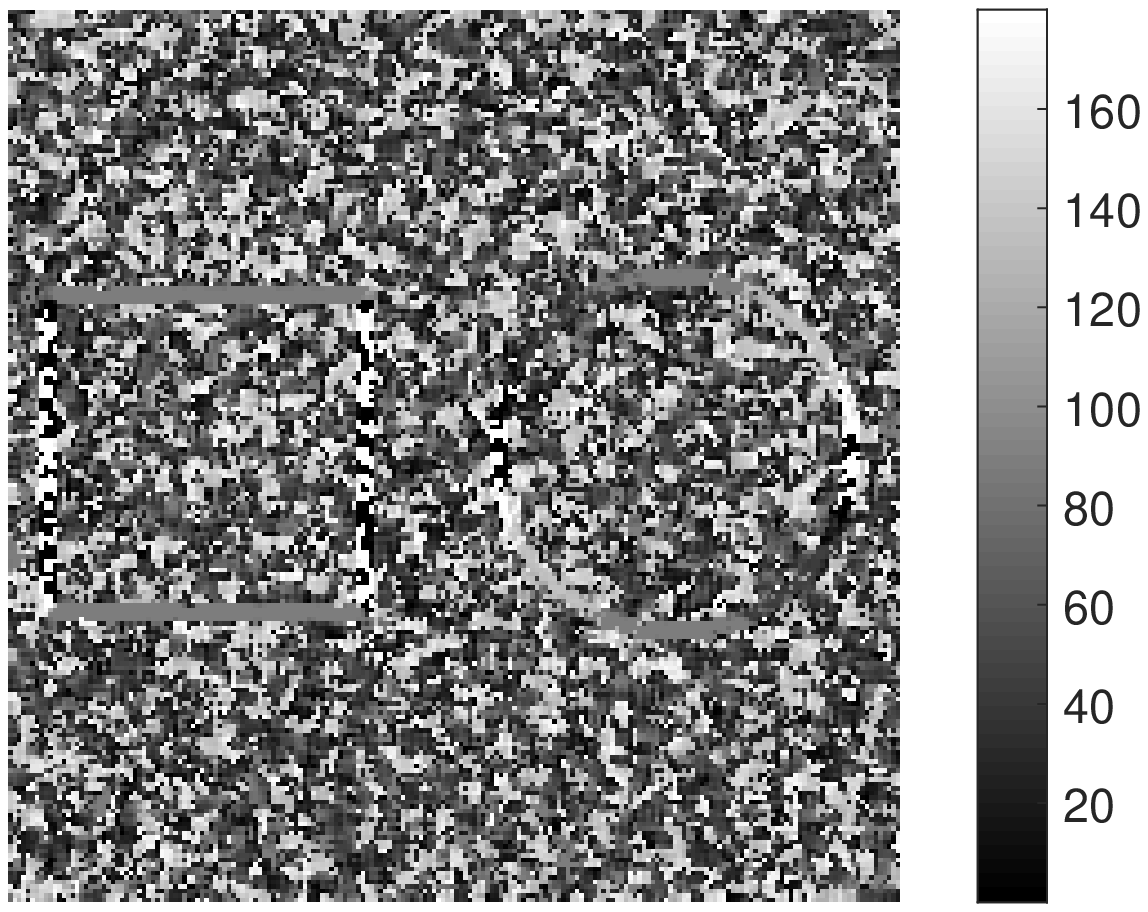}
    \caption{$\theta$ map.}
    \label{fig:mapgeom5}
    \end{subfigure}
    \begin{subfigure}[t]{0.3\textwidth}
    \centering
        \includegraphics[width=4.4cm]{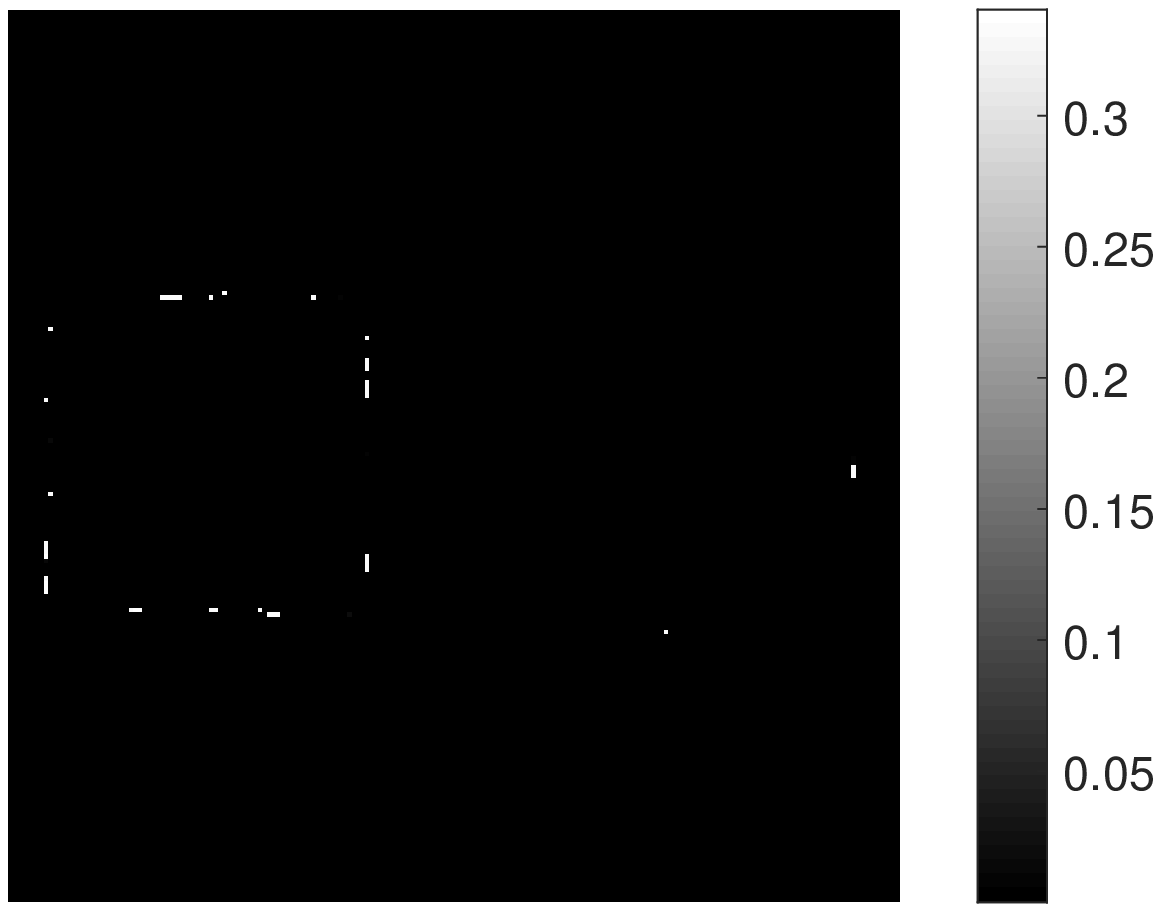}
        \caption{$m$ map.}
    \label{fig:mapgeom6}
    \end{subfigure}
	\caption{Test on \texttt{synthetic} image.}
	\label{fig:mapgeom}
\end{figure}

\begin{figure}[h!]
	\centering
	\begin{subfigure}[t]{0.3\textwidth}
	    \centering
		\includegraphics[height=3.3cm]{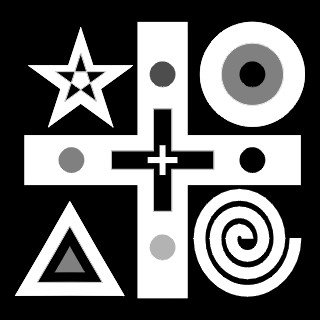}
		\caption{Test image.}
	\label{fig:mapgeometric_image1}
    \end{subfigure}
    \begin{subfigure}[t]{0.3\textwidth}
    	\centering
		\includegraphics[height=3.3cm]{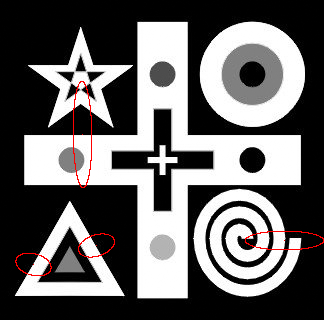}
		\caption{Anisotropy ellipses.}
	\label{fig:mapgeometric_image2}
    \end{subfigure}
    \begin{subfigure}[t]{0.3\textwidth}
    	    \centering
	\includegraphics[width=4.5cm]{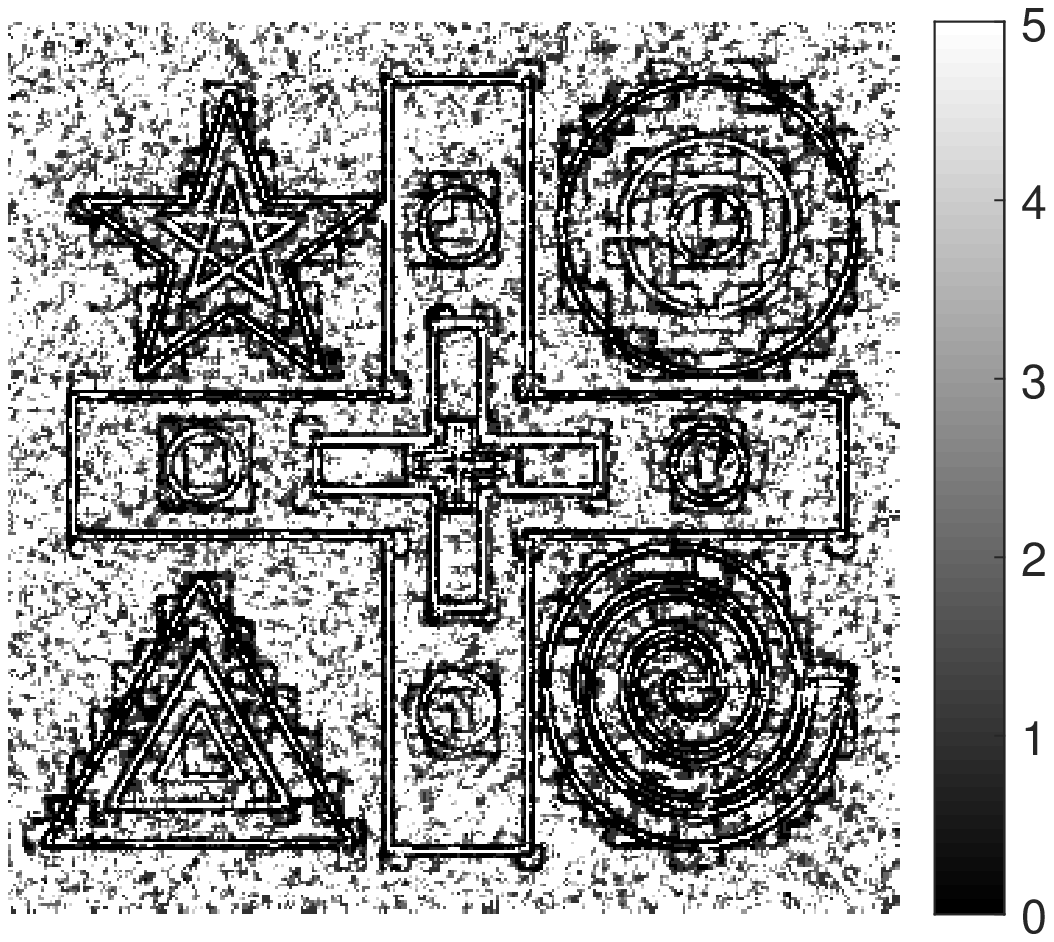}
    \caption{$p$ map.}
    \label{fig:mapgeometric_image3}
    \end{subfigure}\\
    \begin{subfigure}[t]{0.31\textwidth}
    \centering
        \includegraphics[width=4.5cm]{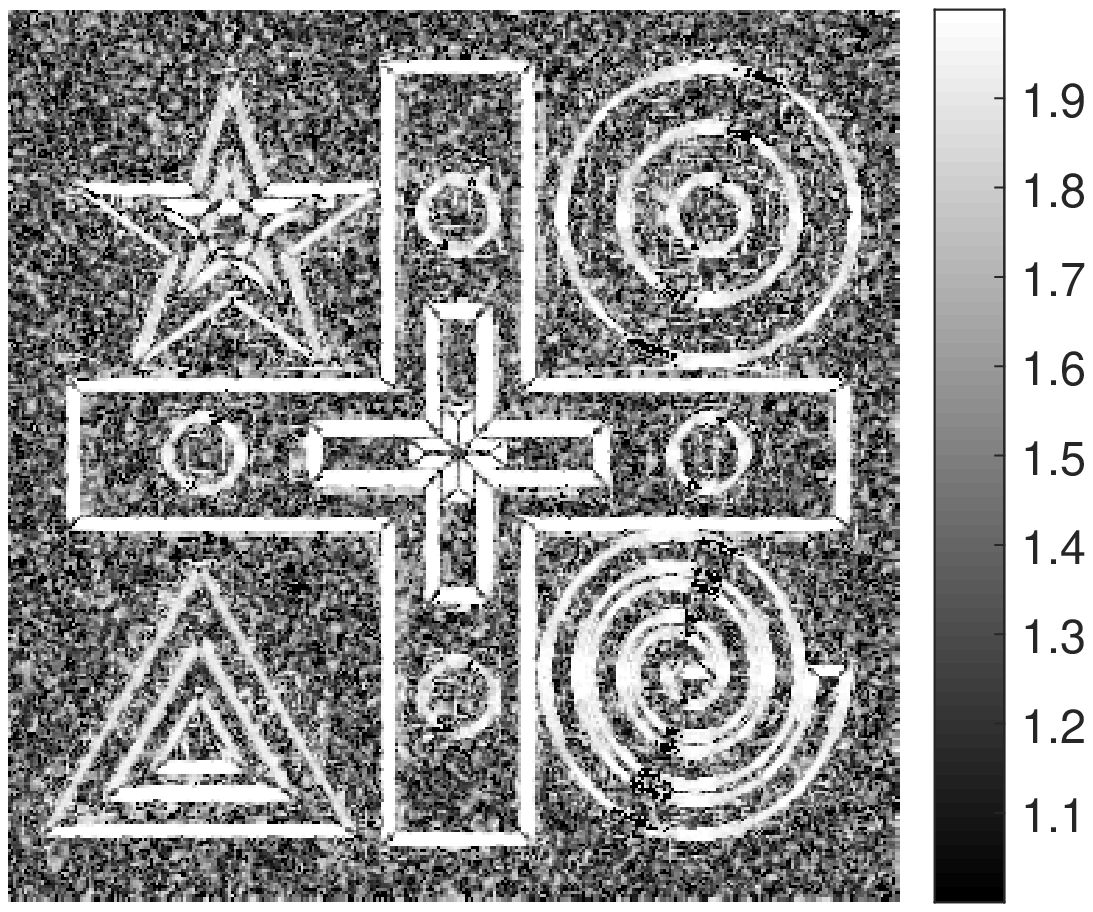}
        \caption{${e^{(1)}}$ map.}
    \label{fig:mapgeometric_image4}
    \end{subfigure}
        \begin{subfigure}[t]{0.31\textwidth}
    	    \centering
	\includegraphics[width=4.5cm]{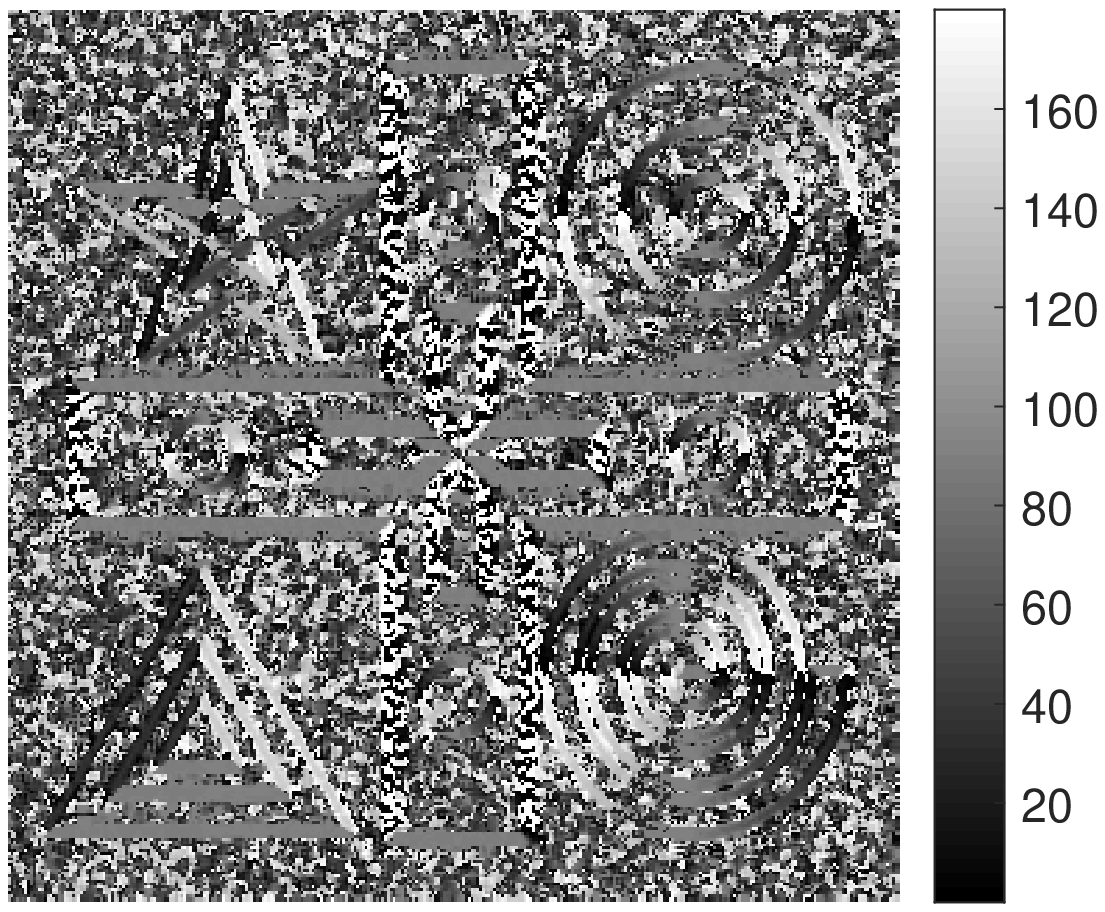}
    \caption{$\theta$ map.}
    \label{fig:mapgeometric_image5}
    \end{subfigure}
    \begin{subfigure}[t]{0.31\textwidth}
    \centering
        \includegraphics[width=4.5cm]{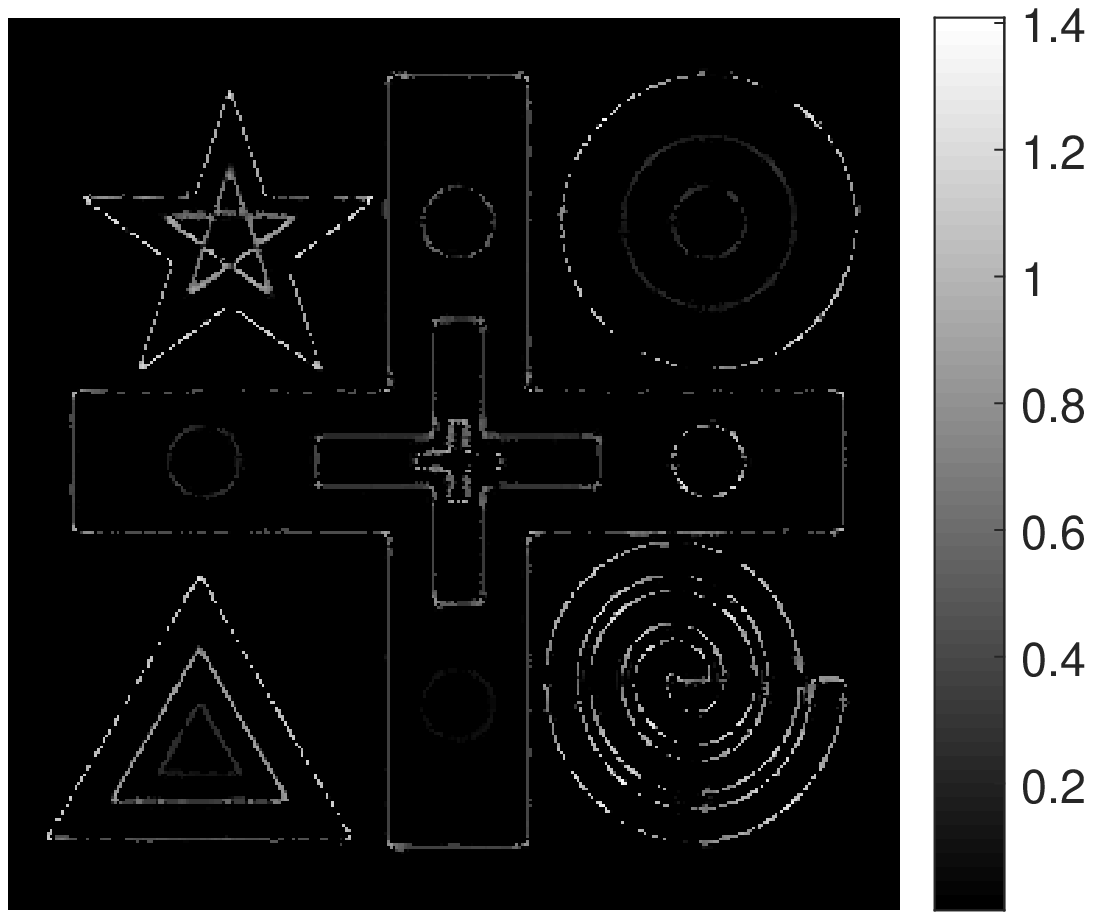}
        \caption{$m$ map.}
    \label{fig:mapgeometric_image6}
    \end{subfigure}
	\caption{Test on \texttt{geometric} image.}
	\label{fig:mapgeometric_image}
\end{figure}

\begin{remark}
In order to generate the samples used in the parameter map estimation above, one has to choose a suitable discretisation of the image gradient. Here, we considered central differences schemes. 
Compared to standard forward/backward difference schemes, this choice avoids the undesired correlation between the horizontal and the vertical components. As preliminary numerical tests showed, such correlation may result indeed into a deviation between the estimated $\theta^*$ from the one estimated above.


\end{remark}

\section{Applications to image denoising and deblurring}
\label{sec:reconstruction}

In this section, we evaluate the performance of the DTV$_{p}^{\mathrm{sv}}$-L$_2$ image reconstruction model \eqref{eq:PMa}-\eqref{eq:PMb} applied to the restoration of grey-scale images corrupted by (known) blur and AWGN.

Denoting by $u$ the ground-truth image, the quality of the given corrupted images $g$ and of the restored images
$u^*$ is measured by means of standard image quality measures, i.e. the Blurred Signal-to-Noise Ratio
\[
\;\mathrm{BSNR}(u^*,u) := 10\log_{10} \frac{\|Ku - \overline{Ku}\|_2^2}{\|u^*-Ku\|_2^2},
\]
where by $\overline{Ku}$ we have denoted the average intensity of the blurred image $Ku$, and the Improved Signal-to-Noise Ratio
\[
\mathrm{ISNR}(g,u,u^*) := 10\log_{10}\frac{\|g-u\|_2^2}{\|u^*-u\|_2^2},
\] 
defined also in terms of the given noisy $g$. The larger the BSNR and the ISNR values, the higher the quality of restoration. For a more visual-inspired standard quality measure, we will also  quantify our results in terms of the standard Structural Similarity Index (SSIM), \cite{ssim}.


The DTV$_{p}^{\mathrm{sv}}$-L$_2$ model will be compared with the following ones:
\begin{itemize}
\item The ROF or TV-L$_2$ model, \cite{ROF}; 
\item The TV$_p$-L$_2$ model, with constant $p \in (0,2]$, see \cite{tvpl2};
\item The TV$_{\alpha,p}^{\mathrm{sv}}$-L$_2$ model, with space-variant $p_i \in (0,2]$, $i \in \{1,\ldots,n\}$, \cite{vip,CMBBE}.
\end{itemize}

We stress that in order to compute the following results, an accurate and reliable estimation of the parameters appearing in the DTV$_{p}^{\mathrm{sv}}$ needs to be performed. We do that by means of the ML procedure described in Section \ref{sec:parameter_estimation} whose accuracy has been extensively confirmed by the tests in Section \ref{sec:parameter_estimation_results}.

For the numerical solution of the DTV$_{p}^{\mathrm{sv}}$-L$_2$ model we use the ADMM-based algorithm \ref{alg:1} where for all tests we manually set the penalty parameters $\beta_t$ and $\beta_r$.
Iterations are stopped whenever the following stopping criterion is verified:	%
	\begin{equation}
	\frac{\big\| u^{(k)} - u^{(k-1)} \big\|_{2}}{\big\| u^{(k-1)}\big\|_{2}} \,\;{<}\;\,10^{-4}.
	\end{equation}
Finally, the parameter $\mu$ is set based on the discrepancy principle, stated in \eqref{discr_set}.

\paragraph{Barbara image} 

We start testing the reconstruction algorithm on a zoom of a high resolution ($1024 \times 1024$) \texttt{barbara} test image with size $471 \times 361$, characterised by the joint presence of texture and cartoon regions. The image here has been corrupted by Gaussian blur of \texttt{band}$=9$ and \texttt{sigma}$=2$ and AWGN resulting in BSNR values equal to 20 dB,15 dB and 10dB. The original image and the observed image, as well as the four parameter maps, computed considering a neighbourhood of size $7\times 7$, are shown in Figure \ref{fig:barb}. In order to avoid inaccurate estimations of the parameters due to the presence of possibly large noise, the parameter $p^*$ in the TV$_p$-L$_2$ model as well as the local maps of the parameters in the DTV$_p^{\mathrm{sv}}$-L$_2$ have been computed after few iterations (usually 5) of the TV-L$_2$ model. Furthermore, as discussed in Section \ref{sec:compact}, the $p$ parameter has been computed by restricting the admissible range to $[0.1,2]$.
In Tables \ref{tab:1} and \ref{tab:2} the ISNR and SSIM values achieved by the TV-L$_2$, TV$_p$-L$_2$ (with estimated global $p=0.92$), TV$_{\alpha,p}^{\mathrm{sv}}$-L$_2$ (with space variant parameters estimated as in \cite{CMBBE}) and DTV$_{p}^{\mathrm{sv}}$-L$_2$ models for different values of initial BSNR are reported. We note that the proposed model outperforms the competing ones. 
As shown in Figure \ref{fig:barbrec}, the flexibility of the DTV$_{p}^{\mathrm{sv}}$ regulariser strongly improves the reconstruction quality mainly in terms of better texture preservation.

\begin{figure}[tbh]
	\centering
\begin{subfigure}[t]{0.31\textwidth}
\centering
\includegraphics[height=5cm]{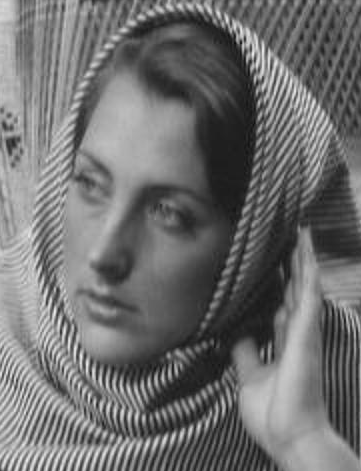}
\caption{Zoom of original $u$.}
\end{subfigure}
\begin{subfigure}[t]{0.31\textwidth}
\centering
\includegraphics[height=5cm]{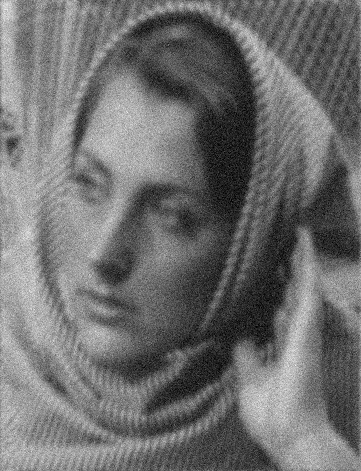}
\caption{Zoom of $g$.}
\end{subfigure}
\begin{subfigure}[t]{0.31\textwidth}
\centering
\includegraphics[height=5cm]{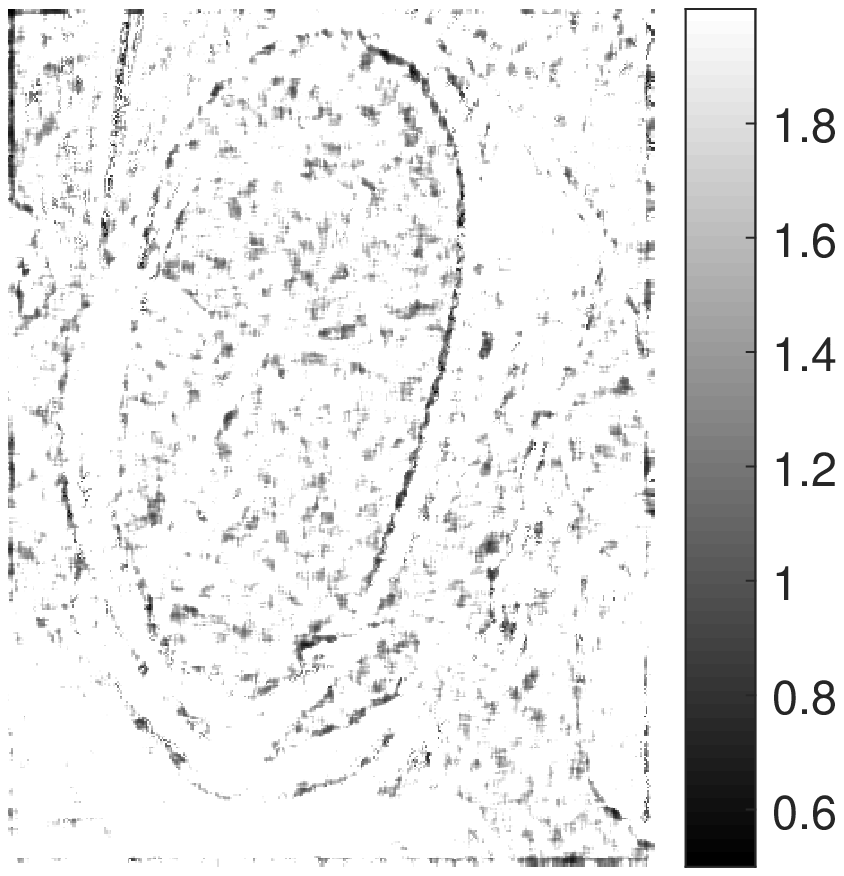}
\caption{$p$ map.}
\end{subfigure}\\
\begin{subfigure}[t]{0.31\textwidth}
\centering
\includegraphics[height=5cm]{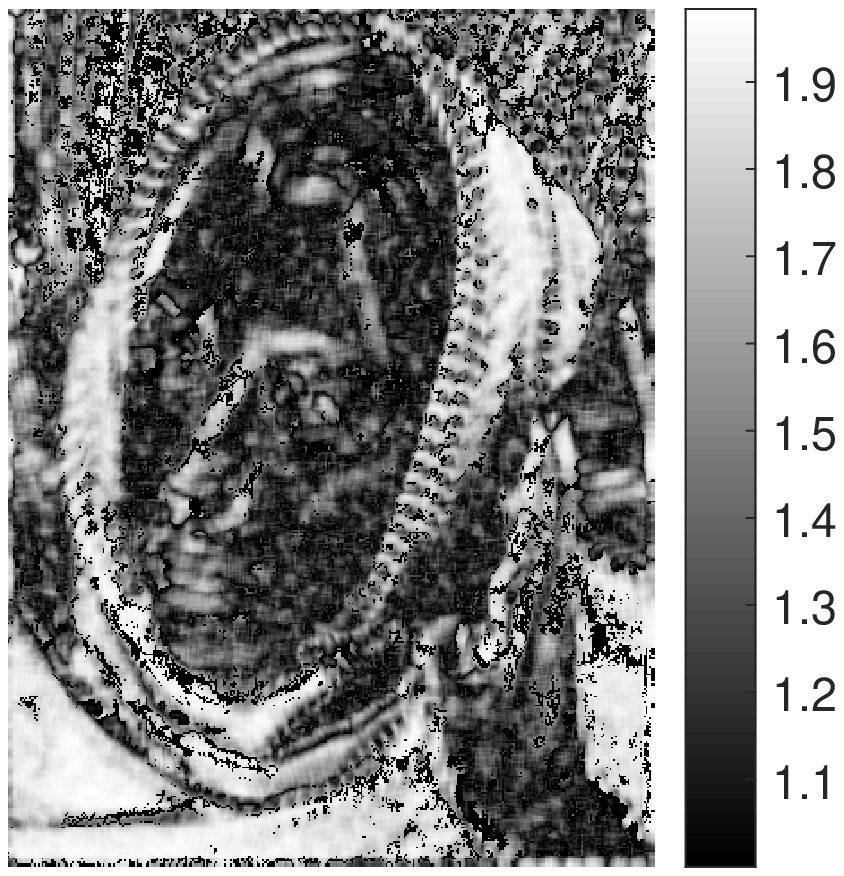}
\caption{${e^{(1)}}$ map.}
\end{subfigure}
\begin{subfigure}[t]{0.31\textwidth}
\centering
\includegraphics[height=5cm]{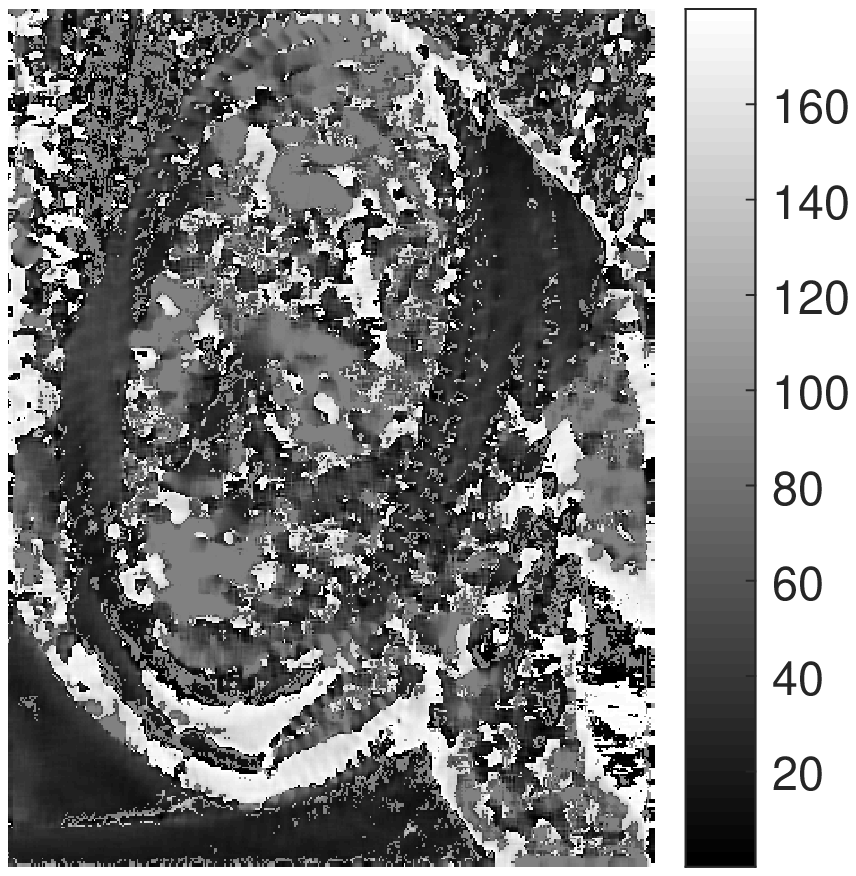}
\caption{$\theta$ map.}
\end{subfigure}
\begin{subfigure}[t]{0.31\textwidth}
\centering
\includegraphics[height=5cm]{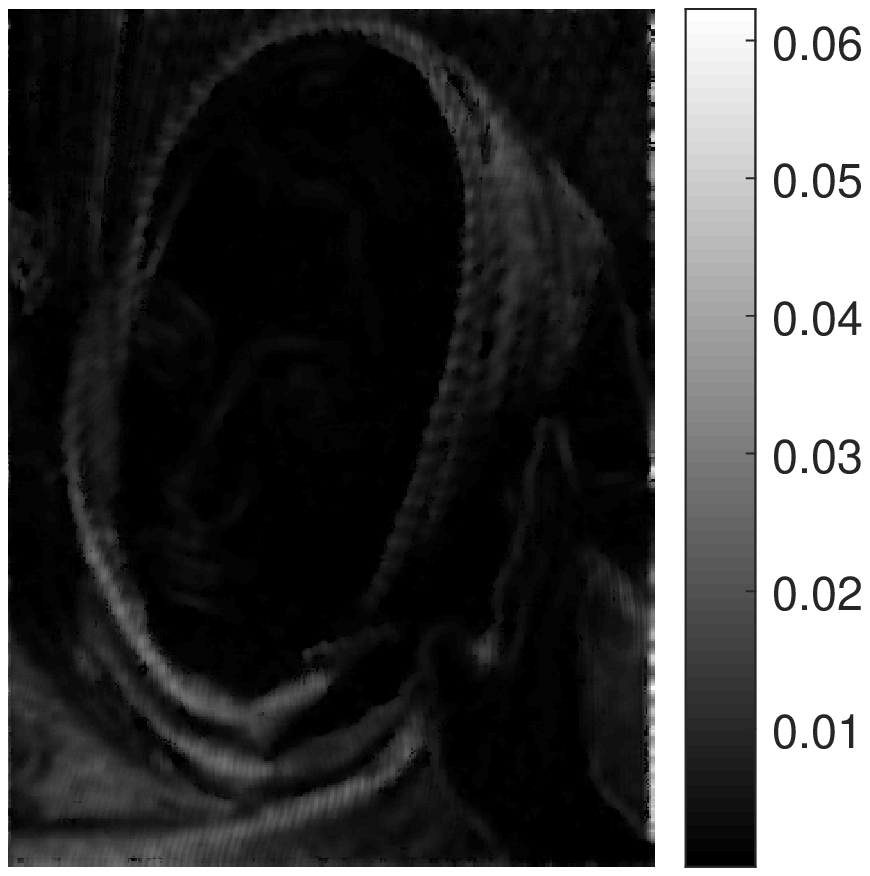}
\caption{$m$ map.}
\end{subfigure}
\caption{Parameter maps for a zoom of the \texttt{barbara} test image. Image is corrupted by AWGN and blur for a BSNR = 10 dB.}
\label{fig:barb}
\end{figure}

\begin{table}
	\caption{ISNR values for the  \texttt{barbara} test image for decreasing BSNR $=20,15,10$ dB.}
	\label{tab:1}     
	\centering
	\begin{tabular}{rrrrr}
		\hline\noalign{\smallskip}
		BSNR & TV-L$_2$& TV$_{p}$-L$_2$ &TV$_{\alpha,p}^{\mathrm{sv}}$-L$_2$ &DTV$_p^{\mathrm{sv}}$-L$_2$ \\
		\noalign{\smallskip}\hline\noalign{\smallskip}
		20               & 2.46 & 3.14 &3.23 & \textbf{3.61} \\
		15               & 1.74 & 1.99 &2.14 & \textbf{2.79} \\
		10               & 1.59  & 2.02 &2.13 & \textbf{2.90} \\
		\noalign{\smallskip}\hline
	\end{tabular}
	
\end{table}

\begin{table}
	\caption{SSIM values for the \texttt{barbara} test image for decreasing BSNR $=20,15,10$ dB.}
	\label{tab:2}   
	\centering
	\begin{tabular}{rrrrr}
		\hline\noalign{\smallskip}
		BSNR & TV-L$_2$& TV$_{p}$-L$_2$ &TV$_{\alpha,p}^{\mathrm{SV}}$-L$_2$ &DTV$_p^{\mathrm{SV}}$-L$_2$ \\
		\noalign{\smallskip}\hline\noalign{\smallskip}
		20               & 0.80 & 0.83 &0.83&\textbf{0.85} \\
		15               & 0.74 & 0.75 & 0.77& \textbf{0.80} \\
		10               & 0.65  & 0.68 & 0.69& \textbf{0.74} \\
		\noalign{\smallskip}\hline
	\end{tabular}
\end{table}

\begin{figure}[tbh]
	\centering
\begin{subfigure}[t]{0.21\textwidth}
\includegraphics[width=1.2in]{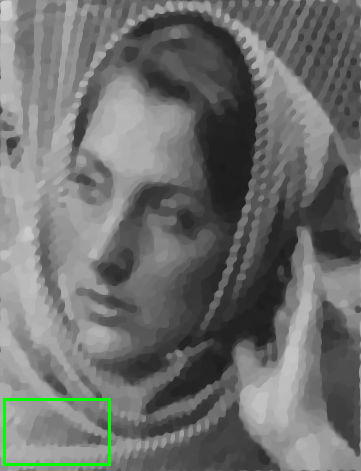}
\caption{TV-L$_2$.}
\label{fig:barbara1}
\end{subfigure}
\begin{subfigure}[t]{0.21\textwidth}
\includegraphics[width=1.2in]{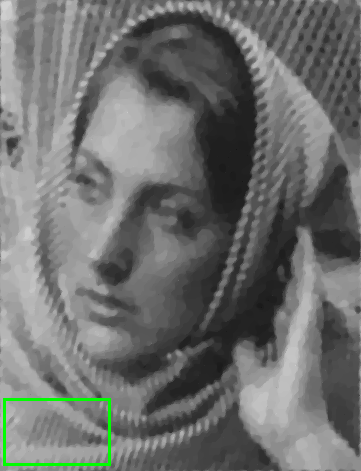}
\caption{TV$_p$-L$_2$.}
\label{fig:barbara2}
\end{subfigure}
\begin{subfigure}[t]{0.21\textwidth}
\includegraphics[width=1.2in]{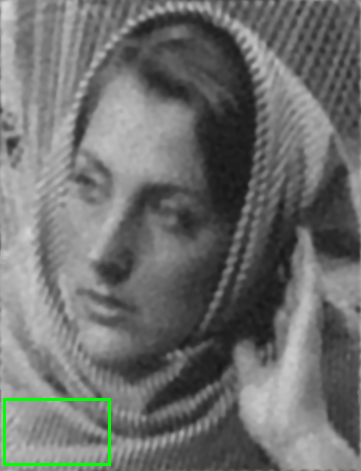}
\caption{TV$_{\alpha,p}^{\mathrm{sv}}$-L$_2$.}
\label{fig:barbara3}
\end{subfigure}
\begin{subfigure}[t]{0.21\textwidth}
\includegraphics[width=1.2in]{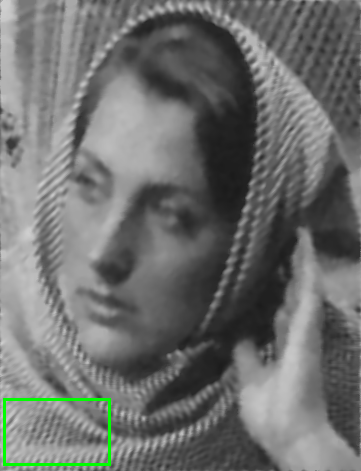}
\caption{DTV$_{p}^{\mathrm{sv}}$-L$_2$.}
\label{fig:barbara4}
\end{subfigure}  \\
\begin{subfigure}[t]{0.21\textwidth}
\includegraphics[width=1.2in]{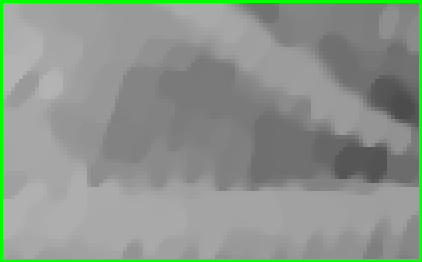}
\caption{Zoom of \ref{fig:barbara1}.}
\end{subfigure}
\begin{subfigure}[t]{0.21\textwidth}
\includegraphics[width=1.2in]{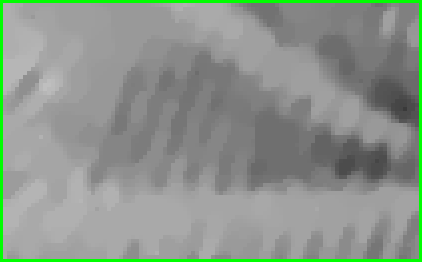}
\caption{Zoom of \ref{fig:barbara2}.}
\end{subfigure}
\begin{subfigure}[t]{0.21\textwidth}
\includegraphics[width=1.2in]{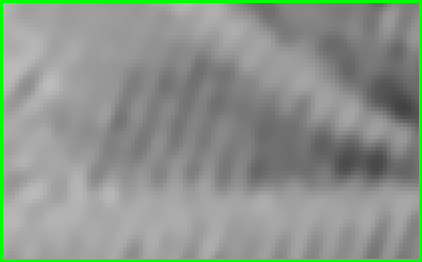} 
\caption{Zoom of \ref{fig:barbara3}.}
\end{subfigure}
\begin{subfigure}[t]{0.21\textwidth}
\includegraphics[width=1.2in]{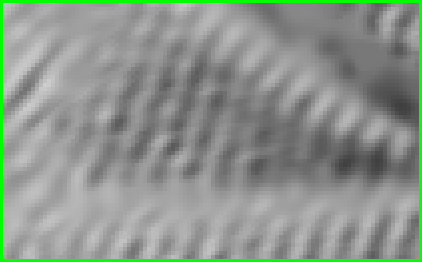}
\caption{Zoom of \ref{fig:barbara4}.}
\end{subfigure} 
\caption{Detail of reconstruction of \texttt{barbara} image \ref{fig:barb}. Texture components are much better preserved by encoding directional information.}
\label{fig:barbrec}
\end{figure}

\paragraph{Natural image} As a second test, we compared the performance of DTV$_p^{\textrm{SV}}$-L$_2$ restoration model on a $500 \times 500$ portion of a high resolution ($1024 \times 1024$) natural test image characterised by fine-scale textures of different types. As in the previous example, we similarly corrupt the image by AWGN and Gaussian blur of \texttt{band}$=9$ and \texttt{sigma}$=2$ with BSNR = 20 dB, 15 dB and 10 dB. The original and the observed images, as well as the four parameter maps computed considering neighbourhoods of size $3\times 3$ are shown in Figure \ref{fig:text}. Similarly as for the numerical test above few preliminary iterations of TV-L$_2$ are performed before computing the parameter maps. The research interval for the $p$ parameter has been set equal to $[0.1,2]$. It is worth remarking that the very small neighbourhood size used for the parameter estimation is the one yielding the best restoration results for this test. We believe that this is motivated by the very fine scale of details in the test image.
In Tables \ref{tab:3} and \ref{tab:4}, the ISNR and SSIM values achieved by the TV-L$_2$, the TV$_p$-L$_2$ (with estimated global $p=0.7$), the TV$_{\alpha,p}^{\mathrm{sv}}$-L$_2$  (with space variant parameters estimated as in \cite{CMBBE}) and the DTV$_{p}^{\mathrm{sv}}$-L$_2$ models for  different values of BSNR are reported. Also in this case,  the proposed model outperforms the competing ones. Note that the improvement is actually more significant in correspondence of higher noise levels.
In Figure \ref{fig:textrec}, a visual comparison between the reconstructions obtained by the different models for BSNR$=10  $dB is proposed.

\begin{figure}[tbh]
	\centering
\begin{subfigure}[t]{0.31\textwidth}
\centering
\includegraphics[height=3.6cm]{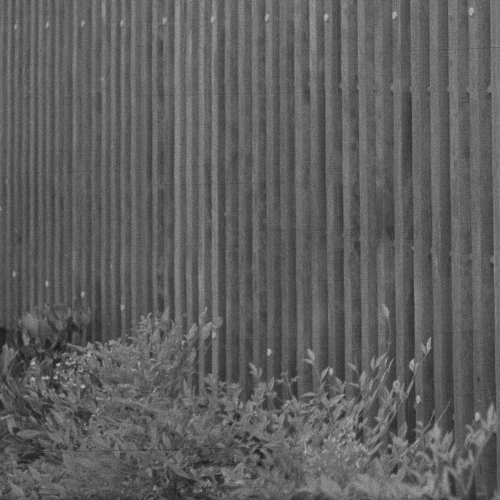}
\caption{Zoom of original $u$.}
\end{subfigure}
\begin{subfigure}[t]{0.31\textwidth}
\centering
\includegraphics[height=3.6cm]{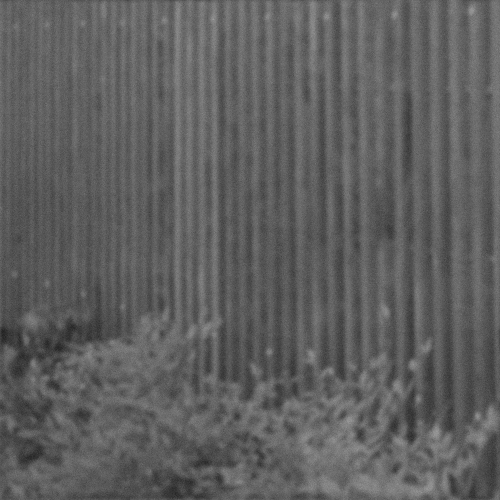}
\caption{Zoom of $g$.}
\end{subfigure}
\begin{subfigure}[t]{0.31\textwidth}
\centering
\includegraphics[height=3.8cm]{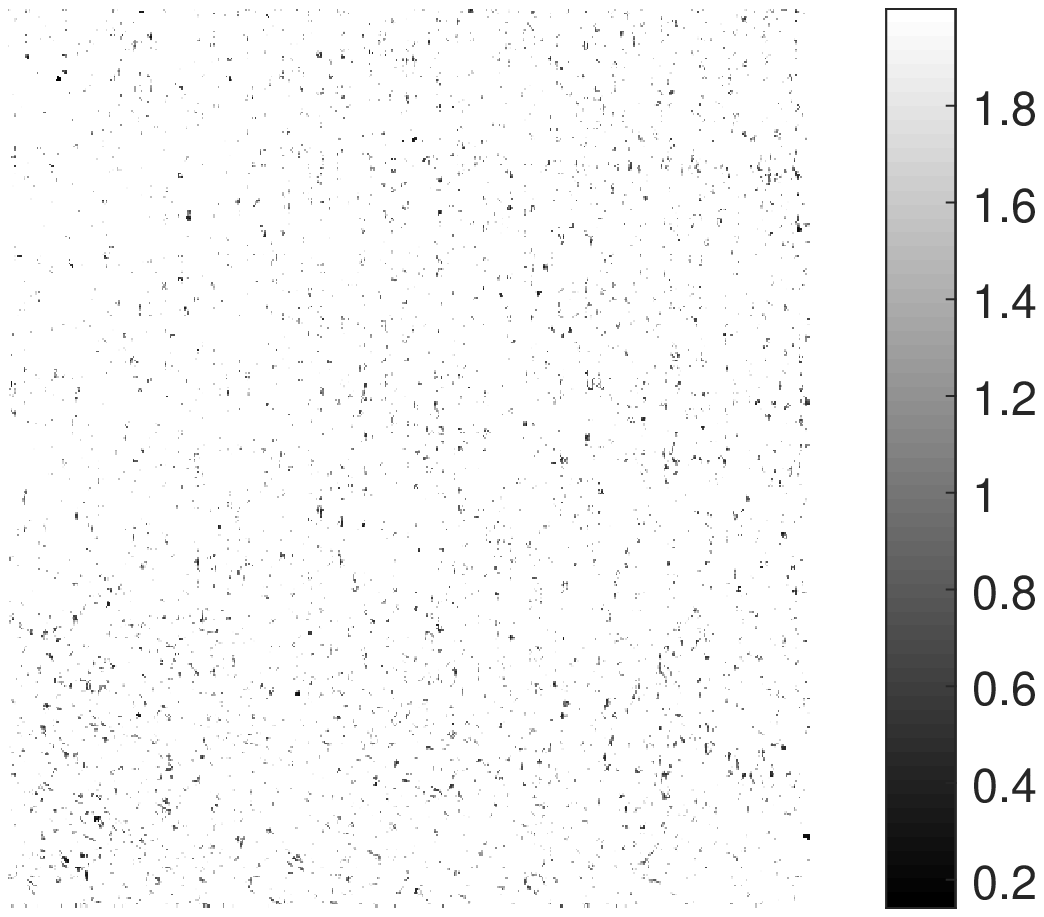} 
\caption{$p$ map.}
\end{subfigure}\\
\begin{subfigure}[t]{0.31\textwidth}
\centering
\includegraphics[height=3.8cm]{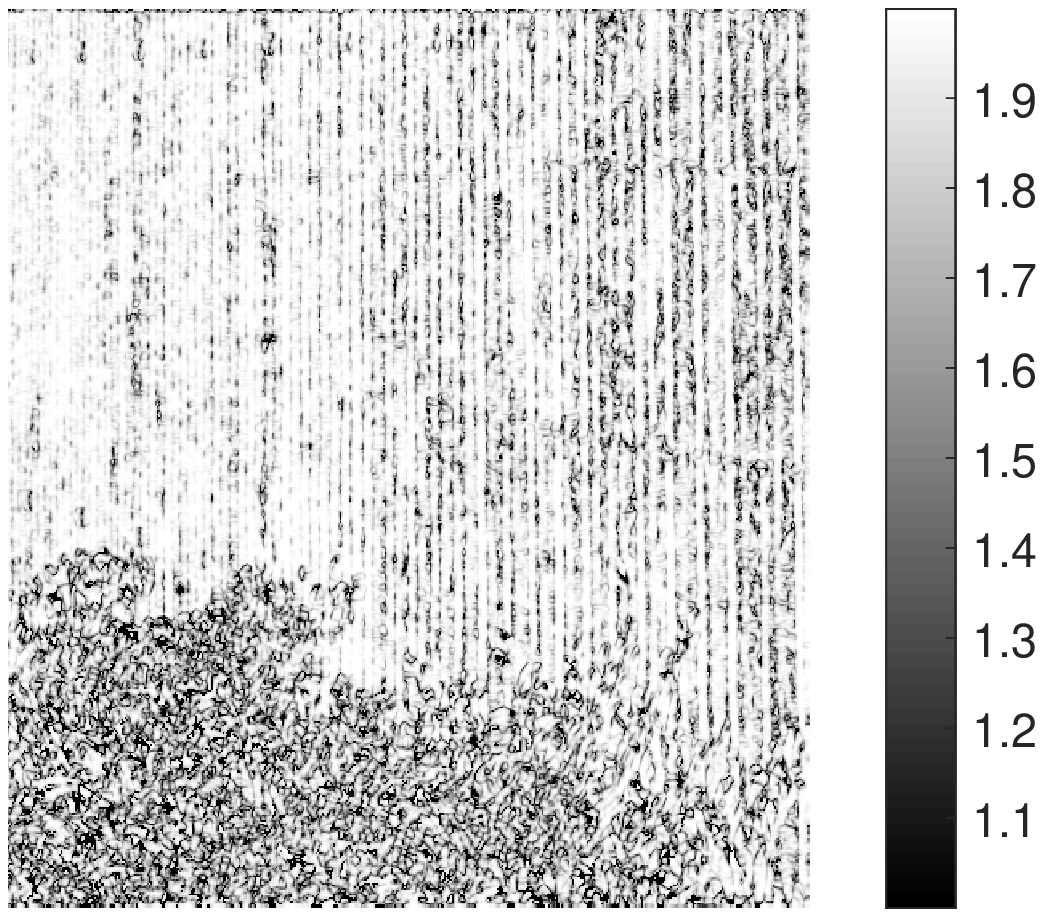}
\caption{${e^{(1)}}$ map.}
\end{subfigure}
\begin{subfigure}[t]{0.31\textwidth}
\centering
\includegraphics[height=3.8cm]{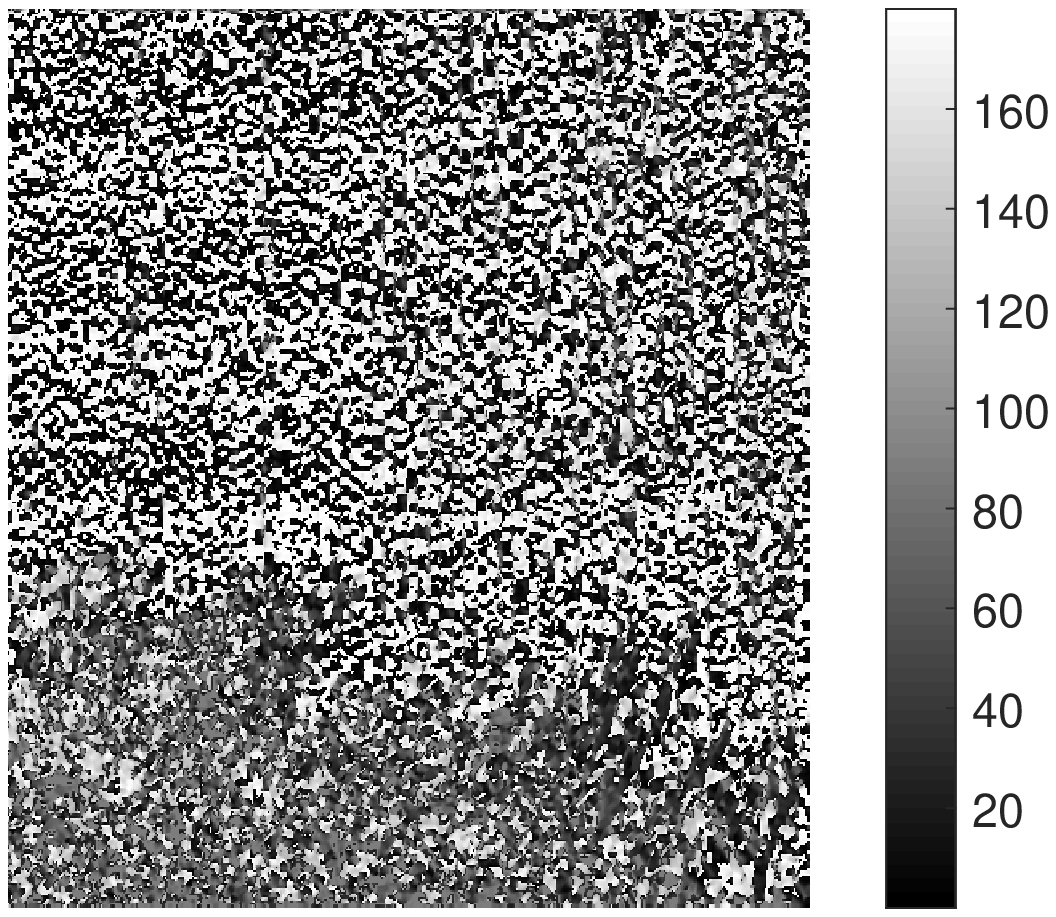}
\caption{$\theta$ map.}
\end{subfigure}
\begin{subfigure}[t]{0.31\textwidth}
\centering
\includegraphics[height=3.8cm]{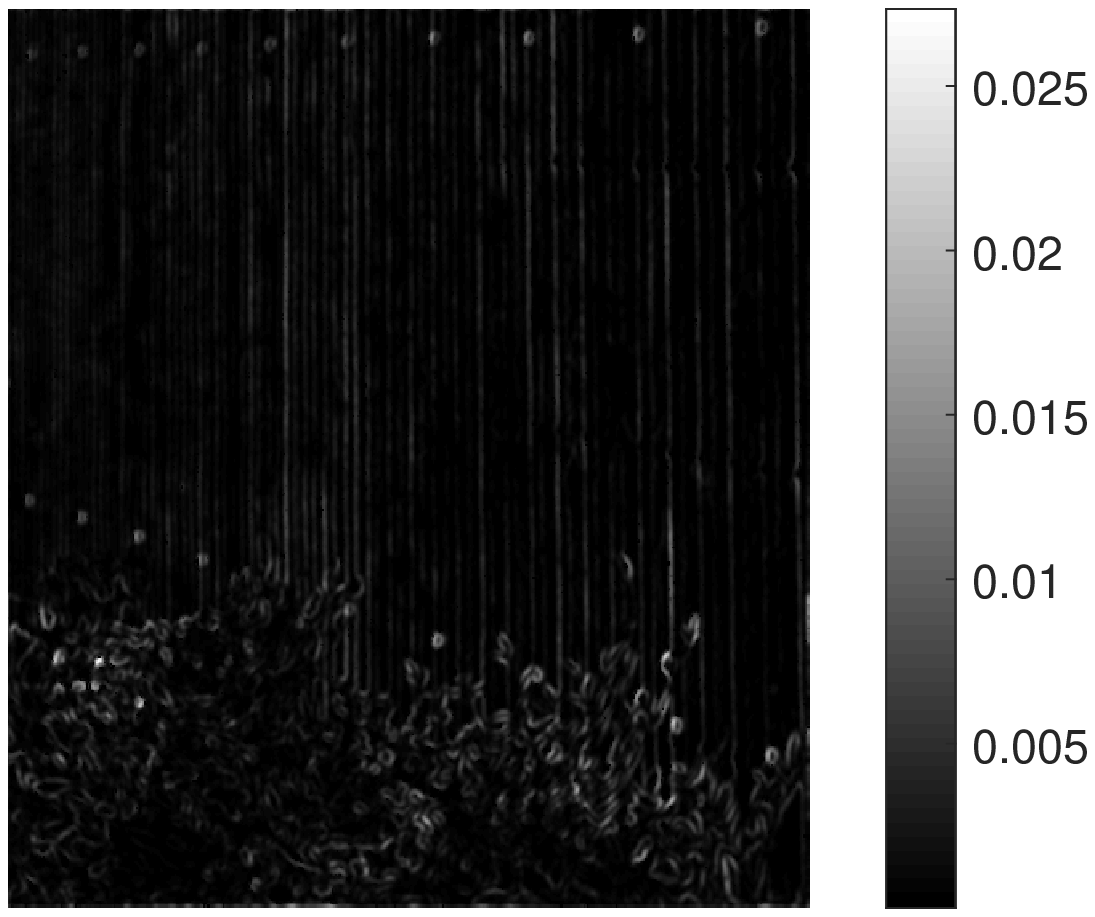}
\caption{$m$ map.}
\end{subfigure}
\caption{Parameter maps for a zoom of a natural test image. Image is corrupted by AWGN and blur for a BSNR = 10 dB.}
	\label{fig:text}
\end{figure}

\begin{table}
	\caption{ISNR values for the test image in \ref{fig:text} for BSNR $=20,15,10$ dB.}
	\label{tab:3}       
	%
	%
	\centering
	\begin{tabular}{rrrrr}
		\hline\noalign{\smallskip}
		BSNR & TV-L$_2$& TV$_{p}$-L$_2$ &TV$_{p}^{\mathrm{sv}}$-L$_2$ &DTV$_p^{\mathrm{sv}}$-L$_2$ \\
		\noalign{\smallskip}\hline\noalign{\smallskip}
		20               & 2.07 & 2.43 &2.53 & \textbf{2.78} \\
		15               & 1.83 & 2.06 &2.26 & \textbf{2.56} \\
		10               & 0.94  & 1.55 &1.86 & \textbf{2.45} \\
		\noalign{\smallskip}\hline
	\end{tabular}
	
\end{table}

\begin{table}
	\caption{SSIM values for the test image in \ref{fig:text} for BSNR $=20,15,10$ dB .}
	\label{tab:4}       
	%
	%
	\centering
	\begin{tabular}{rrrrr}
		\hline\noalign{\smallskip}
		BSNR & TV-L$_2$& TV$_{p}$-L$_2$ &TV$_{p}^{\mathrm{sv}}$-L$_2$ &DTV$_p^{\mathrm{sv}}$-L$_2$ \\
		\noalign{\smallskip}\hline\noalign{\smallskip}
		20               & 0.78 & 0.79 &0.80& \textbf{0.81} \\
		15               & 0.76 & 0.77 & 0.78& \textbf{0.79} \\
		10               & 0.70  & 0.72 & 0.74& \textbf{0.76} \\
		\noalign{\smallskip}\hline
	\end{tabular}
\end{table}

\begin{figure}[tbh]
	\centering
\begin{subfigure}[t]{0.21\textwidth}
\includegraphics[width=1.2in]{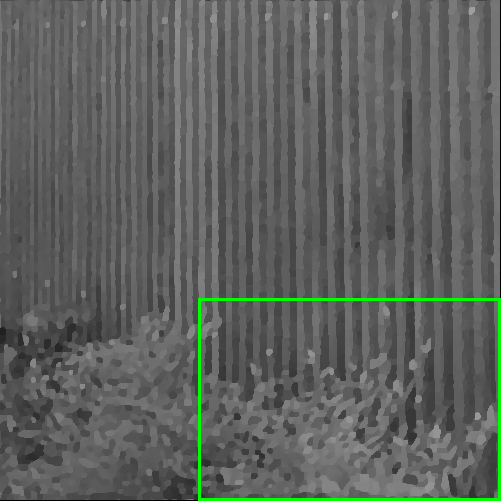}
\caption{TV-L$_2$.}
\label{fig:barbara_text1}
\end{subfigure}
\begin{subfigure}[t]{0.21\textwidth}
\centering
\includegraphics[width=1.2in]{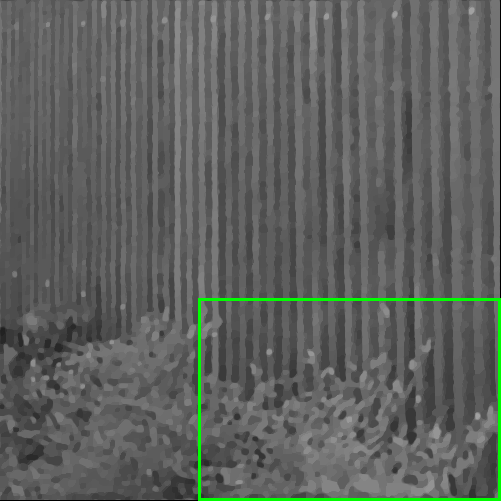}
\caption{TV$_p$-L$_2$.}
\label{fig:barbara_text2}
\end{subfigure}
\begin{subfigure}[t]{0.21\textwidth}
\centering
\includegraphics[width=1.2in]{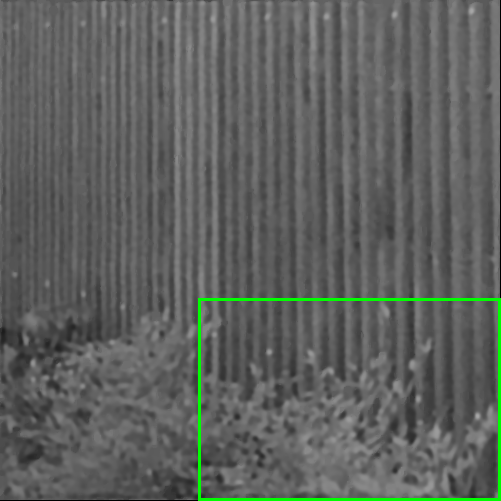} 
\caption{TV$_{\alpha,p}^{\mathrm{sv}}$-L$_2$.}
\label{fig:barbara_text3}
\end{subfigure}
\begin{subfigure}[t]{0.21\textwidth}
\centering
\includegraphics[width=1.2in]{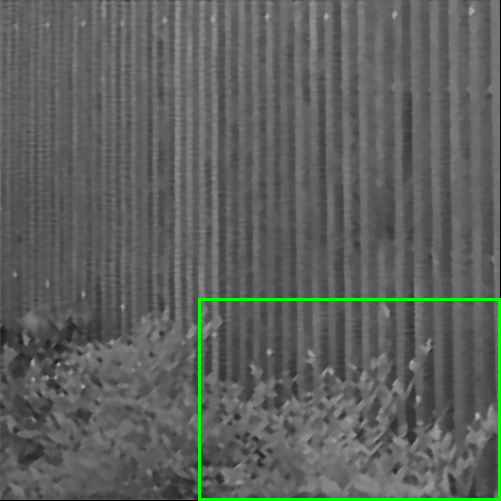}
\caption{DTV$_{p}^{\mathrm{sv}}$-L$_2$.}
\label{fig:barbara_text4}
\end{subfigure}  \\
\begin{subfigure}[t]{0.21\textwidth}
\includegraphics[width=1.2in]{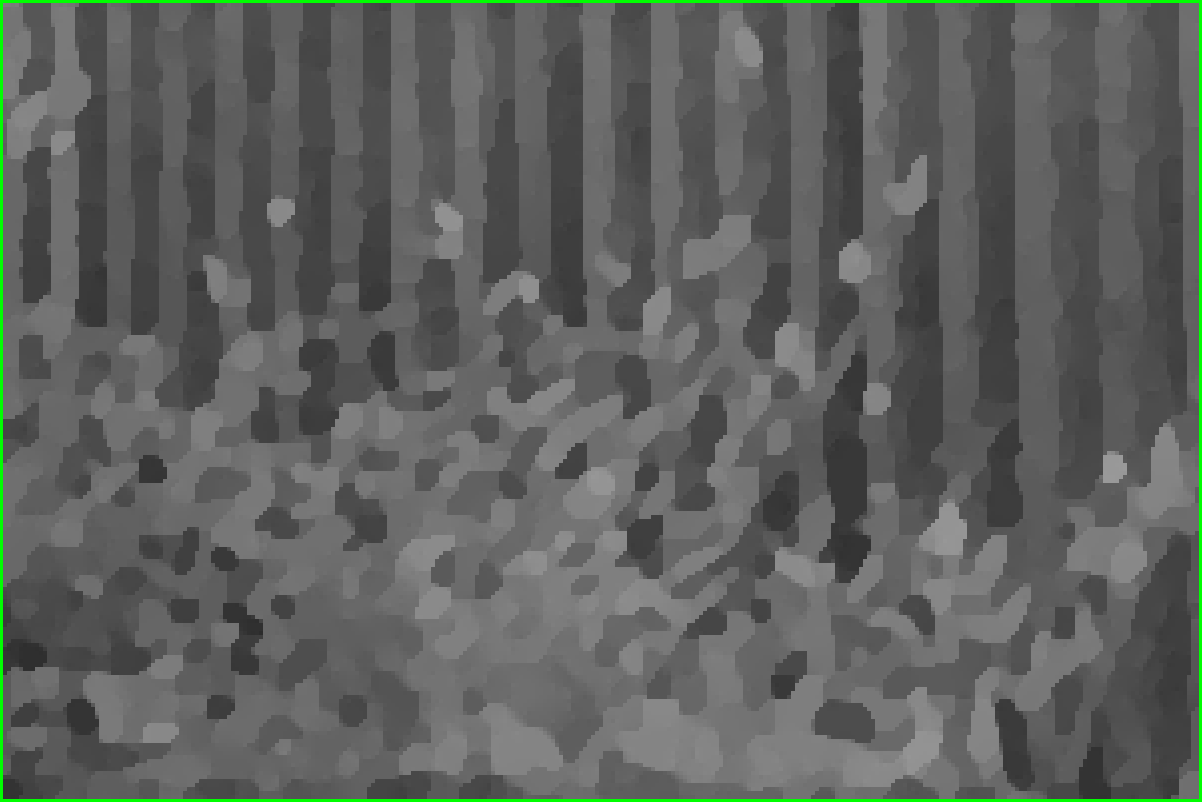} 
\caption{Zoom of \ref{fig:barbara_text1}.}
\end{subfigure}
\begin{subfigure}[t]{0.21\textwidth}
\centering
\includegraphics[width=1.2in]{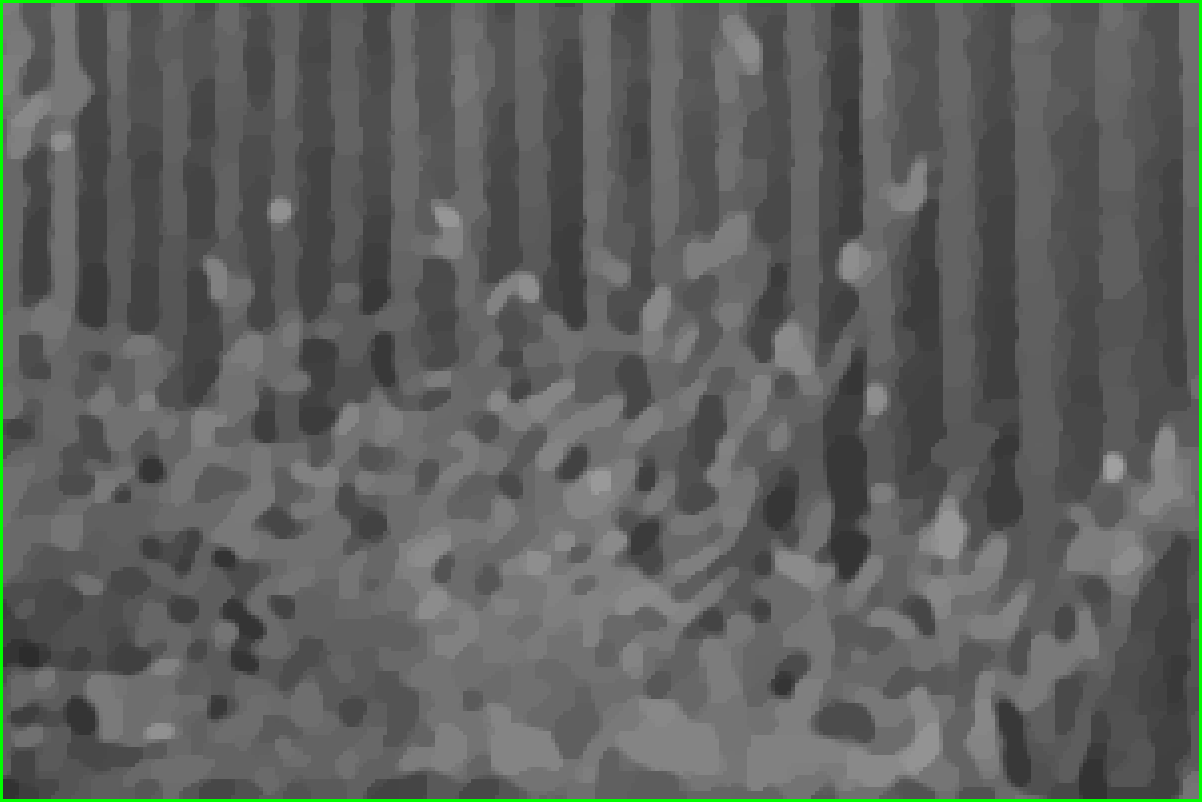}
\caption{Zoom of \ref{fig:barbara_text2}.}
\end{subfigure}
\begin{subfigure}[t]{0.21\textwidth}
\centering
\includegraphics[width=1.2in]{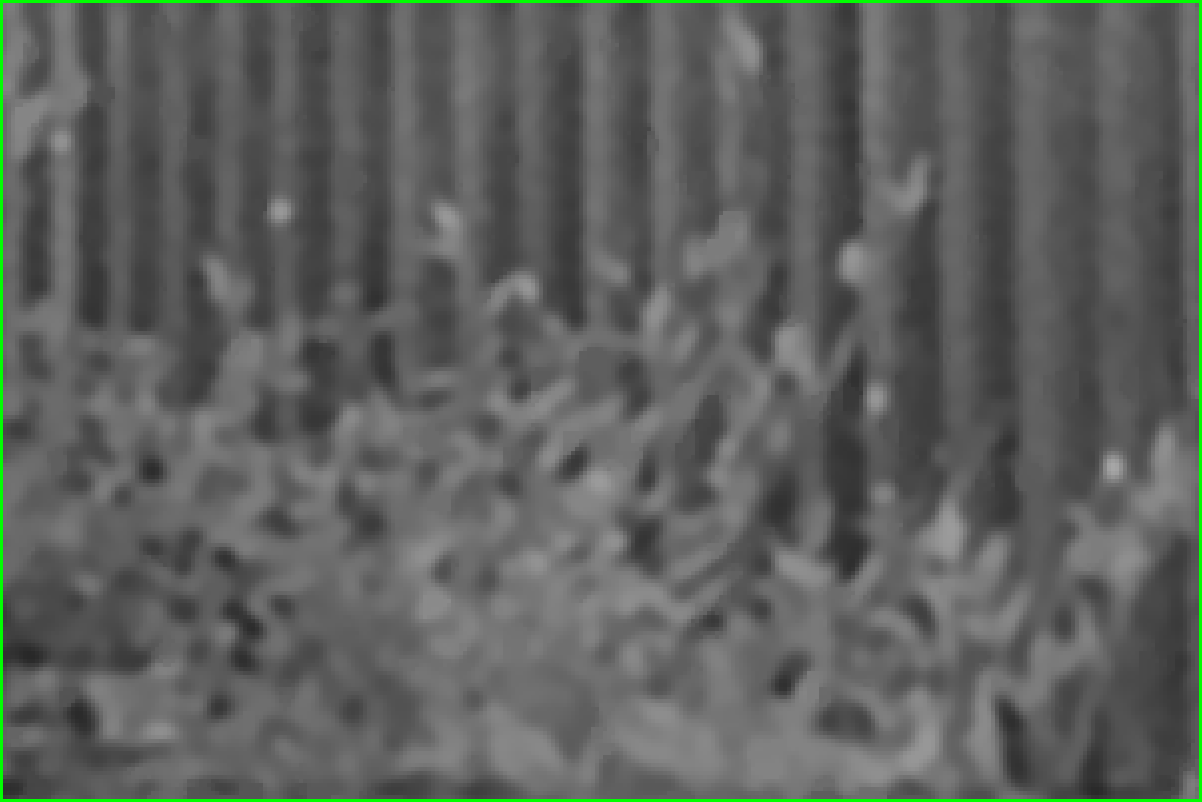}
\caption{Zoom of \ref{fig:barbara_text3}.}
\end{subfigure}
\begin{subfigure}[t]{0.21\textwidth}
\centering
\includegraphics[width=1.2in]{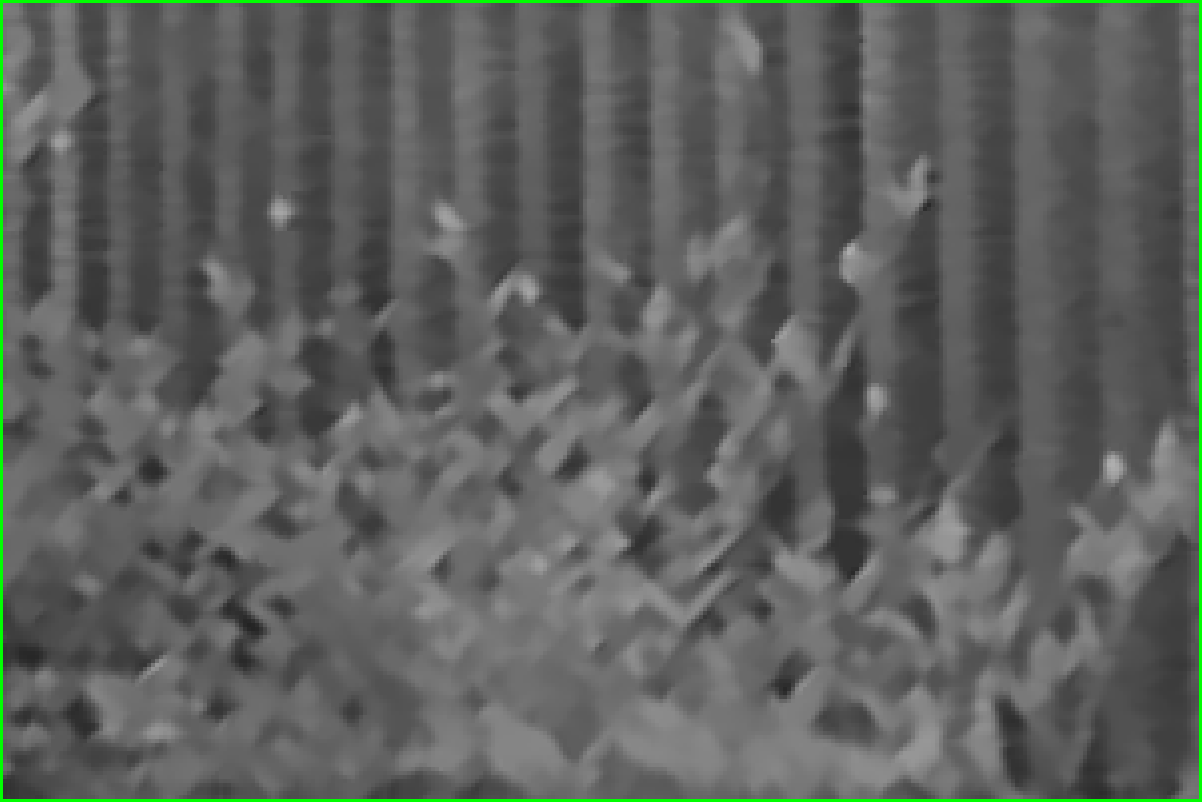}
\caption{Zoom of \ref{fig:barbara_text4}.}
\end{subfigure} 
\caption{Detail of reconstruction of natural test image \ref{fig:text}. Texture components are much better preserved by encoding directional information.}
	\label{fig:textrec}
\end{figure}

\section{Conclusions and outlook} \label{sec:conc}

We presented a new space-variant anisotropic image regularisation term for image restoration problems based on the statistical assumption that the gradients of the target image are distributed locally according to a BGGD. This leads to a highly flexible regulariser characterised by four per-pixel free parameters. For their automatic and effective selection, we propose a neighbourhood-based estimation procedure relying on the ML approach. 
We empirically show the good asymptotic properties of the estimator and its consistency with the geometric intuition about the behaviour of the BGGD in various image regions (edges, corners and homogeneous areas).
In terms of such parameters, we then study the corresponding space-variant and directional energy functional and apply it to  the problem of image restoration in case of additive white Gaussian noise.  Numerically, the restored image is computed efficiently by means of an iterative algorithm based on ADMM. 
The proposed regulariser is shown to outperform other space-variant restoration models and it is shown to achieve high quality restoration results, even when dealing with high levels of blur and noise. The directional feature of the regularisation considered results, in particular, in a better preservation of texture and details.

Future research directions include, first, the design of numerical algorithms other than ADMM with proved convergence properties also in the non-convex case such as, e.g., some suitable adaptation of the generalized Krylov subspace approaches proposed in \cite{lanzaKry2015,lanzaKry2017}. 
Then, automatic selection from the observed image of the ``optimal'' neighbourood size for the preliminary parameter estimation step is a matter worthy to be investigated. 
Finally, it would be very interesting to couple the proposed regulariser with other data fidelity terms, so as to deal with noises other than additive Gaussian.

\bibliographystyle{siamplain}
\bibliography{references}

\begin{thebibliography}{10}

\bibitem{AmbrosioBV}
{\sc L.~Ambrosio, N.~Fusco, and D.~Pallara}, {\em {Functions of bounded
  variation and free discontinuity problems}}, Oxford University Press, USA,
  2000.

\bibitem{BayramDTV2012}
{\sc I.~Bayram and M.~E. Kamasak}, {\em Directional total variation}, IEEE
  Signal Processing Letters, 19 (2012), pp.~781--784,
  \url{https://doi.org/10.1109/LSP.2012.2220349}.

\bibitem{benningFidelities}
{\sc M.~Benning and M.~Burger}, {\em Error estimates for general fidelities},
  Electronic Transactions on Numerical Analysis, 38 (2011), pp.~44--68,
  \url{https://doi.org/10.1.1.385.2286}.

\bibitem{BlomgrenTVp}
{\sc P.~Blomgren, T.~F. Chan, P.~Mulet, and C.~K. Wong}, {\em Total variation
  image restoration: numerical methods and extensions}, in Proceedings of
  International Conference on Image Processing, vol.~3, Oct 1997, pp.~384--387
  vol.3, \url{https://doi.org/10.1109/ICIP.1997.632128}.

\bibitem{Bolte2018}
{\sc J.~Bolte, S.~Sabach, and M.~Teboulle}, {\em Nonconvex lagrangian-based
  optimization: Monitoring schemes and global convergence}, Mathematics of
  Operations Research, 43 (2018), pp.~1210--1232,
  \url{https://doi.org/10.1287/moor.2017.0900},
  \url{https://doi.org/10.1287/moor.2017.0900},
  \url{https://arxiv.org/abs/https://doi.org/10.1287/moor.2017.0900}.

\bibitem{MGGD}
{\sc Z.~Boukouvalas, S.~Said, L.~Bombrun, Y.~Berthoumieu, and T.~Adalı}, {\em
  A new riemannian averaged fixed-point algorithm for mggd parameter
  estimation}, IEEE Signal Processing Letters, 22 (2015), pp.~2314--2318,
  \url{https://doi.org/10.1109/LSP.2015.2478803}.

\bibitem{BOYD_ADMM}
{\sc S.~Boyd, N.~Parikh, E.~Chu, B.~Peleato, and J.~Eckstein}, {\em Distributed
  optimization and statistical learning via the alternating direction method of
  multipliers}, Found. Trends Mach. Learn., 3 (2011), pp.~1--122,
  \url{https://doi.org/10.1561/2200000016}.

\bibitem{TGV}
{\sc K.~Bredies, K.~Kunisch, and T.~Pock}, {\em Total generalized variation},
  SIAM Journal on Imaging Sciences, 3 (2010), pp.~492--526,
  \url{https://doi.org/10.1137/090769521}.

\bibitem{bilevellearning}
{\sc L.~Calatroni, C.~Chung, J.~C. De~Los~Reyes, C.-B. Sch\"onlieb, and
  T.~Valkonen}, {\em Bilevel approaches for learning of variational imaging
  models}, in RADON book Series on Computational and Applied Mathematics, vol.
  18, Berlin, Boston: De Gruyter, 2017.

\bibitem{Calatroni2017}
{\sc L.~Calatroni, J.~De~Los~Reyes, and C.~Sch\"onlieb}, {\em Infimal
  convolution of data discrepancies for mixed noise removal}, SIAM Journal on
  Imaging Sciences, 10 (2017), pp.~1196--1233,
  \url{https://doi.org/10.1137/16M1101684}.

\bibitem{chambollelions1997}
{\sc A.~Chambolle and P.-L. Lions}, {\em Image recovery via total variation
  minimization and related problems}, Numerische Mathematik, 76 (1997),
  pp.~167--188, \url{https://doi.org/10.1007/s002110050258}.

\bibitem{chambolle2016introduction}
{\sc A.~Chambolle and T.~Pock}, {\em An introduction to continuous optimization
  for imaging}, Acta Numerica, 25 (2016), pp.~161--319,
  \url{https://doi.org/10.1017/S096249291600009X}.

\bibitem{Chen2006TVp}
{\sc Y.~Chen, S.~Levine, and M.~Rao}, {\em Variable exponent, linear growth
  functionals in image restoration}, SIAM Journal on Applied Mathematics, 66
  (2006), pp.~1383--1406, \url{https://doi.org/10.1137/050624522}.

\bibitem{Ciak2015}
{\sc R.~Ciak}, {\em Coercive functions from a topological viewpoint and
  properties of minimizing sets of convex functions appearing in image
  restoration}, doctoralthesis, Technische Universit{\"a}t Kaiserslautern,
  2015, \url{http://nbn-resolving.de/urn:nbn:de:hbz:386-kluedo-41000}.

\bibitem{duvaltvL1}
{\sc V.~Duval, J.-F. Aujol, and Y.~Gousseau}, {\em The {TV-$L^1$} model: a
  geometric point of view}, Multiscale Modeling \& Simulation, 8 (2009),
  pp.~154--189, \url{https://doi.org/10.1137/090757083}.

\bibitem{Ehrhardt2016}
{\sc M.~Ehrhardt and M.~Betcke}, {\em Multicontrast mri reconstruction with
  structure-guided total variation}, SIAM Journal on Imaging Sciences, 9
  (2016), pp.~1084--1106, \url{https://doi.org/10.1137/15M1047325},
  \url{https://doi.org/10.1137/15M1047325},
  \url{https://arxiv.org/abs/https://doi.org/10.1137/15M1047325}.

\bibitem{engl2000regularization}
{\sc H.~Engl, M.~Hanke, and A.~Neubauer}, {\em Regularization of Inverse
  Problems}, Mathematics and Its Applications, Springer Netherlands, 2000.

\bibitem{Hare2009}
{\sc W.~Hare and C.~Sagastiz{\'a}bal}, {\em Computing proximal points of
  nonconvex functions}, Mathematical Programming, 116 (2009), pp.~221--258,
  \url{https://doi.org/10.1007/s10107-007-0124-6},
  \url{https://doi.org/10.1007/s10107-007-0124-6}.

\bibitem{HY}
{\sc B.~He and X.~Yuan}, {\em On the $o(1/n)$ convergence rate of the
  {D}ouglas-{R}achford {A}lternating {D}irection {M}ethod}, SIAM Journal on
  Numerical Analysis, 50 (2012), pp.~700--709,
  \url{https://doi.org/10.1137/110836936}.

\bibitem{APE}
{\sc C.~He, C.~Hu, W.~Zhang, and B.~Shi}, {\em A fast adaptive parameter
  estimation for total variation image restoration}, IEEE Transactions on Image
  Processing, 23 (2014), pp.~4954--4967,
  \url{https://doi.org/10.1109/TIP.2014.2360133}.

\bibitem{Hintermuller2015}
{\sc M.~Hinterm\"{u}ller, T.~Valkonen, and T.~Wu}, {\em Limiting aspects of
  nonconvex ${TV}^\phi$ models}, SIAM Journal on Imaging Sciences, 8 (2015),
  pp.~2581--2621, \url{https://doi.org/10.1137/141001457}.

\bibitem{Hintermuller2013}
{\sc M.~Hinterm\"{u}ller and T.~Wu}, {\em Nonconvex ${TV}^q$-models in image
  restoration: Analysis and a trust-region regularization--based superlinearly
  convergent solver}, SIAM Journal on Imaging Sciences, 6 (2013),
  pp.~1385--1415, \url{https://doi.org/10.1137/110854746}.

\bibitem{HLR}
{\sc M.~Hong, Z.~Luo, and M.~Razaviyayn}, {\em Convergence analysis of
  alternating direction method of multipliers for a family of nonconvex
  problems}, SIAM Journal on Optimization, 26 (2016), pp.~337--364,
  \url{https://doi.org/10.1137/140990309}.

\bibitem{lanzaKry2017}
{\sc G.~Huang, A.~Lanza, S.~Morigi, L.~Reichel, and F.~Sgallari}, {\em
  Majorization–-minimization generalized {K}rylov subspace methods for
  $\ell_p$-$\ell_q$ optimization applied to image restoration}, BIT Numerical
  Mathematics, 57 (2017), pp.~351--378,
  \url{https://doi.org/10.1007/s10543-016-0643-8}.

\bibitem{Huang1999}
{\sc J.~Huang and D.~Mumford}, {\em Statistics of natural images and models},
  in Proceedings. 1999 IEEE Computer Society Conference on Computer Vision and
  Pattern Recognition, vol.~1, June 1999, pp.~541--547 Vol. 1,
  \url{https://doi.org/10.1109/CVPR.1999.786990}.

\bibitem{KonDonKnu17}
{\sc R.~Kongskov and Y.~Dong}, {\em Directional total generalized variation
  regularization for impulse noise removal}, in Scale Space and Variational
  Methods in Computer Vision, F.~Lauze, Y.~Dong, and A.~B. Dahl, eds., Cham,
  2017, Springer International Publishing, pp.~221--231.

\bibitem{KonDonKnuDTGV17}
{\sc R.~Kongskov, Y.~Dong, and Knudsen}, {\em Directional total generalized
  variation regularization},  (2017).
\newblock arXiv preprint: \url{https://arxiv.org/abs/1701.02675}.

\bibitem{Langer2017}
{\sc A.~Langer}, {\em Automated parameter selection for total variation
  minimization in image restoration}, Journal of Mathematical Imaging and
  Vision, 57 (2017), pp.~239--268,
  \url{https://doi.org/10.1007/s10851-016-0676-2}.

\bibitem{CMBBE}
{\sc A.~Lanza, S.~Morigi, M.~Pragliola, and F.~Sgallari}, {\em Space-variant
  generalised gaussian regularisation for image restoration}, Computer Methods
  in Biomechanics and Biomedical Engineering{:} Imaging and Visualization, 13
  (2018).

\bibitem{vip}
{\sc A.~Lanza, S.~Morigi, M.~Pragliola, and F.~Sgallari}, {\em Space-variant
  {TV} regularization for image restoration}, in VipIMAGE 2017, J.~M.~R.
  Tavares and R.~Natal~Jorge, eds., Cham, 2018, Springer International
  Publishing, pp.~160--169.

\bibitem{lanzaKry2015}
{\sc A.~Lanza, S.~Morigi, L.~Reichel, and F.~Sgallari}, {\em A generalized
  {K}rylov subspace method for $\ell_p$-$\ell_q$ minimization}, SIAM Journal on
  Scientific Computing, 37 (2015), pp.~S30--S50,
  \url{https://doi.org/10.1137/140967982}.

\bibitem{tvpl2}
{\sc A.~Lanza, S.~Morigi, and F.~Sgallari}, {\em Constrained {$TV_p$-$\ell^2$}
  model for image restoration}, Journal of Scientific Computing, 68 (2016),
  pp.~64--91, \url{https://doi.org/10.1007/s10915-015-0129-x},
  \url{https://doi.org/10.1007/s10915-015-0129-x}.

\bibitem{lanza2013}
{\sc A.~Lanza, S.~Morigi, F.~Sgallari, and Y.-W. Wen}, {\em Image restoration
  with {P}oisson-{G}aussian mixed noise}, Computer Methods in Biomechanics and
  Biomedical Engineering: Imaging \& Visualization, 2 (2014), pp.~12--24,
  \url{https://doi.org/10.1080/21681163.2013.811039}.

\bibitem{Li2010}
{\sc F.~Li, Z.~Li, and L.~Pi}, {\em Variable exponent functionals in image
  restoration}, Applied Mathematics and Computation, 216 (2010), pp.~870 --
  882, \url{https://doi.org/https://doi.org/10.1016/j.amc.2010.01.094}.

\bibitem{Nikolova2004}
{\sc M.~Nikolova}, {\em A variational approach to remove outliers and impulse
  noise}, Journal of Mathematical Imaging and Vision, 20 (2004), pp.~99--120,
  \url{https://doi.org/10.1023/B:JMIV.0000011326.88682.e5}.

\bibitem{NikolovaNonCvx2010}
{\sc M.~Nikolova, M.~K. Ng, and C.~P. Tam}, {\em Fast nonconvex nonsmooth
  minimization methods for image restoration and reconstruction}, IEEE
  Transactions on Image Processing, 19 (2010), pp.~3073--3088,
  \url{https://doi.org/10.1109/TIP.2010.2052275}.

\bibitem{Pascal2013}
{\sc F.~Pascal, L.~Bombrun, J.~Tourneret, and Y.~Berthoumieu}, {\em Parameter
  estimation for multivariate generalized gaussian distributions}, IEEE
  Transactions on Signal Processing, 61 (2013), pp.~5960--5971,
  \url{https://doi.org/10.1109/TSP.2013.2282909}.

\bibitem{Peter2015}
{\sc P.~Peter, J.~Weickert, A.~Munk, T.~Krivobokova, and H.~Li}, {\em
  Justifying tensor-driven diffusion from structure-adaptive statistics of
  natural images}, in Energy Minimization Methods in Computer Vision and
  Pattern Recognition, X.-C. Tai, E.~Bae, T.~F. Chan, and M.~Lysaker, eds.,
  Cham, 2015, Springer International Publishing, pp.~263--277.

\bibitem{Rodriguez2010}
{\sc P.~Rodriguez}, {\em Multiplicative updates algorithm to minimize the
  generalized total variation functional with a non-negativity constraint},
  2010 IEEE International Conference on Image Processing,  (2010),
  pp.~2509--2512.

\bibitem{Roth2009}
{\sc S.~Roth and M.~J. Black}, {\em Fields of experts}, International Journal
  of Computer Vision, 82 (2009), p.~205,
  \url{https://doi.org/10.1007/s11263-008-0197-6},
  \url{https://doi.org/10.1007/s11263-008-0197-6}.

\bibitem{Roussos2010}
{\sc A.~Roussos and P.~Maragos}, {\em Tensor-based image diffusions derived
  from generalizations of the total variation and beltrami functionals}, in
  2010 IEEE International Conference on Image Processing, Sep. 2010,
  pp.~4141--4144, \url{https://doi.org/10.1109/ICIP.2010.5653241}.

\bibitem{ROF}
{\sc L.~I. Rudin, S.~Osher, and E.~Fatemi}, {\em Nonlinear total variation
  based noise removal algorithms}, Physica D: Nonlinear Phenomena, 60 (1992),
  pp.~259 -- 268,
  \url{https://doi.org/https://doi.org/10.1016/0167-2789(92)90242-F}.

\bibitem{Scharr2003}
{\sc H.~Scharr, M.~J. Black, and H.~W. Haussecker}, {\em Image statistics and
  anisotropic diffusion}, in Proceedings Ninth IEEE International Conference on
  Computer Vision, Oct 2003, pp.~840--847 vol.2,
  \url{https://doi.org/10.1109/ICCV.2003.1238435}.

\bibitem{Sciacchitano2015}
{\sc F.~Sciacchitano, Y.~Dong, and T.~Zeng}, {\em Variational approach for
  restoring blurred images with cauchy noise}, SIAM Journal on Imaging
  Sciences, 8 (2015), pp.~1894--1922, \url{https://doi.org/10.1137/140997816}.

\bibitem{shape2}
{\sc K.~Sharifi and A.~Leon-Garcia}, {\em Estimation of shape parameter for
  generalized gaussian distributions in subband decompositions of video}, IEEE
  Transactions on Circuits and Systems for Video Technology, 5 (1995),
  pp.~52--56.

\bibitem{shape1}
{\sc K.-S. Song}, {\em A globally convergent and consistent method for
  estimating the shape parameter of a generalized gaussian distribution}, IEEE
  Transactions on Information Theory, 52 (2006), pp.~510--527,
  \url{https://doi.org/10.1109/TIT.2005.860423}.

\bibitem{stuartInversebook}
{\sc A.~M. Stuart}, {\em Inverse problems: a {B}ayesian perspective}, Acta
  Numerica, 19 (2010), pp.~451--559,
  \url{https://doi.org/10.1017/S0962492910000061}.

\bibitem{Tovey2019}
{\sc R.~Tovey, M.~Benning, C.~Brune, M.~J. Lagerwerf, S.~M. Collins, R.~K.
  Leary, P.~A. Midgley, and C.-B. Schönlieb}, {\em Directional sinogram
  inpainting for limited angle tomography}, Inverse Problems, 35 (2019),
  p.~024004, \url{https://doi.org/10.1088/1361-6420/aaf2fe},
  \url{https://doi.org/10.1088%2F1361-6420%2Faaf2fe}.

\bibitem{vese2001}
{\sc L.~Vese}, {\em A study in the {BV} space of a denoising--deblurring
  variational problem}, Applied Mathematics \& Optimization, 44 (2001),
  pp.~131--161, \url{https://doi.org/10.1007/s00245-001-0017-7}.

\bibitem{Wang2018}
{\sc Y.~Wang, W.~Yin, and J.~Zeng}, {\em Global convergence of {ADMM} in
  nonconvex nonsmooth optimization}, Journal of Scientific Computing, 78
  (2019), pp.~29--63, \url{https://doi.org/10.1007/s10915-018-0757-z},
  \url{https://doi.org/10.1007/s10915-018-0757-z}.

\bibitem{weickert98}
{\sc J.~Weickert}, {\em {Anisotropic Diffusion in Image Processing}}, B.G.
  Teubner, Stuttgart, 1998.

\bibitem{weickert02}
{\sc J.~Weickert and H.~Scharr}, {\em A scheme for coherence-enhancing
  diffusion filtering with optimized rotation invariance}, Journal of Visual
  Communication and Image Representation, 13 (2002), pp.~103 -- 118,
  \url{https://doi.org/10.1006/jvci.2001.0495}.

\bibitem{Bro02}
{\sc J.~Weickert and T.Brox}, {\em Diffusion and regularization of vector- and
  matrix-valued images}, in Inverse Problems, Image Analysis, and Medical
  Imaging, AMS, Dec 2002, pp.~251--268,
  \url{http://lmb.informatik.uni-freiburg.de/Publications/2002/Bro02a}.

\bibitem{WC12}
{\sc Y.~W. Wen and R.~H. Chan}, {\em Parameter selection for
  total-variation-based image restoration using discrepancy principle}, IEEE
  Transactions on Image Processing, 21 (2012), pp.~1770--1781,
  \url{https://doi.org/10.1109/TIP.2011.2181401}.

\bibitem{Zhang2013}
{\sc H.~Zhang and Y.~Wang}, {\em Edge adaptive directional total variation},
  The Journal of Engineering, 2013 (2013), pp.~61--62,
  \url{https://doi.org/10.1049/joe.2013.0116}.

\bibitem{ssim}
{\sc W.~Zhou, A.~Bovik, H.~Sheikh, and E.~Simoncelli}, {\em Image qualifty
  assessment: From error visibility to structural similarity.}, IEEE
  Transactions on Image Processing, 13 (2004).

\bibitem{Zhu1998}
{\sc S.~C. Zhu, Y.~Wu, and D.~Mumford}, {\em Filters, random fields and maximum
  entropy ({FRAME}): Towards a unified theory for texture modeling},
  International Journal of Computer Vision, 27 (1998), pp.~107--126,
  \url{https://doi.org/10.1023/A:1007925832420},
  \url{https://doi.org/10.1023/A:1007925832420}.

\end{thebibliography}

\end{document}